\documentclass[12pt]{article}
\usepackage[latin1]{inputenc}
\usepackage{amssymb}
\usepackage{hyperref}
\usepackage{amsmath}
\usepackage{latexsym}
\usepackage{times}
\usepackage{cite}
\usepackage{amssymb}
\usepackage{amsfonts}
\usepackage{amscd}
\usepackage{xcolor}
\usepackage{amstext,amsmath,amssymb,amsfonts}
\newtheorem{theorem}{Theorem}

\newtheorem{corollary}[theorem]{Corollary}

\newtheorem{proof}[theorem]{Proof}
\newtheorem{proposition}[theorem]{Proposition}
\newtheorem{remark}[theorem]{Remark}

\newcommand{\beq}{\begin{eqnarray}}
\newcommand{\eeq}{\end{eqnarray}}
\newcommand{\beqs}{\begin{eqnarray*}}
	\newcommand{\eeqs}{\end{eqnarray*}}
\newcommand{\bpro}{\begin{pro}}
	\newcommand{\epro}{\end{pro}}
\newcommand{\blem}{\begin{lem}}
	\newcommand{\elem}{\end{lem}}
\newcommand{\bdfn}{\begin{dfn}}
	\newcommand{\edfn}{\end{dfn}}
\newcommand{\bcor}{\begin{cor}}
	\newcommand{\ecor}{\end{cor}}
\newcommand{\bthm}{\begin{thm}}
	\newcommand{\ethm}{\end{thm}}
\newcommand{\bex}{\begin{ex}}
	\newcommand{\eex}{\end{ex}}
\newcommand{\brmk}{\begin{rmk}}
	\newcommand{\ermk}{\end{rmk}}
\newcommand{\bpr}{\begin{pr}}
	\newcommand{\epr}{\end{pr}}
\newcommand{\benum}{\begin{enumerate}}
	\newcommand{\eenum}{\end{enumerate}}
\newcommand{\bitem}{\begin{itemize}}
	\newcommand{\eitem}{\end{itemize}}

\newcommand{\cqfd}{\hfill{\square}}
\chardef\bslash=`\\
\numberwithin{equation}{section}
\numberwithin{table}{section}
\numberwithin{theorem}{section}

\begin{document}
\begin{center}
	{\Large { $\mathcal{R}(p,q)$-multivariate discrete  probability distributions}}\\
	\vspace{0,5cm}
	Fridolin Melong  \\
	\vspace{0.5cm}
	{\em Institut f\"ur Mathematik, Universit\"at Z\"urich,\\Winterthurerstrasse 190, CH-8057, Z\"urich, Switzerland}
	\\
	fridomelong@gmail.com
\end{center}
\today

\begin{abstract}
	We construct the multivariate probability distributions (P\'olya, inverse P\'olya, hypergeometric and negative hypergeometric) from the generalized quantum algebra. Moreover, we derive the  bivariate probability distributions and determine  their properties( $\mathcal{R}(p,q)$- factorial moments and covariance). Besides, we deduce  particular cases of probability distributions from the quantum algebras  known in the literature.
\end{abstract}
{\noindent
	{\bf Keywords.}
	$\mathcal{R}(p,q)-$ calculus, quantum algebra, multinomial coefficient, multivariate Vandermonde formula, multivariate P\'olya distribution,  multivariate hypergeometric distribution, bivariate distribution, $\mathcal{R}(p,q)$- factorial moments, $\mathcal{R}(p,q)$-covariance.\\
	MSC (2020)17B37, 81R50, 60E05, 05A30.	
}
\tableofcontents

\section{Introduction}
Charalambos presented the $q-$ deformed Vandermonde and Cauchy formulae. Moreover, the $q-$ deformed univariate discrete probability distributions were investigated.Their properties and limiting distributions were derived \cite{CA1}.

Furthermore, the $q-$ deformed multinomial coefficients was defined and their recurrence relations were deduced. Also, the $q-$ deformed multinomial  and negative $q-$ deformed multinomial probability distributions of the first and second kind  were presented \cite{CA2}. 

The same author extended the multivariate $q-$ deformed vandermonde and Cauchy formulae. Also, the multivariate $q-$ Pol\'ya and inverse  $q-$ Pol\'ya were constructed \cite{CA3}.

Let $p$ and $q$ be two positive real numbers such that $ 0<q<p<1.$ We consider a meromorphic function ${\mathcal R}$ defined on $\mathbb{C}\times\mathbb{C}$ by\cite{HB}:\begin{equation}\label{r10}
\mathcal{R}(u,v)= \sum_{s,t=-l}^{\infty}r_{st}u^sv^t,
\end{equation}
with an eventual isolated singularity at the zero, 
where $r_{st}$ are complex numbers, $l\in\mathbb{N}\cup\left\lbrace 0\right\rbrace,$ $\mathcal{R}(p^n,q^n)>0,  \forall n\in\mathbb{N},$ and $\mathcal{R}(1,1)=0$ by definition. We denote by $\mathbb{D}_{R}$ the bidisk \begin{eqnarray*}
	\mathbb{D}_{R}
	&=&\left\lbrace w=(w_1,w_2)\in\mathbb{C}^2: |w_j|<R_{j} \right\rbrace,
\end{eqnarray*}
where $R$ is the convergence radius of the series (\ref{r10}) defined by Hadamard formula as follows:
\begin{eqnarray*}
	\lim\sup_{s+t \longrightarrow \infty} \sqrt[s+t]{|r_{st}|R^s_1\,R^t_2}=1.
\end{eqnarray*}
For the proof and more details see \cite{TN}. 
We denote by 
${\mathcal O}(\mathbb{D}_R)$ the set of holomorphic functions defined
on $\mathbb{D}_R.$

The  $\mathcal{R}(p,q)-$ deformed numbers is defined by \cite{HB}: 
\begin{eqnarray}\label{Rpqn}
[x]_{\mathcal{R}(p,q)}:= \mathcal{R}(p^x,q^x),\qquad x\in\mathbb{N},
\end{eqnarray}
the $\mathcal{R}(p,q)-$ deformed factorials and binomial coefficients are given  as:
\begin{eqnarray}\label{Rpqf}
[x]!_{\mathcal{R}(p,q)}:= \left\{\begin{array}{lr} 1 \quad \mbox{for} \quad n=0 \quad \\
\mathcal{R}(p,q)\cdots \mathcal{R}(p^x,q^x) \quad \mbox{for} \quad x\geq 1, \quad \end{array} \right.
\end{eqnarray}
and
\begin{eqnarray}\label{Rpqbc}
\left[\begin{array}{c} x  \\ y\end{array} \right]_{\mathcal{R}(p,q)}:=
\frac{[x]!_{\mathcal{R}(p,q)}}{[y]!_{\mathcal{R}(p,q)}[x-y]!_{\mathcal{R}(p,q)}},\quad x, y= 0, 1, 2, \cdots;\quad x\geq y.
\end{eqnarray}

We consider the following linear operators   on ${\mathcal O}(\mathbb{D}_R)$  given by:
\begin{eqnarray}\label{operat}
&&\quad Q: \varphi \longmapsto Q\varphi(z) := \varphi(qz)
\nonumber \\
&&\quad P: \varphi \longmapsto P\varphi(z): = \varphi(pz)
,
\end{eqnarray}
leading to define  the $\mathcal{R}(p,q)-$ deformed derivative:
\begin{equation}{\label{deriva1}}
\partial_{{\mathcal R},p,q} := \partial_{p,q}\frac{p - q}{P-Q}{\mathcal R}(P, Q)
= \frac{p - q}{pP-qQ}{\mathcal R}(pP, qQ)\partial_{p,q},
\end{equation}
where $\partial_{p,q}$ is the $(p,q)-$ derivative:
\begin{eqnarray}
\partial_{p,q}:\varphi \longmapsto
\partial_{p,q}\varphi(z) := \frac{\varphi(pz) - \varphi(qz)}{z(p-q)}.
\end{eqnarray}
The quantum algebra associated with the $\mathcal{R}(p,q)-$ deformation, denoted by 
${\mathcal A}_{\mathcal{R}(p,q)}$ is generated by the
set of operators $\{1, A, A^{\dagger}, N\}$ satisfying the following
commutation relations \cite{HB1}:
\begin{eqnarray}
&& \label{algN1}
\quad A A^\dag= [N+1]_{\mathcal{R}(p,q)},\quad\quad\quad A^\dag  A = [N]_{\mathcal{R}(p,q)}.
\cr&&\left[N,\; A\right] = - A, \qquad\qquad\quad \left[N,\;A^\dag\right] = A^\dag
\end{eqnarray}
with its realization on ${\mathcal O}(\mathbb{D}_R)$ given by:
\begin{eqnarray*}\label{algNa}
	A^{\dagger} := z,\qquad A:=\partial_{\mathcal{R}(p,q)}, \qquad N:= z\partial_z,
\end{eqnarray*}
where $\partial_z:=\frac{\partial}{\partial z}$ is the usual derivative on $\mathbb{C}.$

The $\mathcal{R}(p,q)-$ deformed numbers \eqref{Rpqn} can be rewritten as follows \cite{HMD}:
\begin{eqnarray}
[x]=\frac{\tau^{x}_1-\tau^{x}_2}{ \tau_1-\tau_2}, \quad \tau_1\neq \tau_2,
\end{eqnarray}
where $\big(\tau_i\big)_{i\in\{1,2\}}$ are functions depending of the parameters deformations $p$ and $q.$

The following relations hold \cite{HMRC}:
\begin{eqnarray}\label{011}
[x]_{\mathcal{R}(p^{-1},q^{-1})} &=& (\tau_1\,\tau_2)^{1-x}\,[x]_{\mathcal{R}(p,q)},\\
\,[r]_
{\mathcal{R}(p^{-1},q^{-1})}!&=& (\tau_1\,\tau_2)^{- {r  \choose 2}}\,\,[r]_{\mathcal{R}(p,q)}!,\label{014}\\
\,{[x]_{r,\mathcal{R}(p^{-1},q^{-1})}}
&=& (\tau_1\,\tau_2)^{-xr + {r +1 \choose 2}}\,{[x]_{r,\mathcal{R}(p,q)}}\label{015}.
\end{eqnarray}
Furthermore, the generalized Vandermonde, Cauchy formulae and univariate probability distributions induced from the $\mathcal{R}(p,q)-$ deformed quantum algebras are investigated in \cite{HMD}. 

Our aims is to construct the multinomial coeficient, multivariate Vandermonde, and Cauchy formulae, multivariate probability distributions, and properties associated to the  $\mathcal{R}(p,q)-$ deformed quantum algebras \cite{HB1}.  

This paper is organized as follows: In section $2,$ we investigate the $\mathcal{R}(p,q)-$ multinomial coefficient, the  multivariate Vandermonde formula associated to the $\mathcal{R}(p,q)-$ deformed quantum algebras. Inverse multivariate $\mathcal{R}(p,q)-$ Vandermonde and negative multivariate $\mathcal{R}(p,q)-$ Vandermonde formula are computed. Moreover,  the $\mathcal{R}(p,q)-$ deformed Cauchy formula is deduced. Section $3,$ is reserved to the construction of the multivariate probability distributions. We derive the case of bivariate distributions and compute their $\mathcal{R}(p,q)$-factorial moments and covariance. Relevant particular cases corresponding to the quantum algbras known in the literature are derived from the general formalism.
\section{$\mathcal{R}(p,q)$-  multivariate Vandermonde and  Cauchy formulae }
In this section, we investigate the generalized  multinomial coefficients,  multivariate Vandermonde formula, and  multivariate Cauchy formula  associated to the $\mathcal{R}(p,q)-$ deformed quantum algebras. Their recurrence relations are also derived. 
\begin{proposition}
	The generalized  multinomial coefficient
	\begin{eqnarray}\label{eq2.1}
	\left[\begin{array}{c}x \\r_1,r_2,\ldots,r_k\end{array} \right]_{\mathcal{R}(p,q)}=\frac{[x]_{r_1+r_2+\cdots+r_k,\mathcal{R}(p,q)}}{ [r_1]_{\mathcal{R}(p,q)}![r_2]_{\mathcal{R}(p,q)}!\cdots[r_k]_{\mathcal{R}(p,q)}!}
	\end{eqnarray} 
	satisfies the recursion relation:
	\begin{small}
		\begin{eqnarray*}
		&&\genfrac{[}{]}{0pt}{}{x}{r_1,r_2,\ldots,r_k}_{{\mathcal R}(p,q)}=\tau_1^{s_k}\genfrac{[}{]}{0pt}{}{x-1}{r_1,r_2,\ldots,r_k}_{{\mathcal R}(p,q)}
		+\tau_2^{x-m_1}\genfrac{[}{]}{0pt}{}{x-1}{r_1-1,r_2,\ldots,r_k}_{\mathcal {R}(p,q)}\cr&& \quad\quad\quad
		+\tau_2^{x-m_2}\genfrac{[}{]}{0pt}{}{x-1}{r_1,r_2-1,\ldots,r_k}_{{\mathcal R}(p,q)}+\cdots+\tau_2^{x-m_k}\genfrac{[}{]}{0pt}{}{x-1}{r_1,r_2,\ldots,r_k-1}_{{\mathcal R}(p,q)}.
		\end{eqnarray*}
		Equivalently,
		\begin{eqnarray*}
		&&\genfrac{[}{]}{0pt}{}{x}{r_1,r_2,\ldots,r_k}_{\mathcal {R}(p,q)}=\tau_2^{s_k}\genfrac{[}{]}{0pt}{}{x-1}{r_1,r_2,\ldots,r_k}_{{\mathcal R}(p,q)}
		+\tau^{x-m_1}_1\genfrac{[}{]}{0pt}{}{x-1}{r_1-1,r_2,\ldots,r_k}_{{\mathcal R}(p,q)}\cr&&\quad\quad
		+\tau^{x-m_2}_1\tau_2^{s_1}\genfrac{[}{]}{0pt}{}{x-1}{r_1,r_2-1,\ldots,r_k}_{\mathcal {R}(p,q)}
		+\cdots+\tau^{x-m_k}_1\tau_2^{s_{k-1}}\genfrac{[}{]}{0pt}{}{x-1}{r_1,r_2,\ldots,r_k-1}_{{\mathcal R}(p,q)},
		\end{eqnarray*}
		where $r_j\in\mathbb{N}$ and $j\in\{1,2,\ldots,k\}$, with $m_j=\sum_{i=j}^kr_i$ and $s_j=\sum_{i=1}^jr_i.$
	\end{small}
\end{proposition}
\begin{proof} 
	It is straightforwad by computation. $\cqfd$
\end{proof}

From the relations (\ref{011}), (\ref{014}) and (\ref{015}), we obtain another expression of the genralized multinomial coefficient in the simpler form:
\begin{small}
	\begin{equation}\label{eq2.2}
	\genfrac{[}{]}{0pt}{}{x}{r_1,r_2,\cdots,r_k}_{\mathcal{R}(p^{-1},q^{-1})}
	=(\tau_1\tau_2)^{-\displaystyle\sum_{j=1}^k r_j(x-m_j)}\genfrac{[}{]}{0pt}{}{x}{r_1,r_2,\cdots,r_k}_{\mathcal{R}(p,q)}\end{equation}and
	\begin{equation}\label{eq2.2a}
	\genfrac{[}{]}{0pt}{}{x}{r_1,r_2,\cdots,r_k}_{\mathcal{R}(p^{-1},q^{-1})}=(\tau_1\tau_2)^{-\displaystyle\sum_{j=1}^k r_j(x-s_j)}\genfrac{[}{]}{0pt}{}{x}{r_1,r_2,\cdots,r_k}_{\mathcal{R}(p,q)},
	\end{equation}
	where $s_j=\displaystyle\sum_{i=1}^jr_i,$  $m_j=\displaystyle\sum_{i=j}^kr_i,$ $r_j\in\mathbb{N},$ $j\in\{1,2,\cdots,k\}$ and $k\in\mathbb{N}.$ 
\end{small}

Another  recurrence relations can be obtained by  using the expression  (\ref{eq2.2}), respectively. Thus, we get:
\begin{small}
	\begin{eqnarray*} \label{eq2.5}
		\genfrac{[}{]}{0pt}{}{x}{r_1,r_2,\cdots,r_k}_{\mathcal {R}(p,q)}&=&\tau_2^{m_1}\genfrac{[}{]}{0pt}{}{x-1}{r_1,r_2,\cdots,r_k}_{\mathcal {R}(p,q)}
		+\tau_2^{m_2}\genfrac{[}{]}{0pt}{}{x-1}{r_1-1,r_2,\cdots,r_k}_{{\mathcal R}(p,q)}\nonumber\\
		&+&\tau_2^{m_3}\genfrac{[}{]}{0pt}{}{x-1}{r_1,r_2-1,\cdots,r_k}_{\mathcal{R}(p,q)}
		+\cdots+\tau^x_1\genfrac{[}{]}{0pt}{}{x-1}{r_1,r_2,\cdots,r_k-1}_{\mathcal{R}(p,q)}
	\end{eqnarray*}
	and
	\begin{eqnarray}\label{eq2.5a}
		\genfrac{[}{]}{0pt}{}{x}{r_1,r_2,\cdots,r_k}_{{\mathcal R}(p,q)}&=&\tau^x_1\genfrac{[}{]}{0pt}{}{x-1}{r_1,r_2,\cdots,r_k}_{{\mathcal R}(p,q)}
		+\tau_2^{x-s_1}\genfrac{[}{]}{0pt}{}{x-1}{r_1-1,r_2,\cdots,r_k}_{{\mathcal R}(p,q)}\nonumber\\
		&+&\tau_2^{x-s_2}\genfrac{[}{]}{0pt}{}{x-1}{r_1,r_2-1,\ldots,r_k}_{{\mathcal R}(p,q)}
		+\cdots+\tau_2^{x-s_k}\genfrac{[}{]}{0pt}{}{x-1}{r_1,r_2,\ldots,r_k-1}_{{\mathcal R}(p,q)}.
	\end{eqnarray}
\end{small}
\begin{theorem}\label{thm2.1}
	The generalized multivariate  Vandermonde formula is given by the following relations:
	\begin{small}
		\begin{equation}\label{eq2.7}
		\Big[\sum_{i=1}^{k+1}x_i\Big]_{n,\mathcal {R}(p,q)}=\sum_{r_j=0}^{n}\genfrac{[}{]}{0pt}{}{n}{r_1,r_2,\cdots,r_k}_{\mathcal{R}(p,q)}
		\mathcal{V}(\tau_1,\tau_2,n)\prod_{j=1}^{k+1}[x_j]_{r_j,{\mathcal{R}(p,q)}}
		\end{equation}
		and
		\begin{equation}\label{eq2.8}
		\Big[\sum_{i=1}^{k+1}x_i\Big]_{n,\mathcal {R}(p,q)}=\sum_{r_j=0}^{n}\genfrac{[}{]}{0pt}{}{n}{r_1,r_2,\cdots,r_k}_{\mathcal{R}(p,q)}
		\mathcal{V}(\tau_2,\tau_1,n)\prod_{j=1}^{k+1}[x_j]_{r_j,\mathcal{R}(p,q)},
		\end{equation}
	\end{small}
	where $\mathcal{V}(\tau_1,\tau_2)=\tau_1^{\sum_{j=1}^{k}r_j(z_j-(n-s_j))}\tau_2^{\sum_{j=1}^{k}(n-s_j)(x_j-r_j)},$ $s_j=\sum_{i=1}^jr_i,$ $z_j=\sum_{i=j+1}^{k+1}x_i,$ $j\in\{1,2,\cdots,k\},$ $r_{k+1}=n-s_k$, and  $\sum_{i=1}^kr_i\leq n.$
\end{theorem}
\begin{proof}
We use the same procedure as Theorem in \cite{HMD}.
$\cqfd$
\end{proof}
\begin{proposition}
	The negative generalized multivariate  Vandermonde formula is described  by the following relations:
	\begin{small}
		\begin{equation}
		\Big[\sum_{i=1}^{k+1}x_i\Big]_{-n,\mathcal {R}(p,q)}=\sum_{r_j=0}^{n}\genfrac{[}{]}{0pt}{}{-n}{r_1,\cdots,r_k}_{\mathcal{R}(p,q)}
		\mathcal{V}(\tau_1,\tau_2,-n)\prod_{j=1}^{k+1}[x_j]_{r_j,{\mathcal{R}(p,q)}}
		\end{equation}
		and
		\begin{equation}
		\Big[\sum_{i=1}^{k+1}x_i\Big]_{-n,\mathcal {R}(p,q)}=\sum_{r_j=0}^{n}\genfrac{[}{]}{0pt}{}{-n}{r_1,\cdots,r_k}_{\mathcal{R}(p,q)}
		\mathcal{V}(\tau_2,\tau_1,-n)\prod_{j=1}^{k+1}[x_j]_{r_j,\mathcal{R}(p,q)},
		\end{equation}
	\end{small}
	where $s_j=\sum_{i=1}^jr_i$, $z_j=\sum_{i=j+1}^{k+1}x_i$, $j\in\{1,2,\cdots,k\},$ $r_{k+1}=n-s_k,$ and $\sum_{i=1}^kr_i\leq n.$
\end{proposition}
\begin{proof}
	We use the same procedure as the proof of Theorem \eqref{thm2.1}. 
\end{proof}
\begin{remark}
By replacing $x_j$ by $-x_j$, for $j\in\{1,2,\ldots,k+1\},$ in the relations (\ref{eq2.7}) and (\ref{eq2.8}), and using the following expressions:
\begin{eqnarray}
[-x]_{r,\mathcal{R}(p,q)}=(-1)^x(\tau_1\tau_2)^{-xr-\binom{r}{2}}[x+r-1]_{r,\mathcal{R}(p,q)},
\end{eqnarray}
and
\begin{eqnarray}\label{eq2.id}
\binom{n}{2}=\sum_{j=1}^{k+1}\binom{r_j}{2}+\sum_{j=1}^kr_j(n-s_j),
\end{eqnarray}
the generalized multivariate   Vandermonde formula can be expressed as:
\begin{small}
	\begin{eqnarray}\label{eq2.7a}
	\Big[\sum_{i=1}^{k+1}x_i+n-1\Big]_{n,\mathcal {R}(p,q)}&=&\sum_{r_j=0}^{n}\genfrac{[}{]}{0pt}{}{n}{r_1,r_2,\cdots,r_k}_{\mathcal {R}(p,q)}\tau_1^{\sum_{j=1}^{k}x_j(n-s_j)}
	\tau_2^{\sum_{j=1}^{k}r_jz_j}\nonumber\\ &\times& \prod_{j=1}^{k+1}[x_j+r_j-1]_{r_j,\mathcal {R}(p,q)},
	\end{eqnarray}
	and
	\begin{eqnarray}\label{eq2.8a}
	\Big[\sum_{i=1}^{k+1}x_i+n-1\Big]_{n,\mathcal {R}(p,q)}&=&\sum_{r_j=0}^{n}\genfrac{[}{]}{0pt}{}{n}{r_1,r_2,\ldots,r_k}_{\mathcal {R}(p,q)}\tau_1^{\sum_{j=1}^{k}r_jz_j}
	\tau_2^{\sum_{j=1}^{k}x_j(n-s_j)}\nonumber\\&\times&\prod_{j=1}^{k+1}[x_j+r_j-1]_{r_j,\mathcal {R}(p,q)},
	\end{eqnarray}
\end{small}
where $s_j=\sum_{i=1}^jr_i$, $z_j=\sum_{i=j+1}^{k+1}x_i$, $j\in\{1,2,\ldots,k\}$, and $r_{k+1}=n-s_k$, and  $\sum_{i=1}^kr_i\leq n$.
\end{remark}
\begin{theorem}\label{thm2.2}
	The generalized multivariate inverse  Vandermonde formula is presented by the following relations as follows:
	\begin{small}
		\begin{eqnarray*}\label{eq2.11}
		\frac{1}{[x_{k+1}]_{n,{\mathcal R}(p,q)}}=\sum_{r_j=0}^{n}\genfrac{[}{]}{0pt}{}{n+s_k-1}{r_1,\ldots,r_k}_{{\mathcal R}(p,q)}\frac{\tau_1^{\sum_{j=1}^{k}r_j(z_j-s_k+s_j-n+1)}}{
		\tau_2^{\sum_{j=1}^{k}(-n-s_k+s_j)(x_j-r_j)}}\frac{\prod_{j=1}^k[x_j]_{r_j,{\mathcal R}(p,q)}}{[\sum_{i=1}^{k+1}x_i]_{n+s_k,{\mathcal R}(p,q)}},
		\end{eqnarray*}
		provided $|\big({\tau_2\over \tau_1}\big)^{-x_{k+1}}|<1$, and
		\begin{eqnarray}\label{eq2.12}
		\frac{1}{[x_{k+1}]_{n,\mathcal {R}(p,q)}}=\sum_{r_j=0}^{n}\genfrac{[}{]}{0pt}{}{n+s_k-1}{r_1,\cdots,r_k}_{\mathcal {R}(p,q)}\frac{\tau_1^{\sum_{j=1}^{k}(n+s_k-s_j)(x_j-r_j)}}{
		\tau_2^{\sum_{j=1}^{k}r_j(-z_j+s_k-s_j+n-1)}}\frac{\prod_{j=1}^k[x_j]_{r_j,\mathcal{R}(p,q)}}{[\sum_{i=1}^{k+1}x_i]_{n+s_k,\mathcal {R}(p,q)}},\quad
		\end{eqnarray}
	\end{small}
	provided $|\big({\tau_2\over \tau_1}\big)^{x_{k+1}}|<1,$ where $s_j=\sum_{i=1}^jr_i$ and $z_j=\sum_{i=j+1}^{k+1}x_i,$ for $j\in\{1,2,\cdots,k\}.$
\end{theorem}
\begin{proof}. From the  inverse ${\mathcal R}(p,q)$-Vandermonde formula \cite{HMRC}, we get:
\begin{small}
\begin{eqnarray}
\frac{1}{[x_{k+1}]_{n,{\mathcal R}(p,q)}}=\sum_{r_k=0}^\infty\genfrac{[}{]}{0pt}{}{n+r_k-1}{r_k}_{{\mathcal R}(p,q)}\tau^{r_k(x_{k+1}-n+1)}_1
\tau_2^{n(x_k-r_k)}\frac{[x_k]_{r_k,{\mathcal R}(p,q)}}{[x_k+x_{k+1}]_{n+r_k,{\mathcal R}(p,q)}}.
\end{eqnarray}
Similarly,
	\begin{eqnarray}
	\frac{1}{[x_k+x_{k+1}]_{n+r_k,{\mathcal R}(p,q)}}&=&\sum_{r_{k-1}=0}^\infty\genfrac{[}{]}{0pt}{}{n+r_k+r_{k-1}-1}{r_{k-1}}_{{\mathcal R}(p,q)}\tau^{r_{k-1}(x_k+x_{k+1}-n-r_k+1)}_1\nonumber\\&\times&
	\frac{\tau_2^{(n+r_k)(x_{k-1}-r_{k-1})}[x_{k-1}]_{r_{k-1},\mathcal {R}(p,q)}}{[x_{k-1}+x_k+x_{k+1}]_{n+r_k+r_{k-1},\mathcal {R}(p,q)}}
	\end{eqnarray}
and finally,
	\begin{eqnarray}
	\frac{1}{[x_2+x_3+\cdots+x_{k+1}]_{n+s_k-s_1,{\mathcal R}(p,q)}}&=&\sum_{r_1=0}^\infty\genfrac{[}{]}{0pt}{}{n+s_k-1}{r_1}_{{\mathcal R}(p,q)}\tau^{r_1(x_2+x_3+\ldots+x_{k+1}-n-s_k+s_1+1)}_1\nonumber\\&\times&
	\frac{\tau_2^{(n+s_k-s_1)(x_1-r_1)}[x_1]_{r_1,\mathcal {R}(p,q)}}{[x_1+x_2+\cdots+x_{k+1}]_{n+s_k,{\mathcal R}(p,q)}}.
	\end{eqnarray}
Applying these $k$ expansions, one after the other in the inner sum of each step, and using the relation:
\begin{small}
	\begin{eqnarray}
	\genfrac{[}{]}{0pt}{}{n+r_k-1}{r_k}_{\mathcal {R}(p,q)}\genfrac{[}{]}{0pt}{}{n+r_k+r_{k-1}-1}{r_{k-1}}_{\mathcal {R}(p,q)}\genfrac{[}{]}{0pt}{}{n+s_k-1}{r_1}_{\mathcal {R}(p,q)}
	=\genfrac{[}{]}{0pt}{}{n+s_k-1}{r_1,r_2,\ldots,r_k}_{\mathcal {R}(p,q)},
	\end{eqnarray}
	\end{small}
expansion (\ref{eq2.11}) is obtained. The alternative expansion (\ref{eq2.12}), is similarly deduced by using the following inverse ${\mathcal R}(p,q)$-Vandermonde expansions \cite{HMRC}:
	\begin{eqnarray}
	\frac{1}{[x_{j+1}+\cdots+x_{k+1}]_{n+s_k-s_j,{\mathcal R}(p,q)}}&=&\sum_{r_j=0}^\infty\genfrac{[}{]}{0pt}{}{n+s_k-s_{j-1}-1}{r_j}_{{\mathcal R}(p,q)}\nonumber\\&\times&
	\frac{\tau_2^{r_j(z_j-s_k+s_j-n+1)}[x_j]_{r_j,{\mathcal R}(p,q)}}{[x_j+\cdots+x_{k+1}]_{n+s_k-s_{j-1},{\mathcal R}(p,q)}},
	\end{eqnarray}
for $j\in\{1,2,\cdots,k\},$ with $s_0=0$.
\end{small}
$\cqfd$
\end{proof}

Now, we investigated the generalized multivariate   Cauchy formula and its related formulae. 
\begin{corollary}\label{cor2.1}
	The generalized multivariate Cauchy formula is furnished by:
	\begin{small}
		\begin{eqnarray}\label{eq2.9}
		\genfrac{[}{]}{0pt}{}{\sum_{i=1}^{k+1}x_i}{n}_{{\mathcal R}(p,q)}=\sum_{r_j=0}^{n} \tau_1^{\sum_{j=1}^{k}r_j(z_j-(n-s_j))}\tau_2^{\sum_{j=1}^{k}(n-s_j)(x_j-r_j)}\prod_{j=1}^{k+1}\genfrac{[}{]}{0pt}{}{x_j}{r_j}_{{\mathcal R}(p,q)}
		\end{eqnarray}
		and
		\begin{equation}\label{eq2.10}
		\genfrac{[}{]}{0pt}{}{\sum_{i=1}^{k+1}x_i}{n}_{{\mathcal R}(p,q)}=\sum_{r_j=0}^{n} \tau_1^{\sum_{j=1}^{k}(n-s_j)(x_j-r_j)} \tau_2^{\sum_{j=1}^{k}r_j(z_j-(n-s_j))}\,\prod_{j=1}^{k+1}\genfrac{[}{]}{0pt}{}{x_j}{r_j}_{{\mathcal R}(p,q)},
		\end{equation}
		where $s_j=\sum_{i=1}^jr_i,$ $z_j=\sum_{i=j+1}^{k+1}x_i,$ $j\in\{1,2,\cdots,k\},$ and $r_{k+1}=n-s_k$  with $\sum_{i=1}^kr_i\leq n.$
	\end{small}
\end{corollary}
\begin{remark}
Several formulae can be deduced as follows:
\begin{enumerate}
\item[(a)] The generalized multivariate  Cauchy formula can be re-written as:
\begin{small}
	\begin{eqnarray*}\label{eq2.9a}
		\genfrac{[}{]}{0pt}{}{\sum_{i=1}^{k+1}x_i+n-1}{n}_{\mathcal {R}(p,q)}=\sum_{r_j=0}^{n}\tau_1^{\sum_{j=1}^{k}x_j(n-s_j)} \tau_2^{\sum_{j=1}^{k}r_jz_j}\prod_{j=1}^{k+1}\genfrac{[}{]}{0pt}{}{x_j+r_j-1}{r_j}_{{\mathcal R}(p,q)}\end{eqnarray*}
	and
	\begin{eqnarray*}
		\genfrac{[}{]}{0pt}{}{\sum_{i=1}^{k+1}x_i+n-1}{n}_{{\mathcal R}(p,q)}
		=\sum_{r_j=0}^{n}\tau_1^{\sum_{j=1}^{k}r_jz_j} \tau_2^{\sum_{j=1}^{k}x_j(n\!-\!s_j)}\prod_{j=1}^{k+1}\genfrac{[}{]}{0pt}{}{x_j+r_j-1}{r_j}_{{\mathcal R}(p,q)}.
	\end{eqnarray*}
\end{small}
\item[(b)] Another expressions of the generalized multivariate  Cauchy formula are:
\begin{small}
	\begin{eqnarray*}
		\genfrac{[}{]}{0pt}{}{r+k}{n+k}_{{\mathcal R}(p,q)}=\sum_{r_j=x_j}^{r}\tau_1^{\sum_{j=1}^{k}(x_j+1)(n-s_j-r+y_j)} \tau_2^{\sum_{j=1}^{k}(r_j-x_j)(n-y_j+k-j+1)}\prod_{j=1}^{k+1}\genfrac{[}{]}{0pt}{}{r_j}{x_j}_{{\mathcal R}(p,q)}
	\end{eqnarray*}
	and
	\begin{eqnarray*}
		\,\genfrac{[}{]}{0pt}{}{r+k}{n+k}_{{\mathcal R}(p,q)}
		=\sum_{r_j=x_j}^{r}\tau_1^{\sum_{j=1}^{k}(r_j-x_j)(n-y_j+k-j+1)} \tau_2^{\sum_{j=1}^{k}(x_j+1)(n-s_j-r+y_j)}\prod_{j=1}^{k+1}\genfrac{[}{]}{0pt}{}{r_j}{x_j}_{{\mathcal R}(p,q)},
	\end{eqnarray*}
	where $s_j=\sum_{i=1}^jr_i$, $y_j=\sum_{i=1}^{j}x_i$, $j\in\{1,2,\ldots,k\}$, $r_{k+1}=r-s_k$, $x_{k+1}=r-y_k$,  with $\sum_{i=1}^kr_i\leq r.$
\end{small}
\end{enumerate}
\end{remark}
\begin{remark}
	Here, we determine the particular case of the generalized multivariate Vandermonde and Cauchy formulae corresponding to the quantum deformed algebras known in the literature.
	\begin{enumerate}
	\item[(a)]\begin{small}
		Setting $\mathcal{R}(x)=\frac{x-x^{-1}}{q-q^{-1}},$ we deduce the  multivariate Vandermonde and Cauchy formulae from  the quantum algebra\cite{BC,M}:
		\begin{eqnarray*}
			[x_1+\cdots+x_{k+1}]_{n,q}=\sum_{r_j=0}^{n}\genfrac{[}{]}{0pt}{}{n}{r_1,r_2,\cdots,r_k}_{q}
			\mathcal{V}(q,n)\prod_{j=1}^{k+1}[x_j]_{r_j,q}
		\end{eqnarray*}
		and
		\begin{eqnarray*}
			[x_1+\cdots+x_{k+1}]_{n,q}=\sum_{r_j=0}^{n}\genfrac{[}{]}{0pt}{}{n}{r_1,r_2,\cdots,r_k}_{q}
			\mathcal{V}(q,n)\prod_{j=1}^{k+1}[x_j]_{r_j,q},
		\end{eqnarray*}
		where $\mathcal{V}(q,n)=q^{\sum_{j=1}^{k}r_j(z_j-(n-s_j))}q^{-\sum_{j=1}^{k}(n-s_j)(x_j-r_j)}.$ Moreover, 
		the negative multivariate  $q$ -Vandermonde formula is presented  by:
		\begin{eqnarray*}
			[x_1+\cdots+x_{k+1}]_{-n,q}=\sum\genfrac{[}{]}{0pt}{}{-n}{r_1,r_2,\cdots,r_k}_{q}\,\mathcal{V}(q,-n)
			\prod_{j=1}^{k+1}[x_j]_{r_j,q}
		\end{eqnarray*}
		and
		\begin{eqnarray*}
			[x_1+\cdots+x_{k+1}]_{-n,q}=\sum\genfrac{[}{]}{0pt}{}{-n}{r_1,r_2,\cdots,r_k}_{q}\,\mathcal{V}(q,-n)
			\prod_{j=1}^{k+1}[x_j]_{r_j,q},
		\end{eqnarray*}
		where $\mathcal{V}(q,-n)=q^{\sum_{j=1}^{k}r_j(z_j-(-n-s_j))}\,q^{-\sum_{j=1}^{k}(-n-s_j)(x_j-r_j)}.$ Also, 
		the multivariate $q-$  Vandermonde formula can be rewritten as:
		\begin{equation*}
		\Big[\sum_{i=1}^{k+1}x_i+n-1\Big]_{n,q}=\sum_{r_j=0}^{n}\genfrac{[}{]}{0pt}{}{n}{r_1,r_2,\ldots,r_k}_{q}q^{\sum_{j=1}^{k}x_j(n-s_j)}
		q^{-\sum_{j=1}^{k}r_jz_j}\prod_{j=1}^{k+1}[x_j+r_j-1]_{r_j,q},
		\end{equation*}
		and
		\begin{equation*}
		\Big[\sum_{i=1}^{k+1}x_i+n-1\Big]_{n,q}=\sum_{r_j=0}^{n}\genfrac{[}{]}{0pt}{}{n}{r_1,r_2,\ldots,r_k}_{q}q^{\sum_{j=1}^{k}r_jz_j}
		q^{-\sum_{j=1}^{k}x_j(n-s_j)}\prod_{j=1}^{k+1}[x_j+r_j-1]_{r_j,q},
		\end{equation*}
		and 
		the multivariate inverse $q-$ Vandermonde formula as:
		\begin{equation*}
		\frac{1}{[x_{k+1}]_{n,q}}=\sum_{r_j=0}^{n}\genfrac{[}{]}{0pt}{}{n+s_k-1}{r_1,\ldots,r_k}_{q}\frac{q^{\sum_{j=1}^{k}r_j(z_j-s_k+s_j-n+1)}}{
			q^{-\sum_{j=1}^{k}(-n-s_k+s_j)(x_j-r_j)}}\frac{\prod_{j=1}^k[x_j]_{r_j,q}}{[x_1+\cdots+x_{k+1}]_{n+s_k,q}},
		\end{equation*}
		and
		\begin{equation*}
		\frac{1}{[x_{k+1}]_{n,q}}=\sum_{r_j=0}^{n}\genfrac{[}{]}{0pt}{}{n+s_k-1}{r_1,\ldots,r_k}_{q}\frac{q^{\sum_{j=1}^{k}(n+s_k-s_j)(x_j-r_j)}}{
			q^{-\sum_{j=1}^{k}r_j(-z_j+s_k-s_j+n-1)}}\frac{\prod_{j=1}^k[x_j]_{r_j,q}}{[x_1+\cdots+x_{k+1}]_{n+s_k,q}}.
		\end{equation*}
		Furthermore, the multivariate $q$ - Cauchy formula is furnished by:
		\begin{equation*}
		\genfrac{[}{]}{0pt}{}{x_1+x_2+\cdots+x_{k+1}}{n}_{q}=\sum_{r_j=0}^{n} q^{\sum_{j=1}^{k}r_j(z_j-(n-s_j))}q^{-\sum_{j=1}^{k}(n-s_j)(x_j-r_j)}\prod_{j=1}^{k+1}\genfrac{[}{]}{0pt}{}{x_j}{r_j}_{q}
		\end{equation*}
		and
		\begin{equation*}
		\genfrac{[}{]}{0pt}{}{x_1+x_2+\cdots+x_{k+1}}{n}_{q}=\sum_{r_j=0}^{n} q^{\sum_{j=1}^{k}(n-s_j)(x_j-r_j)} q^{-\sum_{j=1}^{k}r_j(z_j-(n-s_j))}\,\prod_{j=1}^{k+1}\genfrac{[}{]}{0pt}{}{x_j}{r_j}_{q}.
		\end{equation*}
		Moreover, the multivariate $q$ - Cauchy formula can be re-written as:
		\begin{equation*}
		\genfrac{[}{]}{0pt}{}{\sum_{i=1}^{k+1}x_i+n-1}{n}_{q}=\sum_{r_j=0}^{n}q^{\sum_{j=1}^{k}x_j(n-s_j)} q^{-\sum_{j=1}^{k}r_jz_j}\prod_{j=1}^{k+1}\genfrac{[}{]}{0pt}{}{x_j+r_j-1}{r_j}_{q}\end{equation*}
		and
		\begin{equation*}
		\genfrac{[}{]}{0pt}{}{\sum_{i=1}^{k+1}x_i+n-1}{n}_{q}
		=\sum_{r_j=0}^{n}q^{\sum_{j=1}^{k}r_jz_j} q^{-\sum_{j=1}^{k}x_j(n\!-\!s_j)}\prod_{j=1}^{k+1}\genfrac{[}{]}{0pt}{}{x_j+r_j-1}{r_j}_{q}.
		\end{equation*}
		Another expressions of the $q-$ deformed multivariate Cauchy formulae are furnished by :
		\begin{equation*}
		\genfrac{[}{]}{0pt}{}{r+k}{n+k}_{q}=\sum_{r_j=0}^{n}q^{\sum_{j=1}^{k}(x_j+1)(n-s_j-r+y_j)}q^{-\sum_{j=1}^{k}(r_j-x_j)(n-y_j+k-j+1)}\prod_{j=1}^{k+1}\genfrac{[}{]}{0pt}{}{r_j}{x_j}_{q}
		\end{equation*}
		and
		\begin{equation*}
		\genfrac{[}{]}{0pt}{}{r+k}{n+k}_{q}
		=\sum_{r_j=0}^{n}q^{\sum_{j=1}^{k}(r_j-x_j)(n-y_j+k-j+1)}q^{-\sum_{j=1}^{k}(x_j+1)(n-s_j-r+y_j)}\prod_{j=1}^{k+1}\genfrac{[}{]}{0pt}{}{r_j}{x_j}_{q}.
		\end{equation*}
	\end{small}
	\item[(b)] 
	\begin{small}
	Setting $\mathcal{R}(x,y)=\frac{x-y}{p-q},$ we deduce the  multivariate Vandermonde and Cauchy formulae associated to the quantum algebra\cite{JS}:
	\begin{eqnarray*}
	[x_1+\cdots+x_{k+1}]_{n,p,q}=\sum_{r_j=0}^{n}\genfrac{[}{]}{0pt}{}{n}{r_1,r_2,\cdots,r_k}_{p,q}
	\mathcal{V}(p\,q,n)\prod_{j=1}^{k+1}[x_j]_{r_j,p,q}
	\end{eqnarray*}
	and
	\begin{eqnarray*}
	[x_1+\cdots+x_{k+1}]_{n,p,q}=\sum_{r_j=0}^{n}\genfrac{[}{]}{0pt}{}{n}{r_1,r_2,\cdots,r_k}_{p,q}
	\mathcal{V}(q\,p,n)\prod_{j=1}^{k+1}[x_j]_{r_j,p,q},
	\end{eqnarray*}
	where $\mathcal{V}(p\,q,n)=p^{\sum_{j=1}^{k}r_j(z_j-(n-s_j))}q^{\sum_{j=1}^{k}(n-s_j)(x_j-r_j)}.$ Moreover, 
	the negative multivariate  $(p,q)$ -Vandermonde formula is presented  by:
		\begin{eqnarray*}
			[x_1+\cdots+x_{k+1}]_{-n,p,q}=\sum\genfrac{[}{]}{0pt}{}{-n}{r_1,r_2,\cdots,r_k}_{p,q}\,\mathcal{V}(p\,q,-n)
			\prod_{j=1}^{k+1}[x_j]_{r_j,p,q}
		\end{eqnarray*}
		and
		\begin{eqnarray*}
			[x_1+\cdots+x_{k+1}]_{-n,p,q}=\sum\genfrac{[}{]}{0pt}{}{-n}{r_1,r_2,\cdots,r_k}_{p,q}\,\mathcal{V}(q\,p,-n)
			\prod_{j=1}^{k+1}[x_j]_{r_j,p,q},
		\end{eqnarray*}
	where $\mathcal{V}(p\,q,-n)=p^{\sum_{j=1}^{k}r_j(z_j-(-n-s_j))}\,q^{\sum_{j=1}^{k}(-n-s_j)(x_j-r_j)}.$ Also, 
		the multivariate $(p,q)-$  Vandermonde formula can be rewritten as:
			\begin{equation*}
				\Big[\sum_{i=1}^{k+1}x_i+n-1\Big]_{n,p,q}=\sum_{r_j=0}^{n}\genfrac{[}{]}{0pt}{}{n}{r_1,r_2,\ldots,r_k}_{p,q}p^{\sum_{j=1}^{k}x_j(n-s_j)}
				q^{\sum_{j=1}^{k}r_jz_j}\prod_{j=1}^{k+1}[x_j+r_j-1]_{r_j,p,q},
			\end{equation*}
			and
			\begin{equation*}
				\Big[\sum_{i=1}^{k+1}x_i+n-1\Big]_{n,p,q}=\sum_{r_j=0}^{n}\genfrac{[}{]}{0pt}{}{n}{r_1,r_2,\ldots,r_k}_{p,q}p^{\sum_{j=1}^{k}r_jz_j}
				q^{\sum_{j=1}^{k}x_j(n-s_j)}\prod_{j=1}^{k+1}[x_j+r_j-1]_{r_j,p,q},
			\end{equation*}
	and 
	the multivariate inverse $(p,q)-$ Vandermonde formula as:
		\begin{equation*}
			\frac{1}{[x_{k+1}]_{n,p,q}}=\sum_{r_j=0}^{n}\genfrac{[}{]}{0pt}{}{n+s_k-1}{r_1,\ldots,r_k}_{p,q}\frac{p^{\sum_{j=1}^{k}r_j(z_j-s_k+s_j-n+1)}}{
			q^{\sum_{j=1}^{k}(-n-s_k+s_j)(x_j-r_j)}}\frac{\prod_{j=1}^k[x_j]_{r_j,p,q}}{[x_1+\cdots+x_{k+1}]_{n+s_k,p,q}},
		\end{equation*}
		 and
		\begin{equation*}
			\frac{1}{[x_{k+1}]_{n,p,q}}=\sum_{r_j=0}^{n}\genfrac{[}{]}{0pt}{}{n+s_k-1}{r_1,\ldots,r_k}_{p,q}\frac{p^{\sum_{j=1}^{k}(n+s_k-s_j)(x_j-r_j)}}{
			q^{\sum_{j=1}^{k}r_j(-z_j+s_k-s_j+n-1)}}\frac{\prod_{j=1}^k[x_j]_{r_j,p,q}}{[x_1+\cdots+x_{k+1}]_{n+s_k,p,q}}.
		\end{equation*}
	
Furthermore, the multivariate $(p,q)$ - Cauchy formula is furnished by:
		\begin{equation*}
			\genfrac{[}{]}{0pt}{}{x_1+x_2+\cdots+x_{k+1}}{n}_{p,q}=\sum_{r_j=0}^{n} p^{\sum_{j=1}^{k}r_j(z_j-(n-s_j))}q^{\sum_{j=1}^{k}(n-s_j)(x_j-r_j)}\prod_{j=1}^{k+1}\genfrac{[}{]}{0pt}{}{x_j}{r_j}_{p,q}
		\end{equation*}
		and
		\begin{equation*}
			\genfrac{[}{]}{0pt}{}{x_1+x_2+\cdots+x_{k+1}}{n}_{p,q}=\sum_{r_j=0}^{n} p^{\sum_{j=1}^{k}(n-s_j)(x_j-r_j)} q^{\sum_{j=1}^{k}r_j(z_j-(n-s_j))}\,\prod_{j=1}^{k+1}\genfrac{[}{]}{0pt}{}{x_j}{r_j}_{p,q}.
		\end{equation*}
Moreover, the multivariate $(p,q)$ - Cauchy formula can be re-written as:
		\begin{equation*}
			\genfrac{[}{]}{0pt}{}{\sum_{i=1}^{k+1}x_i+n-1}{n}_{p,q}=\sum_{r_j=0}^{n}p^{\sum_{j=1}^{k}x_j(n-s_j)} q^{\sum_{j=1}^{k}r_jz_j}\prod_{j=1}^{k+1}\genfrac{[}{]}{0pt}{}{x_j+r_j-1}{r_j}_{p,q}\end{equation*}
		and
		\begin{equation*}
			\genfrac{[}{]}{0pt}{}{\sum_{i=1}^{k+1}x_i+n-1}{n}_{p,q}
			=\sum_{r_j=0}^{n}p^{\sum_{j=1}^{k}r_jz_j} q^{\sum_{j=1}^{k}x_j(n\!-\!s_j)}\prod_{j=1}^{k+1}\genfrac{[}{]}{0pt}{}{x_j+r_j-1}{r_j}_{p,q}.
		\end{equation*}
	Another expressions of the $(p,q)-$ deformed multivariate Cauchy formulae are furnished by :
		\begin{equation*}
			\genfrac{[}{]}{0pt}{}{r+k}{n+k}_{p,q}=\sum_{r_j=0}^{n}p^{\sum_{j=1}^{k}(x_j+1)(n-s_j-r+y_j)}q^{\sum_{j=1}^{k}(r_j-x_j)(n-y_j+k-j+1)}\prod_{j=1}^{k+1}\genfrac{[}{]}{0pt}{}{r_j}{x_j}_{p,q}
		\end{equation*}
		and
		\begin{equation*}
			\genfrac{[}{]}{0pt}{}{r+k}{n+k}_{p,q}
			=\sum_{r_j=0}^{n}p^{\sum_{j=1}^{k}(r_j-x_j)(n-y_j+k-j+1)}q^{\sum_{j=1}^{k}(x_j+1)(n-s_j-r+y_j)}\prod_{j=1}^{k+1}\genfrac{[}{]}{0pt}{}{r_j}{x_j}_{p,q}.
		\end{equation*}
	\end{small}
	\item[(c)]
	\begin{small}
	Setting $\mathcal{R}(x,y)=\frac{1-xy}{(p^{-1}-q)x},$ we deduce the multivariate Vandermonde and Cauchy formulae related to the quantum algebra \cite{CJ}:
	\begin{eqnarray*}
		[x_1+\cdots+x_{k+1}]_{n,p,q}=\sum_{r_j=0}^{n}\genfrac{[}{]}{0pt}{}{n}{r_1,r_2,\cdots,r_k}_{p^{-1},q}
		\mathcal{V}(p\,q,n)\prod_{j=1}^{k+1}[x_j]_{r_j,p^{-1},q}
	\end{eqnarray*}
	and
	\begin{eqnarray*}
		[x_1+\cdots+x_{k+1}]_{n,p^{-1},q}=\sum_{r_j=0}^{n}\genfrac{[}{]}{0pt}{}{n}{r_1,r_2,\cdots,r_k}_{p^{-1},q}
		\mathcal{V}(q\,p,n)\prod_{j=1}^{k+1}[x_j]_{r_j,p^{-1},q},
	\end{eqnarray*}
	where $\mathcal{V}(p\,q,n)=p^{-\sum_{j=1}^{k}r_j(z_j-(n-s_j))}q^{\sum_{j=1}^{k}(n-s_j)(x_j-r_j)}.$ Moreover, 
	the negative multivariate  $(p^{-1},q)$ -Vandermonde formula is presented  by:
	\begin{eqnarray*}
		[x_1+\cdots+x_{k+1}]_{-n,p^{-1},q}=\sum\genfrac{[}{]}{0pt}{}{-n}{r_1,r_2,\cdots,r_k}_{p^{-1},q}\,\mathcal{V}(p\,q,-n)
		\prod_{j=1}^{k+1}[x_j]_{r_j,p^{-1},q}
	\end{eqnarray*}
	and
	\begin{eqnarray*}
		[x_1+\cdots+x_{k+1}]_{-n,p^{-1},q}=\sum\genfrac{[}{]}{0pt}{}{-n}{r_1,r_2,\cdots,r_k}_{p^{-1},q}\,\mathcal{V}(q\,p,-n)
		\prod_{j=1}^{k+1}[x_j]_{r_j,p^{-1},q},
	\end{eqnarray*}
	where $\mathcal{V}(p\,q,-n)=p^{-\sum_{j=1}^{k}r_j(z_j-(-n-s_j))}\,q^{\sum_{j=1}^{k}(-n-s_j)(x_j-r_j)}.$ Also, 
	the multivariate $(p^{-1},q)-$  Vandermonde formula can be rewritten as:
	\begin{equation*}
	\Big[\sum_{i=1}^{k+1}x_i+n-1\Big]_{n,p^{-1},q}=\sum_{r_j=0}^{n}\genfrac{[}{]}{0pt}{}{n}{r_1,r_2,\ldots,r_k}_{p^{-1},q}p^{-\sum_{j=1}^{k}x_j(n-s_j)}
	q^{\sum_{j=1}^{k}r_jz_j}\prod_{j=1}^{k+1}[x_j+r_j-1]_{r_j,p^{-1},q},
	\end{equation*}
	and
	\begin{equation*}
	\Big[\sum_{i=1}^{k+1}x_i+n-1\Big]_{n,p^{-1},q}=\sum_{r_j=0}^{n}\genfrac{[}{]}{0pt}{}{n}{r_1,r_2,\ldots,r_k}_{p^{-1},q}p^{-\sum_{j=1}^{k}r_jz_j}
	q^{\sum_{j=1}^{k}x_j(n-s_j)}\prod_{j=1}^{k+1}[x_j+r_j-1]_{r_j,p^{-1},q},
	\end{equation*}
	and 
	the multivariate inverse $(p^{-1},q)-$ Vandermonde formula as:
	\begin{equation*}
	\frac{1}{[x_{k+1}]_{n,p^{-1},q}}=\sum_{r_j=0}^{n}\genfrac{[}{]}{0pt}{}{n+s_k-1}{r_1,\ldots,r_k}_{p^{-1},q}\frac{p^{-\sum_{j=1}^{k}r_j(z_j-s_k+s_j-n+1)}}{
		q^{\sum_{j=1}^{k}(-n-s_k+s_j)(x_j-r_j)}}\frac{\prod_{j=1}^k[x_j]_{r_j,p^{-1},q}}{[x_1+\cdots+x_{k+1}]_{n+s_k,p^{-1},q}},
	\end{equation*}
	and
	\begin{equation*}
	\frac{1}{[x_{k+1}]_{n,p^{-1},q}}=\sum_{r_j=0}^{n}\genfrac{[}{]}{0pt}{}{n+s_k-1}{r_1,\ldots,r_k}_{p^{-1},q}\frac{p^{-\sum_{j=1}^{k}(n+s_k-s_j)(x_j-r_j)}}{
		q^{\sum_{j=1}^{k}r_j(-z_j+s_k-s_j+n-1)}}\frac{\prod_{j=1}^k[x_j]_{r_j,p^{-1},q}}{[x_1+\cdots+x_{k+1}]_{n+s_k,p^{-1},q}}.
	\end{equation*}
	
	Furthermore, the multivariate $(p^{-1},q)$ - Cauchy formula is furnished by:
	\begin{equation*}
	\genfrac{[}{]}{0pt}{}{x_1+x_2+\cdots+x_{k+1}}{n}_{p^{-1},q}=\sum_{r_j=0}^{n} p^{-\sum_{j=1}^{k}r_j(z_j-(n-s_j))}q^{\sum_{j=1}^{k}(n-s_j)(x_j-r_j)}\prod_{j=1}^{k+1}\genfrac{[}{]}{0pt}{}{x_j}{r_j}_{p^{-1},q}
	\end{equation*}
	and
	\begin{equation*}
	\genfrac{[}{]}{0pt}{}{x_1+x_2+\cdots+x_{k+1}}{n}_{p^{-1},q}=\sum_{r_j=0}^{n} p^{-\sum_{j=1}^{k}(n-s_j)(x_j-r_j)} q^{\sum_{j=1}^{k}r_j(z_j-(n-s_j))}\,\prod_{j=1}^{k+1}\genfrac{[}{]}{0pt}{}{x_j}{r_j}_{p^{-1},q}.
	\end{equation*}
	Moreover, the multivariate $(p^{-1},q)$ - Cauchy formula can be re-written as:
	\begin{equation*}
	\genfrac{[}{]}{0pt}{}{\sum_{i=1}^{k+1}x_i+n-1}{n}_{p^{-1},q}=\sum_{r_j=0}^{n}p^{-\sum_{j=1}^{k}x_j(n-s_j)} q^{\sum_{j=1}^{k}r_jz_j}\prod_{j=1}^{k+1}\genfrac{[}{]}{0pt}{}{x_j+r_j-1}{r_j}_{p^{-1},q}\end{equation*}
	and
	\begin{equation*}
	\genfrac{[}{]}{0pt}{}{\sum_{i=1}^{k+1}x_i+n-1}{n}_{p^{-1},q}
	=\sum_{r_j=0}^{n}p^{-\sum_{j=1}^{k}r_jz_j} q^{\sum_{j=1}^{k}x_j(n\!-\!s_j)}\prod_{j=1}^{k+1}\genfrac{[}{]}{0pt}{}{x_j+r_j-1}{r_j}_{p^{-1},q}.
	\end{equation*}
	Another expressions of the $(p^{-1},q)-$ deformed multivariate Cauchy formulae are furnished by :
	\begin{equation*}
	\genfrac{[}{]}{0pt}{}{r+k}{n+k}_{p^{-1},q}=\sum_{r_j=0}^{n}p^{-\sum_{j=1}^{k}(x_j+1)(n-s_j-r+y_j)}q^{\sum_{j=1}^{k}(r_j-x_j)(n-y_j+k-j+1)}\prod_{j=1}^{k+1}\genfrac{[}{]}{0pt}{}{r_j}{x_j}_{p^{-1},q}
	\end{equation*}
	and
	\begin{equation*}
	\genfrac{[}{]}{0pt}{}{r+k}{n+k}_{p^{-1},q}
	=\sum_{r_j=0}^{n}p^{-\sum_{j=1}^{k}(r_j-x_j)(n-y_j+k-j+1)}q^{\sum_{j=1}^{k}(x_j+1)(n-s_j-r+y_j)}\prod_{j=1}^{k+1}\genfrac{[}{]}{0pt}{}{r_j}{x_j}_{p^{-1},q}.
	\end{equation*}
\end{small}
	\item[(d)]
	Putting $\mathcal{R}(x,y)=\frac{xy-1}{(q^{-1}-p)y},$ we derive the multivariate Vandermonde and Cauchy formulae related to the quantum algebra \cite{Hounkonnou&Ngompe07a}:
	\begin{eqnarray*}
	[x_1+\cdots+x_{k+1}]^Q_{n,p,q}=\sum_{r_j=0}^{n}\genfrac{[}{]}{0pt}{}{n}{r_1,r_2,\cdots,r_k}^Q_{p,q}
	\mathcal{V}(p\,q,n)\prod_{j=1}^{k+1}[x_j]^Q_{r_j,p,q}
	\end{eqnarray*}
	and
	\begin{eqnarray*}
	[x_1+\cdots+x_{k+1}]^Q_{n,p,q}=\sum_{r_j=0}^{n}\genfrac{[}{]}{0pt}{}{n}{r_1,r_2,\cdots,r_k}^Q_{p,q}
	\mathcal{V}(q\,p,n)\prod_{j=1}^{k+1}[x_j]^Q_{r_j,p,q},
	\end{eqnarray*}
	where $\mathcal{V}(p\,q,n)=p^{\sum_{j=1}^{k}r_j(z_j-(n-s_j))}q^{-\sum_{j=1}^{k}(n-s_j)(x_j-r_j)}.$ It's  can be rewritten as:
		\begin{small}
			\begin{equation*}
				\Big[\sum_{i=1}^{k+1}x_i+n-1\Big]^Q_{n,p,q}=\sum_{r_j=0}^{n}\genfrac{[}{]}{0pt}{}{n}{r_1,r_2,\ldots,r_k}^Q_{p,q}p^{\sum_{j=1}^{k}x_j(n-s_j)}
				q^{-\sum_{j=1}^{k}r_jz_j}\prod_{j=1}^{k+1}[x_j+r_j-1]^Q_{r_j,p,q},
			\end{equation*}
			and
			\begin{equation*}
				\Big[\sum_{i=1}^{k+1}x_i+n-1\Big]^Q_{n,p,q}=\sum_{r_j=0}^{n}\genfrac{[}{]}{0pt}{}{n}{r_1,r_2,\ldots,r_k}^Q_{p,q}p^{\sum_{j=1}^{k}r_jz_j}
				q^{-\sum_{j=1}^{k}x_j(n-s_j)}\prod_{j=1}^{k+1}[x_j+r_j-1]^Q_{r_j,p,q},
			\end{equation*}
		\end{small}
	Furthermore, 
	the corresponding  negative multivariate  Vandermonde formula is presented  by:
	\begin{small}
		\begin{eqnarray*}
			[x_1+\cdots+x_{k+1}]^Q_{-n,p,q}=\sum_{r_j=0}^{n}\genfrac{[}{]}{0pt}{}{-n}{r_1,r_2,\cdots,r_k}^Q_{p,q}\,\mathcal{V}(p\,q,-n)
			\prod_{j=1}^{k+1}[x_j]^Q_{r_j,p,q},
		\end{eqnarray*}
		\begin{eqnarray*}
			[x_1+\cdots+x_{k+1}]^Q_{-n,p,q}=\sum\genfrac{[}{]}{0pt}{}{-n}{r_1,r_2,\cdots,r_k}^Q_{p,q}\,\mathcal{V}(q\,p,-n)
			\prod_{j=1}^{k+1}[x_j]^Q_{r_j,p,q},
		\end{eqnarray*}
	\end{small}
	where $\mathcal{V}(p\,q,-n)=p^{\sum_{j=1}^{k}r_j(z_j-(-n-s_j))}\,q^{\sum_{j=1}^{k}(-n-s_j)(x_j-r_j)},$ and 
	the multivariate inverse  Vandermonde formula as:
	\begin{small}
		\begin{equation*}
			\frac{1}{[x_{k+1}]^Q_{n,p,q}}=\sum_{r_j=0}^{n}\genfrac{[}{]}{0pt}{}{n+s_k-1}{r_1,\ldots,r_k}^Q_{p,q}\frac{p^{\sum_{j=1}^{k}r_j(z_j-s_k+s_j-n+1)}}{
			q^{-\sum_{j=1}^{k}(-n-s_k+s_j)(x_j-r_j)}}\frac{\prod_{j=1}^k[x_j]^Q_{r_j,p,q}}{[x_1+\cdots+x_{k+1}]^Q_{n+s_k,p,q}},
		\end{equation*}
		\begin{equation*}
			\frac{1}{[x_{k+1}]^Q_{n,p,q}}=\sum_{r_j=0}^{n}\genfrac{[}{]}{0pt}{}{n+s_k-1}{r_1,\ldots,r_k}^Q_{p,q}\frac{p^{\sum_{j=1}^{k}(n+s_k-s_j)(x_j-r_j)}}{
			q^{-\sum_{j=1}^{k}r_j(-z_j+s_k-s_j+n-1)}}\frac{\prod_{j=1}^k[x_j]^Q_{r_j,p,q}}{[x_1+\cdots+x_{k+1}]^Q_{n+s_k,p,q}}.
		\end{equation*}
	\end{small}
Moreover, 
		the Hounkonnou-Ngompe generalized $q-$ Quesne multivariate Cauchy formula is furnished by:
	\begin{small}
		\begin{equation*}
			\genfrac{[}{]}{0pt}{}{x_1+x_2+\cdots+x_{k+1}}{n}^Q_{p,q}=\sum_{r_j=0}^{n} p^{\sum_{j=1}^{k}r_j(z_j-(n-s_j))}q^{-\sum_{j=1}^{k}(n-s_j)(x_j-r_j)}\prod_{j=1}^{k+1}\genfrac{[}{]}{0pt}{}{x_j}{r_j}^Q_{p,q},
		\end{equation*}
		\begin{equation*}
			\genfrac{[}{]}{0pt}{}{x_1+x_2+\cdots+x_{k+1}}{n}^Q_{p,q}=\sum_{r_j=0}^{n} p^{\sum_{j=1}^{k}(n-s_j)(x_j-r_j)} q^{-\sum_{j=1}^{k}r_j(z_j-(n-s_j))}\,\prod_{j=1}^{k+1}\genfrac{[}{]}{0pt}{}{x_j}{r_j}^Q_{p,q}. 
		\end{equation*}
	\end{small}
It's  can be re-written as:
	\begin{small}
		\begin{equation*}
			\genfrac{[}{]}{0pt}{}{\sum_{i=1}^{k+1}x_i+n-1}{n}^Q_{p,q}=\sum_{r_j=0}^{n}p^{\sum_{j=1}^{k}x_j(n-s_j)} q^{-\sum_{j=1}^{k}r_jz_j}\prod_{j=1}^{k+1}\genfrac{[}{]}{0pt}{}{x_j+r_j-1}{r_j}^Q_{p,q}
			\end{equation*}
		and
		\begin{equation*}
			\genfrac{[}{]}{0pt}{}{\sum_{i=1}^{k+1}x_i+n-1}{n}^Q_{p,q}
			=\sum_{r_j=0}^{n}p^{\sum_{j=1}^{k}r_jz_j} q^{-\sum_{j=1}^{k}x_j(n\!-\!s_j)}\prod_{j=1}^{k+1}\genfrac{[}{]}{0pt}{}{x_j+r_j-1}{r_j}^Q_{p,q}.
		\end{equation*}
	\end{small}
	Another expressions of the Hounkonnou-Ngompe generalized $q-$ Quesne  multivariate Cauchy formulae are furnished by :
	\begin{small}
		\begin{equation*}
			\genfrac{[}{]}{0pt}{}{r+k}{n+k}^Q_{p,q}=\sum_{r_j=0}^{n}p^{\sum_{j=1}^{k}(x_j+1)(n-s_j-r+y_j)}q^{-\sum_{j=1}^{k}(r_j-x_j)(n-y_j+k-j+1)}\prod_{j=1}^{k+1}\genfrac{[}{]}{0pt}{}{r_j}{x_j}^Q_{p,q}
		\end{equation*}
		and
		\begin{equation*}
			\genfrac{[}{]}{0pt}{}{r+k}{n+k}^Q_{p,q}
			=\sum_{r_j=0}^{n}p^{\sum_{j=1}^{k}(r_j-x_j)(n-y_j+k-j+1)}q^{-\sum_{j=1}^{k}(x_j+1)(n-s_j-r+y_j)}\prod_{j=1}^{k+1}\genfrac{[}{]}{0pt}{}{r_j}{x_j}^Q_{p,q}.
		\end{equation*}
	\end{small}
	\end{enumerate}
\end{remark}
\section{Mutivariate probability distributions from $\mathcal {R}(p,q)-$ deformed quantum algebras }
In this section, we construct some multivariate probability distributions ( P\'olya, inverse P\'olya, hypergeometric and negative hypergeometric) from the generalized quantum algebras \cite{HB1}. The corresponding bivariate probability distributions and properties are also investigated.
\subsection{Multivariate $\mathcal {R}(p,q)-$ deformed P\'{o}lya distribution}\label{sec3}
We denotes by $Y_\mu$  the number of balls of color $c_\mu$ drawn in $n$ $\mathcal{R}(p,q)$-drawings in a multiple $\mathcal{R}(p,q)$-P\'olya urn model, with conditional probability of drawing a ball of color $c_\mu$ at the $i$th $\mathcal{R}(p,q)$-drawing, given that $j_\mu-1$ balls of color $c_\mu$ and a total of $i_{\mu-1}$ balls of colors $c_1,c_2,\cdots,c_{\mu-1}$ are drawn in the previous $i-1$ $\mathcal{R}(p,q)$-drawings, is given by 
\begin{eqnarray}\label{eq3.1}
p_{i,j_\mu}(i_{\mu-1})=\frac{\tau_2^{s_{\mu-1}+mi_{\mu-1}}[r_\mu+m(j_\mu-1)]_{\mathcal
		{R}(p,q)}}{[r+m(i-1)]_{\mathcal
		{R}(p,q)}}
,\quad \mu\in\{1,2,\cdots,k\}
\end{eqnarray}

 The probability distribution of the $\mathcal{R}(p,q)$- random vector $\underline{Y}=\big(Y_1,Y_2,\cdots,Y_k\big)$ may be called multivariate $\mathcal{R}(p,q)$- P\'olya distribution, with parameters $n,$ $(\beta_1,\beta_2,\cdots,\beta_k)$, $\beta$, $p,$  and $q.$
\begin{theorem}\label{thm3.1}
	The mass function of the  multivariate  $\mathcal{R}(p,q)$-P\'olya  distribution, with parameters $n$, $(\beta_1,\beta_2,\ldots,\beta_k)$, $\beta$, $p,$  and $q$, is presented as follows:
	\begin{small}
		\begin{eqnarray}\label{eq3.2}
		P(\underline{Y}=\underline{y})=\Psi_k(p,q)\genfrac{[}{]}{0pt}{}{n}{y_1,\ldots,y_k}_{\mathcal{R}(p^{-m},q^{-m})}\frac{\prod_{j=1}^{k+1}[\beta_j]_{y_j,\mathcal{R}(p^{-m},q^{-m})}}{ [\beta]_{n,\mathcal{R}(p^{-m},q^{-m})}},
		\end{eqnarray}
	\end{small}
	where $\Psi_k(p,q)=\tau^{-m\sum_{j=1}^{k}y_{j}(\beta_{j+1}-y_{j+1})}_1\tau_2^{-m\sum_{j=1}^{k}(n-x_j)(\beta_j-y_j)},$
	 $y_j\in\{0,1,\ldots,n\}$, $j\in\{1,2,\ldots,k\},$  $\sum_{j=1}^k y_j\leq n$, $y_{k+1}=n-\sum_{j=1}^k y_j$, $\beta_{k+1}=\beta-\sum_{j=1}^k \beta_j$, and $y_j=\sum_{i=1}^{j}y_i.$
\end{theorem}
\begin{proof}
 For the proof, we use the total probability theorem. Moreover, 
	  the probabilities (\ref{eq3.2}) sum to unity according to the  multivariate $\mathcal
	{R}(p,q)$-Vandermonde formula (\ref{eq2.7}) , the multivariate $\mathcal
	{R}(p,q)$-Cauchy formula (\ref{eq2.9}). 
	 $\cqfd$
\end{proof} 
\begin{remark}
Another relation of the multivariate P\'olya distribution form generalized quantum algeras is interested.
	\begin{enumerate}
		\item[(i)]
		The multivariate $\mathcal{R}(p,q)$-P\'{o}lya probability  distribution (\ref{eq3.2}) can be rewritten in the form:
		\begin{eqnarray}\label{eq3.2bis}
			P(Y_1=y_1,\ldots,Y_k=y_k)=\frac{\tau^{-m\sum_{j=1}^{k}y_{j}(\beta_{j+1}-x_{j+1})}_1}{\tau_2^{m\sum_{j=1}^{k}(n-x_j)(\beta_j-y_j)}}
			\frac{\prod_{j=1}^{k+1} \genfrac{[}{]}{0pt}{}{\beta_j}{y_j}_{\mathcal{R}(p^{-m},q^{-m})}}{\genfrac{[}{]}{0pt}{}{\beta}{n}_{\mathcal{R}(p^{-m},q^{-m})}}.
		\end{eqnarray}
		 Indeed, it's derived  from \eqref{eq3.2} according to  the relations
	\begin{eqnarray*}
	\genfrac{[}{]}{0pt}{}{n}{y_1,y_2,\cdots,y_k}_{\mathcal
		{R}(p,q)}=\frac{[n]_{\mathcal
			{R}(p,q)}!}{[y_1]_{\mathcal
			{R}(p,q)}![y_2]_{\mathcal
			{R}(p,q)}!\cdots[y_k]_{\mathcal
			{R}(p,q)}![y_{k+1}]_{\mathcal
			{R}(p,q)}!}
	\end{eqnarray*}
	and 
	\begin{eqnarray*}
	\genfrac{[}{]}{0pt}{}{\beta_j}{y_j}_{\mathcal
		{R}(p,q)}=\frac{[\beta_j]_{y_j,{\mathcal
				{R}(p,q)}}}{[y_j]_{\mathcal
			{R}(p,q)}!}.
	\end{eqnarray*}
		\item[(ii)] Taking $k=1,$ we obtain the $\mathcal{R}(p,q)-$ P\'olya distribution given in\cite{HMD}.
	\end{enumerate}
\end{remark}

We assume that the $\mathcal{R}(p,q)$- random vector $\big(Y_1,Y_2,\ldots,Y_k\big)$ satisfy the multivariate  $\mathcal{R}(p,q)$- P\'olya probability distribution with parameters $n,$ $(\beta_1,\beta_2,\ldots,\beta_k)$, $\beta$, $p,$  and $q.$ Then:
\begin{theorem}\label{thm3.2}
	For $\mu\in\{1,2,\cdots,k\},$ then, the  marginal distribution of the $\mathcal{R}(p,q)$- random vector $(Y_1,Y_2,\cdots,Y_\mu)$ is a  $\mathcal{R}(p,q)$-P\'olya distribution of order $\mu,$ with parameters $n,$ $(\beta_1,\beta_2,\cdots,\beta_\mu)$, $\beta$, $p,$ and $q.$ 
	
	Moreover, for $\mu\in\{1,2,\cdots,k-t\}$ and $t\in\{1,2,\cdots,k-1\},$  the conditional distribution of the $\mathcal{R}(p,q)$-random vector$(Y_{\mu+1},Y_{\mu+2},\ldots,Y_{\mu+t})$, given that $(Y_1,Y_2,\ldots,Y_{\mu})=(y_1,y_2,\cdots,y_{\mu})$ is a  $\mathcal{R}(p,q)$-P\'olya probability distribution of order $t,$ with parameters $n$, $(\beta_{\mu+1},\beta_{\mu+2},\cdots,\beta_{\mu+t})$, $\beta-\beta_1-\beta_2-\cdots-\beta_{\mu}$, $p,$ and $q.$ 
\end{theorem}
\begin{proof} For	the proof, we use the probability \eqref{eq3.2bis}. Summing it for $y_j\in\{0,1,\cdots,\linebreak n-x_{\mu}\},$ $j\in\{\mu+1,\mu+2,\cdots,k\},$ with
$y_{\mu+1}+y_{\mu+2}+\cdots+y_k\leq n-x_\mu$,  and according to the relation (\ref{eq2.9}), we have:
\begin{eqnarray*}
P(Y_1=y_1,Y_2=y_2,\cdots,Y_{\mu}=y_{\mu})&=&\Psi_k(p,q)\prod_{j=1}^{\mu}\genfrac{[}{]}{0pt}{}{\beta_j}{y_j}_{\mathcal{R}(p^{-m},q^{-m})}
\sum \Phi_k(p,q)\nonumber\\&\times&\prod_{j=\mu+1}^{k}\genfrac{[}{]}{0pt}{}{\beta_j}{y_j}_{_{\mathcal{R}(p^{-m},q^{-m})}}
\frac{\genfrac{[}{]}{0pt}{}{\beta-\beta_1-\cdots-\beta_k}{n-y_1-\cdots-y_k}_{_{\mathcal{R}(p^{-m},q^{-m})}}}{\genfrac{[}{]}{0pt}{}{\beta}{n}_{_{\mathcal{R}(p^{-m},q^{-m})}}}\nonumber\\
&=&\Psi_k(p,q)\prod_{j=1}^{\mu}\genfrac{[}{]}{0pt}{}{\beta_j}{y_j}_{_{\mathcal{R}(p^{-m},q^{-m})}}
\frac{\genfrac{[}{]}{0pt}{}{\beta-\beta_1-\cdots-\beta_\mu}{n-y_1-\cdots-y_\mu}_{\mathcal{R}(p^{-m},q^{-m})}
}{\genfrac{[}{]}{0pt}{}{\beta}{n}_{\mathcal{R}(p^{-m},q^{-m})}},
\end{eqnarray*}
which is the probability function of a  $\mathcal{R}(p,q)$- P\'{o}lya distribution of order $\mu,$  with parameters $n,$ $(\beta_1,\beta_2,\cdots,\beta_\mu)$, $\beta$, $p,$ and $q.$

Furthermore, according to the relation, 
\begin{align*}
P(Y_\mu=y_\mu,\cdots,Y_{\mu+k-1}&=y_{\mu+k-1}|Y_1=y_1,Y_2=y_2,\cdots,Y_{\mu-1}=y_{\mu-1})\\
&=\frac{P\big(Y_1=y_1,Y_2=y_2,\ldots,Y_{\mu+k-1}=y_{\mu+k-1})}{P(Y_1=y_1,Y_2=y_2,\ldots,Y_{\mu-1}=y_{\mu-1}\big)},
\end{align*}
we obtain,  the density function of the conditional distribution of the $\mathcal{R}(p,q)$- random vector $(Y_{\mu},Y_{\mu+1},\cdots,Y_{\mu+k-1})$, given that $(Y_1,Y_2,\cdots,Y_{\mu-1})=(y_1,y_2,\cdots,y_{\mu-1}):$
\begin{align*}
P(Y_\mu&=y_\mu,\ldots,Y_{\mu+k-1}=y_{\mu+k-1}|Y_1=y_1,Y_2=y_2,\ldots,Y_{\mu-1}=y_{\mu-1})\\
&=\frac{\Psi_k(p,q)\frac{\prod_{j=1}^{\mu+k-1}\genfrac{[}{]}{0pt}{}{\beta_j}{y_j}_{\mathcal{R}(p^{-m},q^{-m})}
	\genfrac{[}{]}{0pt}{}{\beta-\beta_1-\cdots-\beta_{\mu+k-1}}{n-y_1-\cdots-y_{\mu+k-1}}_{\mathcal{R}(p^{-m},q^{-m})}}{
	{\genfrac{[}{]}{0pt}{}{\beta}{n}_{\mathcal{R}(p^{-m},q^{-m})}}}}
	{\Psi_k(p,q)\frac{\prod_{j=1}^{\mu-1}
	\genfrac{[}{]}{0pt}{}{\beta_j}{y_j}_{\mathcal{R}(p^{-m},q^{-m})}\genfrac{[}{]}{0pt}{}{\beta-\beta_1-\cdots-\beta_{\mu-1}}{n-y_1-\cdots-y_{\mu-1}}_{\mathcal{R}(p^{-m},q^{-m})}
	}{\genfrac{[}{]}{0pt}{}{\beta}{n}_{\mathcal{R}(p^{-m},q^{-m})}}}\\
&=\Psi_k(p,q)\frac{\prod_{j=\mu}^{\mu+k-1}\genfrac{[}{]}{0pt}{}{\beta_j}{y_j}_{\mathcal{R}(p^{-m},q^{-m})}
\genfrac{[}{]}{0pt}{}{\beta-\beta_\mu-\cdots-\beta_{\mu+k-1}}{n-y_\mu-\cdots-y_{\mu+k-1}}_{\mathcal{R}(p^{-m},q^{-m})}
}{\genfrac{[}{]}{0pt}{}{\beta}{n}_{\mathcal{R}(p^{-m},q^{-m})}},
\end{align*}
which is the mass function of a  $\mathcal{R}(p,q)$-P\'olya distribution of order $k,$ with parameters $n,$ $(\beta_\mu,\beta_{\mu+1},\ldots,\beta_{\mu+k-1})$, $\beta-\beta_1-\beta_2-\cdots-\beta_{\mu-1}$, $p,$ and $q$. $\cqfd$
\end{proof}

\subsubsection{Bivariate $\mathcal{R}(p,q)$- P\'olya distribution}
For $k=2$ in the multiple $\mathcal{R}(p,q)$- urn model, we obtain the following results: Then, we denote by $\underline{Y}=\big(Y_1,Y_2\big)$
 the $\mathcal{R}(p,q)$- random vector. Also, $\underline{Y}$ follows the  bivariate $\mathcal{R}(p,q)$- P\'olya distribution with parameters $n$, $\underline{\beta}=(\beta_1,\beta_2),$ $p$ and $q.$ Its probability function is derived  by the following relation:
\begin{small}
	\begin{equation*}\label{bp1}
	P(Y_1=y_1,Y_2=y_2)=\Psi_2(p,q)\genfrac{[}{]}{0pt}{}{n}{y_1,y_2}_{\mathcal{R}(p^{-m},q^{-m})}\frac{\prod_{j=1}^{3}[\beta_j]_{y_j,\mathcal{R}(p^{-m},q^{-m})}}{ [\beta]_{n,\mathcal{R}(p^{-m},q^{-m})}},
	\end{equation*}
	equivalently,
	\begin{eqnarray*}
	P(Y_1=y_1,Y_2=y_2)=\Psi_2(p,q)
	\frac{\prod_{j=1}^{3} \genfrac{[}{]}{0pt}{}{\beta_j}{y_j}_{\mathcal{R}(p^{-m},q^{-m})}}{\genfrac{[}{]}{0pt}{}{\beta}{n}_{\mathcal{R}(p^{-m},q^{-m})}},
	\end{eqnarray*}
\end{small}
where $\Psi_2(p,q)=\tau^{-m\sum_{j=1}^{2}y_{j}(\beta_{j+1}-y_{j+1})}_1\tau_2^{-m\sum_{j=1}^{2}(n-x_j)(\beta_j-y_j)},$
$y_j\in\{0,1,\cdots,n\}$, $j\in\{1,2\},$  $y_1+y_2\leq n$, $y_{3}=n-y_1+y_2$, $\beta_{3}=\beta-\beta_1-\beta_2$, $x_1=y_1,$
and $x_2=y_1+y_2.$

The properties of the bivariate $\mathcal{R}(p,q)$- P\'olya distribution  are presented as: 
\begin{theorem}
	The $\mathcal{R}(p,q)$-factorial moments of the bivariate $\mathcal{R}(p,q)$-P\'olya probability distribution, with parameters $n$, $\underline{\beta}=(\beta_1,\beta_2),$ $p$ and $q,$ are given by:
	\begin{small}
	\begin{eqnarray}\label{bp2}
	E\big([Y_1]_{i_1,\mathcal{R}(p^m,q^m)}\big)=\frac{[n]_{i_1,\mathcal{R}(p^m,q^m)}[\beta_1]_{i_1,\mathcal{R}(p^{-m},q^{-m})}}{[\beta]_{i_1,\mathcal{R}(p^{-m},q^{-m})}},\,i_1\in\{0,1,\cdots,n\},
	\end{eqnarray}
	\begin{eqnarray}\label{bp3}
	E\big([Y_2]_{i_1,\mathcal{R}(p^m,q^m)}|Y_1=y_1\big)=\frac{[n-y_1]_{i_2,\mathcal{R}(p^m,q^m)}[\beta_2]_{i_2,\mathcal{R}(p^{-m},q^{-m})}}{[\beta-\beta_1]_{i_2,\mathcal{R}(p^{-m},q^{-m})}},\,i_2\in\{0,1,\cdots,n-y_1\},
	\end{eqnarray}
	\begin{eqnarray}\label{bp4}
	E\big([Y_2]_{i_2,\mathcal{R}(p^m,q^m)}\big)=\frac{[n]_{i_2,\mathcal{R}(p^m,q^m)}[\beta_2]_{i_2,\mathcal{R}(p^{-m},q^{-m})}\tau^{-mi_2\beta_1}_2}{[\beta]_{i_2,\mathcal{R}(p^{-m},q^{-m})}},\,\,i_2\in\{0,1,\cdots,n\},
	\end{eqnarray}
	and
	\begin{small}
	\begin{eqnarray}\label{bp5}
	E\big([Y_1]_{i_1,\mathcal{R}(p^m,q^m)}[Y_2]_{i_2,\mathcal{R}(p^m,q^m)}\big)=\frac{[n]_{i_1+i_2,\mathcal{R}(p^m,q^m)}[\beta_1]_{i_1,\mathcal{R}(p^{-m},q^{-m})}[\beta_2]_{i_2,\mathcal{R}(p^{-m},q^{-m})}}{\tau^{mi_2\beta_1}_2[\beta]_{i_1+i_2,\mathcal{R}(p^{-m},q^{-m})}},
	\end{eqnarray}
\end{small}
	where $i_1\in\{0,1,\cdots,n-i_2\}$ and  $i_2\in\{0,1,\cdots,n\}.$
\end{small}
\end{theorem}
\begin{proof}
The  $\mathcal{R}(p,q)$- random variable  $Y_1$ follows  a $\mathcal{R}(p,q)$-P\'olya distribution, with mass function \cite{HMD}:
\begin{eqnarray*}
	P(Y_1=y_1)=\Psi_2(p,q)\genfrac{[}{]}{0pt}{}{n}{y_1}_{\mathcal{R}(p^{-m},q^{-m})}\frac{[\beta_1]_{y_1,\mathcal{R}(p^{-m},q^{-m})}[\beta-\beta_1]_{n-y_1,\mathcal{R}(p^{-m},q^{-m})}}{ [\beta]_{n,\mathcal{R}(p^{-m},q^{-m})}},
	\end{eqnarray*}
	where $\Psi_2(p,q)=\tau^{-my_{1}(\beta-\beta_1+y_1-n)}_1\tau_2^{-m(n-y_1)(\beta-y_1)}$ and $y_1\in\{0,1,\cdots,n\}.$
	
	Thus, from \cite{HMD}, the $\mathcal{R}(p,q)$-factorial moments of the $\mathcal{R}(p,q)$-random variable $Y_1$ are determined by \eqref{bp2}. Moreover, the conditional distribution of the $\mathcal{R}(p,q)$-random variable $Y_2,$ given that $Y_1=y_1,$ is a $\mathcal{R}(p,q)$- P\'olya distribution, with mass function
	\begin{eqnarray*}
	P(Y_2=y_2|Y_1=y_1)&=&\tau^{-my_{1}(\beta-\beta_1+y_1-n)}_1\tau_2^{-m(n-y_1)(\beta-y_1)}\genfrac{[}{]}{0pt}{}{n-y_1}{y_2}_{\mathcal{R}(p^{-m},q^{-m})}\nonumber\\&\times&\frac{[\beta_2]_{y_2,\mathcal{R}(p^{-m},q^{-m})}[\beta-\beta_1-\beta_2]_{n-y_1-y_2,\mathcal{R}(p^{-m},q^{-m})}}{ [\beta-\beta_1]_{n-y_1,\mathcal{R}(p^{-m},q^{-m})}},
	\end{eqnarray*}
	where  $y_2\in\{0,1,\cdots,n-y_1\}.$ Once, using\cite{HMD}, the conditional $\mathcal{R}(p,q)$- factorial moments of $Y_2,$ given that $Y_1=y_1,$ are furnished by \eqref{bp3}.
	
	Now, we compute the $\mathcal{R}(p,q)$- factorial moments  $E\big([Y_2]_{i_2, \mathcal{R}(p^m,q^m)}\big),\, i_2\in\{0,1,\cdots,n\}.$
	Since, 
	\begin{eqnarray*}
	E\big([Y_2]_{i_2, \mathcal{R}(p^m,q^m)}\big)&=&E\bigg(E\big([Y_2]_{i_2, \mathcal{R}(p^{-m},q^{-m})}|Y_1\big)\bigg)\nonumber\\
	&=&\frac{[\beta_2]_{i_2,\mathcal{R}(p^{-m},q^{-m})}}{[\beta-\beta_1]_{i_1,\mathcal{R}(p^{-m},q^{-m})}}E\big([n-Y_1]_{i_2,\mathcal{R}(p^{m},q^{m})}\big)
	\end{eqnarray*}
	and 
	\begin{eqnarray*}
	E\big([n-Y_1]_{i_2,\mathcal{R}(p^{m},q^{m})}\big)&=&\sum_{y_1=0}^{n-i_2}[n-y_1]_{i_2,\mathcal{R}(p^{m},q^{m})}\genfrac{[}{]}{0pt}{}{n}{y_1}_{\mathcal{R}(p^{-m},q^{-m})}\Psi_2(p,q)\nonumber\\&\times&\frac{[\beta_1]_{y_1,\mathcal{R}(p^{-m},q^{-m})}[\beta-\beta_1]_{n-y_1,\mathcal{R}(p^{-m},q^{-m})}}{[\beta]_{n,\mathcal{R}(p^{-m},q^{-m})}}
	\end{eqnarray*}
	From the relations 
	\begin{eqnarray*}
	[n-y_1]_{i_2,\mathcal{R}(p^{m},q^{m})}=\big(\tau_1\tau_2\big)^{mi_2(n-\beta_1)-m{i_2+1\choose 2}}\big(\tau_1\tau_2\big)^{mi_2(\beta_1-y_1)}[n-y_1]_{i_2,\mathcal{R}(p^{-m},q^{-m})},
	\end{eqnarray*}
	\begin{eqnarray*}
	[n]_{i_2,\mathcal{R}(p^{m},q^{m})}=\big(\tau_1\tau_2\big)^{mi_2\,n-m{i_2+1\choose 2}}[n]_{i_2,\mathcal{R}(p^{-m},q^{-m})},
	\end{eqnarray*}
	and
	\begin{eqnarray*}
	[n-y_1]_{i_2,\mathcal{R}(p^{-m},q^{-m})}\genfrac{[}{]}{0pt}{}{n}{y_1}_{\mathcal{R}(p^{-m},q^{-m})}=[n]_{i_2,\mathcal{R}(p^{-m},q^{-m})}\genfrac{[}{]}{0pt}{}{n-i_2}{y_1}_{\mathcal{R}(p^{-m},q^{-m})},
	\end{eqnarray*}
	we obtain the folllowing expression
	\begin{eqnarray*}
	E\big([n-Y_1]_{i_2,\mathcal{R}(p^{m},q^{m})}\big)&=&\frac{[n]_{i_2,\mathcal{R}(p^{m},q^{m})}[n-\beta_1]_{i_2,\mathcal{R}(p^{-m},q^{-m})}}{\tau^{mi_2\alpha_1}_2[\beta]_{i_2,\mathcal{R}(p^{-m},q^{-m})}}\nonumber\\&\times&\sum_{y_1=0}^{n-i_2}[n-y_1]_{i_2,\mathcal{R}(p^{m},q^{m})}\Psi_2(p,q)\genfrac{[}{]}{0pt}{}{n-i_2}{y_1}_{\mathcal{R}(p^{-m},q^{-m})}\nonumber\\&\times&\frac{[\beta_1]_{y_1,\mathcal{R}(p^{-m},q^{-m})}[\beta-\beta_1-i_2]_{n-i_2-y_1,\mathcal{R}(p^{-m},q^{-m})}}{[\beta-i_2]_{n-i_2,\mathcal{R}(p^{-m},q^{-m})}}.
	\end{eqnarray*}
	Using the $\mathcal{R}(p,q)$- Vandermonde formula \cite{HMD},\cite{HMRC}, we have:
	\begin{eqnarray*}
	E\big([n-Y_1]_{i_2,\mathcal{R}(p^{m},q^{m})}\big)&=&\frac{[n]_{i_2,\mathcal{R}(p^{m},q^{m})}[n-\beta_1]_{i_2,\mathcal{R}(p^{-m},q^{-m})}}{\tau^{mi_2\alpha_1}_2[\beta]_{i_2,\mathcal{R}(p^{-m},q^{-m})}}.
	\end{eqnarray*}
	Thus, the relation \eqref{bp4} follows.
	
	Analougsly, we calculate the joint $\mathcal{R}(p,q)$-factorial moments $E\big([Y_1]_{i_1,\mathcal{R}(p^m,q^m)}[Y_2]_{i_2,\mathcal{R}(p^m,q^m)}\big),$
	where $i_1\in\{0,1,\cdots,n\}$ and  $i_2\in\{0,1,\cdots,n-i_1\}.$ They  can  computed according to the relation
	\begin{eqnarray*}
	E\big([Y_1]_{i_1,\mathcal{R}(p^m,q^m)}[Y_2]_{i_2,\mathcal{R}(p^m,q^m)}\big)&=&E\bigg(E\big([Y_1]_{i_1,\mathcal{R}(p^m,q^m)}[Y_2]_{i_2,\mathcal{R}(p^m,q^m)}|Y_1\big)\bigg)\nonumber\\
	&=&\frac{[\beta_2]_{i_2,\mathcal{R}(p^{-m},q^{-m})}}{[\beta-\beta_1]_{i_1,\mathcal{R}(p^{-m},q^{-m})}}E\big([Y_1]_{i_1,\mathcal{R}(p^m,q^m)}[n-Y_1]_{i_2,\mathcal{R}(p^{m},q^{m})}\big).
	\end{eqnarray*}
	Since
	\begin{small}
	\begin{eqnarray*}
	E\big([Y_1]_{i_1,\mathcal{R}(p^m,q^m)}[n-Y_1]_{i_2,\mathcal{R}(p^{m},q^{m})}\big)&=&\sum_{y_1=0}^{n-i_2}[y_1]_{i_1,\mathcal{R}(p^m,q^m)}[n-y_1]_{i_2,\mathcal{R}(p^{m},q^{m})}\genfrac{[}{]}{0pt}{}{n}{y_1}_{\mathcal{R}(p^{-m},q^{-m})}\nonumber\\&\times&\Psi_2(p,q)\frac{[\beta_1]_{y_1,\mathcal{R}(p^{-m},q^{-m})}[\beta-\beta_1]_{n-y_1,\mathcal{R}(p^{-m},q^{-m})}}{[\beta]_{n,\mathcal{R}(p^{-m},q^{-m})}}
	\end{eqnarray*}
	\end{small}
	and using the relations 
	\begin{small}
	\begin{eqnarray*}
	[y_1]_{i_1,\mathcal{R}(p^m,q^m)}=\big(\tau_1\tau_2\big)^{mi_1\,y_1-m{i_1+1\choose 2}}[y_1]_{i_1,\mathcal{R}(p^{-m},q^{-m})},
	\end{eqnarray*}
	\begin{eqnarray*}
	[n-y_1]_{i_2,\mathcal{R}(p^{m},q^{m})}=\big(\tau_1\tau_2\big)^{mi_2(n-\beta_1)-m{i_2+1\choose 2}}\big(\tau_1\tau_2\big)^{mi_2(\beta_1-y_1)}[n-y_1]_{i_2,\mathcal{R}(p^{-m},q^{-m})},
	\end{eqnarray*}
	\begin{eqnarray*}
	[n]_{i_2,\mathcal{R}(p^{m},q^{m})}=\big(\tau_1\tau_2\big)^{mi_2\,n-m{i_2+1\choose 2}}[n]_{i_2,\mathcal{R}(p^{-m},q^{-m})},
	\end{eqnarray*}
	\begin{eqnarray*}
	[y_1]_{i_1,\mathcal{R}(p^{-m},q^{-m})}[n-y_1]_{i_2,\mathcal{R}(p^{-m},q^{-m})}\genfrac{[}{]}{0pt}{}{n}{y_1}_{\mathcal{R}(p^{-m},q^{-m})}=[n]_{i_1+i_2,\mathcal{R}(p^{-m},q^{-m})}\genfrac{[}{]}{0pt}{}{n-i_1-i_2}{y_1-i_1}_{\mathcal{R}(p^{-m},q^{-m})},
	\end{eqnarray*}
	and
	\begin{eqnarray*}
	[n]_{i_1+i_2,\mathcal{R}(p^{m},q^{m})}=\big(\tau_1\tau_2\big)^{m(i_1+i_2)n-{i_1+i_2+1\choose 2}}[n]_{i_1+i_2,\mathcal{R}(p^{-m},q^{-m})}
	\end{eqnarray*}
	\end{small}
	we have
	\begin{small}
	\begin{eqnarray*}
	E\big([Y_1]_{i_1,\mathcal{R}(p^m,q^m)}[n-Y_1]_{i_2,\mathcal{R}(p^{m},q^{m})}\big)&=&\frac{[n]_{i_1+i_2,\mathcal{R}(p^{m},q^{m})}[\beta_1]_{i_1,\mathcal{R}(p^{-m},q^{-m})}[n-\beta_1]_{i_2,\mathcal{R}(p^{-m},q^{-m})}}{\tau^{mi_2\alpha_1}_2[\beta]_{i_1+i_2,\mathcal{R}(p^{-m},q^{-m})}}\nonumber\\&\times&\sum_{y_1=0}^{n-i_2}[n-y_1]_{i_2,\mathcal{R}(p^{m},q^{m})}\Psi_2(p,q)\genfrac{[}{]}{0pt}{}{n-i_2-i_1}{y_1-i_1}_{\mathcal{R}(p^{-m},q^{-m})}\nonumber\\&\times&\frac{[\beta_1-i_1]_{y_1-i_1,\mathcal{R}(p^{-m},q^{-m})}[\beta-\beta_1-i_2]_{n-i_2-y_1,\mathcal{R}(p^{-m},q^{-m})}}{[\beta-i_2-i_1]_{n-i_2-i_1,\mathcal{R}(p^{-m},q^{-m})}}.
	\end{eqnarray*}
	From the $\mathcal{R}(p,q)$- Vandermonde formula \cite{HMD},\cite{HMRC}, we get:
	\begin{eqnarray*}
	E\big([Y_1]_{i_1,\mathcal{R}(p^m,q^m)}[n-Y_1]_{i_2,\mathcal{R}(p^{m},q^{m})}\big)=\frac{[n]_{i_1+i_2,\mathcal{R}(p^{m},q^{m})}[\beta_1]_{i_1,\mathcal{R}(p^{-m},q^{-m})}[n-\beta_1]_{i_2,\mathcal{R}(p^{-m},q^{-m})}}{\tau^{mi_2\alpha_1}_2[\beta]_{i_1+i_2,\mathcal{R}(p^{-m},q^{-m})}}.
	\end{eqnarray*}
	Thus, 
\begin{eqnarray*}
	E\big([Y_1]_{i_1,\mathcal{R}(p^m,q^m)}[Y_2]_{i_2,\mathcal{R}(p^m,q^m)}\big)
	&=&\frac{[\beta_2]_{i_2,\mathcal{R}(p^{-m},q^{-m})}}{[\beta-\beta_1]_{i_1,\mathcal{R}(p^{-m},q^{-m})}}\nonumber\\&\times&E\big([Y_1]_{i_1,\mathcal{R}(p^m,q^m)}[n-Y_1]_{i_2,\mathcal{R}(p^{m},q^{m})}\big)\nonumber\\
	&=&\frac{[n]_{i_1+i_2,\mathcal{R}(p^{m},q^{m})}[\beta_1]_{i_1,\mathcal{R}(p^{-m},q^{-m})}[\beta_2]_{i_2,\mathcal{R}(p^{-m},q^{-m})}}{\tau^{mi_2\alpha_1}_2[\beta]_{i_1+i_2,\mathcal{R}(p^{-m},q^{-m})}}.
	\end{eqnarray*}	
	and the relation \eqref{bp5} holds.
	\end{small}
	$\cqfd$
\end{proof}
\begin{corollary}
 The $\mathcal{R}(p,q)$- covariance of $[Y_1]_{\mathcal{R}(p^m,q^m)}$ and $[Y_2]_{\mathcal{R}(p^m,q^m)}$ is given by the relation:
\begin{small}
\begin{eqnarray*}
Cov\big([Y_1]_{\mathcal{R}(p^m,q^m)},[Y_2]_{\mathcal{R}(p^m,q^m)}\big)=\frac{[n]_{\mathcal{R}(p^m,q^m)}[\beta_1]_{\mathcal{R}(p^{-m},q^{-m})}[\beta_2]_{\mathcal{R}(p^{-m},q^{-m})}}{\tau^{m\beta_1}_2[\beta]_{\mathcal{R}(p^{-m},q^{-m})}}\Delta(n,m,\beta),
\end{eqnarray*}
where 
\begin{eqnarray*}
\Delta(n,m,\beta)=\frac{[n-1]_{\mathcal{R}(p^m,q^m)}}{[\beta-1]_{\mathcal{R}(p^{-m},q^{-m})}}-\frac{[n]_{\mathcal{R}(p^m,q^m)}}{[\beta_1]_{\mathcal{R}(p^{-m},q^{-m})}}
\end{eqnarray*}
and  $\underline{Y}=\big(
Y_1,Y_2\big)$  a $\mathcal{R}(p,q)$-random vector verifying   the bivariate $\mathcal{R}(p,q)$-P\'olya  distribution, with parameters $n$, $\underline{\beta}=(\beta_1,\beta_2),$ $p$ and $q.$
\end{small}
\end{corollary}
\begin{proof}
By definition, we get:
\begin{eqnarray*}
Cov\big([Y_1]_{\mathcal{R}(p^m,q^m)},[Y_2]_{\mathcal{R}(p^m,q^m)}\big)&=&E\big([Y_1]_{\mathcal{R}(p^m,q^m)}[Y_2]_{\mathcal{R}(p^m,q^m)}\big)\nonumber\\&-&E\big([Y_1]_{\mathcal{R}(p^m,q^m)}\big)E\big([Y_2]_{\mathcal{R}(p^m,q^m)}\big).
\end{eqnarray*}
Taking $i_1=i_2=1,$ in the relations \eqref{bp2}, \eqref{bp4}, and \eqref{bp5}, we obtain:
\begin{eqnarray*}
E\big([Y_1]_{\mathcal{R}(p^m,q^m)}[Y_2]_{\mathcal{R}(p^m,q^m)}\big)&=&\frac{[n]_{2,\mathcal{R}(p^m,q^m)}[\beta_1]_{\mathcal{R}(p^{-m},q^{-m})}[\beta_2]_{\mathcal{R}(p^{-m},q^{-m})}}{\tau^{m\beta_1}_2\,[\beta]_{2,\mathcal{R}(p^{-m},q^{-m})}}\nonumber\\
&=& \frac{[n]_{\mathcal{R}(p^m,q^m)}[n-1]_{\mathcal{R}(p^m,q^m)}[\beta_1]_{\mathcal{R}(p^{-m},q^{-m})}[\beta_2]_{\mathcal{R}(p^{-m},q^{-m})}}{\tau^{m\beta_1}_2\,[\beta]_{\mathcal{R}(p^{-m},q^{-m})}\,[\beta-1]_{\mathcal{R}(p^{-m},q^{-m})}}
\end{eqnarray*}
and 
\begin{eqnarray*}
E\big([Y_1]_{\mathcal{R}(p^m,q^m)}\big)E\big([Y_2]_{\mathcal{R}(p^m,q^m)}\big)=\frac{[n]^2_{\mathcal{R}(p^m,q^m)}[\beta_1]_{\mathcal{R}(p^{-m},q^{-m})}[\beta_2]_{\mathcal{R}(p^{-m},q^{-m})}\tau^{-m\beta_1}_2}{[\beta]^2_{\mathcal{R}(p^{-m},q^{-m})}}.
\end{eqnarray*}
After computation, the proof is achieved.
$\cqfd$
\end{proof}
\begin{remark}
The multivariate P\'olya probability distribution and properties generated by quantum algebras are presented in the sequel:
\begin{enumerate}
	\item[(a)]The mass function of the  multivariate  $q$-P\'olya  distribution, with parameters $n$, $(\beta_1,\beta_2,\ldots,\beta_k)$, $\beta$,  and $q$, is presented as follows:
	\begin{small}
		\begin{eqnarray*}
		P(\underline{Y}=\underline{y})=\Psi_k(q)\genfrac{[}{]}{0pt}{}{n}{y_1,\ldots,y_k}_{q^{-m}}\frac{\prod_{j=1}^{k+1}[\beta_j]_{y_j,q^{-m}}}{ [\beta]_{n,q^{-m}}},
		\end{eqnarray*}
	or
	\begin{eqnarray*}
	P(Y_1=y_1,\ldots,Y_k=y_k)=\Psi_k(q)
	\frac{\prod_{j=1}^{k+1} \genfrac{[}{]}{0pt}{}{\beta_j}{y_j}_{q^{-m}}}{\genfrac{[}{]}{0pt}{}{\beta}{n}_{q^{-m}}},
	\end{eqnarray*}
	\end{small}
	where $\Psi_k(q)=q^{-m\sum_{j=1}^{k}y_{j}(\beta_{j+1}-y_{j+1})}q^{m\sum_{j=1}^{k}(n-x_j)(\beta_j-y_j)},$
	$y_j\in\{0,1,\ldots,n\}$, $j\in\{1,2,\ldots,k\},$  $\sum_{j=1}^k y_j\leq n$, $y_{k+1}=n-\sum_{j=1}^k y_j$, $\beta_{k+1}=\beta-\sum_{j=1}^k \beta_j$, and $x_j=\sum_{i=1}^{j}y_i.$ Taking $k=2,$
	the probability distribution of the $q$- random vector $\underline{Y}=\big(Y_1,Y_2\big)$ is called the bivariate $q$- P\'olya distribution with parameters $n$, $\underline{\beta}=(\beta_1,\beta_2),$ and $q.$ Also, it's probability function is given by the following relation:
	\begin{small}
		\begin{eqnarray*}
			P(Y_1=y_1,Y_2=y_2)=\Psi_2(q)\genfrac{[}{]}{0pt}{}{n}{y_1,y_2}_{q^{-m}}\frac{\prod_{j=1}^{3}[\beta_j]_{y_j,q^{-m}}}{ [\beta]_{n,q^{-m}}},
		\end{eqnarray*}
	\end{small}
	where $\Psi_2(q)=q^{-m\sum_{j=1}^{2}y_{j}(\beta_{j+1}-y_{j+1})}q^{m\sum_{j=1}^{2}(n-x_j)(\beta_j-y_j)},$
	$y_j\in\{0,1,\ldots,n\}$, $j\in\{1,2\},$  $y_1+y_2\leq n$, $y_{3}=n-y_1+y_2$, $\beta_{3}=\beta-\beta_1-\beta_2$, $x_1=y_1,$
	and $x_2=y_1+y_2.$
	Besides, its 
	$q$-factorial moments  are given by:
	\begin{small}
		\begin{eqnarray*}
			E\big([Y_1]_{i_1,q^m}\big)=\frac{[n]_{i_1,q^m}[\beta_1]_{i_1,q^{-m}}}{[\beta]_{i_1,q^{-m}}},\,i_1\in\{0,1,\cdots,n\},
		\end{eqnarray*}
		\begin{eqnarray*}
			E\big([Y_2]_{i_1,q^m}|Y_1=y_1\big)=\frac{[n-y_1]_{i_2,q^m}[\beta_2]_{i_2,q^{-m}}}{[\beta-\beta_1]_{i_2,q^{-m}}},\,i_2\in\{0,1,\cdots,n-y_1\},
		\end{eqnarray*}
		\begin{eqnarray*}
			E\big([Y_2]_{i_2,q^m}\big)=\frac{[n]_{i_2,q^m}[\beta_2]_{i_2,q^{-m}}q^{-mi_2\beta_1}}{[\beta]_{i_2,q^{-m}}},\,\,i_2\in\{0,1,\cdots,n\},
		\end{eqnarray*}
		and
		\begin{eqnarray*}
			E\big([Y_1]_{i_1,q^m}[Y_2]_{i_2,q^m}\big)=\frac{[n]_{i_1+i_2,q^m}[\beta_1]_{i_1,q^{-m}}[\beta_2]_{i_2,q^{-m}}}{q^{-mi_2\beta_1}\,[\beta]_{i_1+i_2,q^{-m})}},
		\end{eqnarray*}
		where $i_1\in\{0,1,\cdots,n-i_2\}$ and  $i_2\in\{0,1,\cdots,n\}.$
	\end{small}
	Moreover,  the $q$- covariance of $[Y_1]_{q^m}$ and $[Y_2]_{q^m}$ is given by:
	\begin{small}
		\begin{eqnarray*}
			Cov\big([Y_1]_{q^m},[Y_2]_{q^m}\big)=\frac{[n]_{q^m}[\beta_1]_{q^{-m}}[\beta_2]_{q^{-m}}}{q^{-m\beta_1}[\beta]_{q^{-m}}}\Delta(n,m,\beta),
		\end{eqnarray*}
		where 
		\begin{eqnarray*}
			\Delta(n,m,\beta)=\frac{[n-1]_{q^m}}{[\beta-1]_{q^{-m}}}-\frac{[n]_{q^m}}{[\beta_1]_{q^{-m}}}.
		\end{eqnarray*}
	\end{small}
\item[(b)]The mass function of the  multivariate  $(p,q)$-P\'olya  distribution, with parameters $n$, $(\beta_1,\beta_2,\ldots,\beta_k)$, $\beta$, $p,$  and $q$, is presented as follows:
\begin{small}
	\begin{eqnarray*}
		P(\underline{Y}=\underline{y})=\Psi_k(p,q)\genfrac{[}{]}{0pt}{}{n}{y_1,\ldots,y_k}_{p^{-m},q^{-m}}\frac{\prod_{j=1}^{k+1}[\beta_j]_{y_j,p^{-m},q^{-m}}}{ [\beta]_{n,p^{-m},q^{-m}}},
	\end{eqnarray*}
	or
	\begin{eqnarray*}
		P(Y_1=y_1,\ldots,Y_k=y_k)=\Psi_k(p,q)
		\frac{\prod_{j=1}^{k+1} \genfrac{[}{]}{0pt}{}{\beta_j}{y_j}_{p^{-m},q^{-m}}}{\genfrac{[}{]}{0pt}{}{\beta}{n}_{p^{-m},q^{-m}}},
	\end{eqnarray*}
\end{small}
where $\Psi_k(p,q)=p^{-m\sum_{j=1}^{k}y_{j}(\beta_{j+1}-y_{j+1})}\,q^{-m\sum_{j=1}^{k}(n-x_j)(\beta_j-y_j)},$
$y_j\in\{0,1,\ldots,n\}$, $j\in\{1,2,\ldots,k\},$  $\sum_{j=1}^k y_j\leq n$, $y_{k+1}=n-\sum_{j=1}^k y_j$, $\beta_{k+1}=\beta-\sum_{j=1}^k \beta_j$, and $x_j=\sum_{i=1}^{j}y_i.$

The probability distribution of the $(p,q)$- random vector $\underline{Y}=\big(Y_1,Y_2\big)$ is called the bivariate $(p,q)$- P\'olya distribution with parameters $n$, $\underline{\beta}=(\beta_1,\beta_2),$ $p$ and $q.$ Also, it's probability function is given by the following relation:
\begin{small}
	\begin{eqnarray*}
	P(Y_1=y_1,Y_2=y_2)=\Psi_2(p,q)\genfrac{[}{]}{0pt}{}{n}{y_1,y_2}_{p^{-m},q^{-m}}\frac{\prod_{j=1}^{3}[\beta_j]_{y_j,p^{-m},q^{-m}}}{ [\beta]_{n,p^{-m},q^{-m}}},
	\end{eqnarray*}
\end{small}
where $\Psi_2(p,q)=p^{-m\sum_{j=1}^{2}y_{j}(\beta_{j+1}-y_{j+1})}q^{-m\sum_{j=1}^{2}(n-x_j)(\beta_j-y_j)},$
$y_j\in\{0,1,\ldots,n\}$, $j\in\{1,2\},$  $y_1+y_2\leq n$, $y_{3}=n-y_1+y_2$, $\beta_{3}=\beta-\beta_1-\beta_2$, $x_1=y_1,$
and $x_2=y_1+y_2.$
Besides, its 
	$(p,q)$-factorial moments  are given by:
	\begin{small}
	\begin{eqnarray*}
	E\big([Y_1]_{i_1,p^m,q^m}\big)=\frac{[n]_{i_1,p^m,q^m}[\beta_1]_{i_1,p^{-m},q^{-m}}}{[\beta]_{i_1,p^{-m},q^{-m}}},\,i_1\in\{0,1,\cdots,n\},
	\end{eqnarray*}
	\begin{eqnarray*}
	E\big([Y_2]_{i_1,p^m,q^m}|Y_1=y_1\big)=\frac{[n-y_1]_{i_2,p^m,q^m}[\beta_2]_{i_2,p^{-m},q^{-m}}}{[\beta-\beta_1]_{i_2,p^{-m},q^{-m}}},\,i_2\in\{0,1,\cdots,n-y_1\},
	\end{eqnarray*}
	\begin{eqnarray*}
	E\big([Y_2]_{i_2,p^m,q^m}\big)=\frac{[n]_{i_2,p^m,q^m}[\beta_2]_{i_2,p^{-m},q^{-m}}q^{-mi_2\beta_1}}{[\beta]_{i_2,p^{-m},q^{-m}}},\,\,i_2\in\{0,1,\cdots,n\},
	\end{eqnarray*}
	and
	\begin{eqnarray*}
	E\big([Y_1]_{i_1,p^m,q^m}[Y_2]_{i_2,p^m,q^m}\big)=\frac{[n]_{i_1+i_2,p^m,q^m}[\beta_1]_{i_1,p^{-m},q^{-m}}[\beta_2]_{i_2,p^{-m},q^{-m}}}{q^{mi_2\beta_1}\,[\beta]_{i_1+i_2,p^{-m},q^{-m})}},
	\end{eqnarray*}
	where $i_1\in\{0,1,\cdots,n-i_2\}$ and  $i_2\in\{0,1,\cdots,n\}.$
\end{small}
Moreover,  the $(p,q)$- covariance of $[Y_1]_{p^m,q^m}$ and $[Y_2]_{p^m,q^m}$ is given by:
\begin{small}
\begin{eqnarray*}
Cov\big([Y_1]_{p^m,q^m},[Y_2]_{p^m,q^m}\big)=\frac{[n]_{p^m,q^m}[\beta_1]_{p^{-m},q^{-m}}[\beta_2]_{p^{-m},q^{-m}}}{q^{m\beta_1}[\beta]_{p^{-m},q^{-m}}}\Delta(n,m,\beta),
\end{eqnarray*}
where 
\begin{eqnarray*}
\Delta(n,m,\beta)=\frac{[n-1]_{p^m,q^m}}{[\beta-1]_{p^{-m},q^{-m}}}-\frac{[n]_{p^m,q^m}}{[\beta_1]_{p^{-m},q^{-m}}}.
\end{eqnarray*}
\end{small}
\item[(c)]The mass function of the  multivariate  $(p^{-1},q)$-P\'olya  distribution, with parameters $n$, $(\beta_1,\beta_2,\ldots,\beta_k)$, $\beta$, $p,$  and $q$, is presented as follows:
\begin{small}
	\begin{eqnarray*}
		P(\underline{Y}=\underline{y})=\Psi_k(p^{-1},q)\genfrac{[}{]}{0pt}{}{n}{y_1,\ldots,y_k}_{p^{m},q^{-m}}\frac{\prod_{j=1}^{k+1}[\beta_j]_{y_j,p^{m},q^{-m}}}{ [\beta]_{n,p^{m},q^{-m}}},
	\end{eqnarray*}
	or
	\begin{eqnarray*}
		P(Y_1=y_1,\ldots,Y_k=y_k)=\Psi_k(p^{-1},q)
		\frac{\prod_{j=1}^{k+1} \genfrac{[}{]}{0pt}{}{\beta_j}{y_j}_{p^{m},q^{-m}}}{\genfrac{[}{]}{0pt}{}{\beta}{n}_{p^{m},q^{-m}}},
	\end{eqnarray*}
\end{small}
where $\Psi_k(p^{-1},q)=p^{m\sum_{j=1}^{k}y_{j}(\beta_{j+1}-y_{j+1})}\,q^{-m\sum_{j=1}^{k}(n-x_j)(\beta_j-y_j)},$
$y_j\in\{0,1,\ldots,n\}$, $j\in\{1,2,\ldots,k\},$  $\sum_{j=1}^k y_j\leq n$, $y_{k+1}=n-\sum_{j=1}^k y_j$, $\beta_{k+1}=\beta-\sum_{j=1}^k \beta_j$, and $x_j=\sum_{i=1}^{j}y_i.$

The probability distribution of the $(p^{-1},q)$- random vector $\underline{Y}=\big(Y_1,Y_2\big)$ is called the bivariate $(p^{-1},q)$- P\'olya distribution with parameters $n$, $\underline{\beta}=(\beta_1,\beta_2),$ $p$ and $q.$ Also, it's probability function is given by the following relation:
\begin{small}
	\begin{eqnarray*}
		P(Y_1=y_1,Y_2=y_2)=\Psi_2(p^{-1},q)\genfrac{[}{]}{0pt}{}{n}{y_1,y_2}_{p^{-m},q^{-m}}\frac{\prod_{j=1}^{3}[\beta_j]_{y_j,p^{m},q^{-m}}}{ [\beta]_{n,p^{m},q^{-m}}},
	\end{eqnarray*}
\end{small}
where $\Psi_2(p^{-1},q)=p^{m\sum_{j=1}^{2}y_{j}(\beta_{j+1}-y_{j+1})}q^{-m\sum_{j=1}^{2}(n-x_j)(\beta_j-y_j)},$
$y_j\in\{0,1,\ldots,n\}$, $j\in\{1,2\},$  $y_1+y_2\leq n$, $y_{3}=n-y_1+y_2$, $\beta_{3}=\beta-\beta_1-\beta_2$, $x_1=y_1,$
and $x_2=y_1+y_2.$
Besides, its 
$(p^{-1},q)$-factorial moments  are given by:
\begin{small}
	\begin{eqnarray*}
		E\big([Y_1]_{i_1,p^{-m},q^m}\big)=\frac{[n]_{i_1,p^{-m},q^m}[\beta_1]_{i_1,p^{m},q^{-m}}}{[\beta]_{i_1,p^{m},q^{-m}}},\,i_1\in\{0,1,\cdots,n\},
	\end{eqnarray*}
	\begin{eqnarray*}
		E\big([Y_2]_{i_1,p^{-m},q^m}|Y_1=y_1\big)=\frac{[n-y_1]_{i_2,p^{-m},q^m}[\beta_2]_{i_2,p^{m},q^{-m}}}{[\beta-\beta_1]_{i_2,p^{m},q^{-m}}},\,i_2\in\{0,1,\cdots,n-y_1\},
	\end{eqnarray*}
	\begin{eqnarray*}
		E\big([Y_2]_{i_2,p^{-m},q^m}\big)=\frac{[n]_{i_2,p^{-m},q^m}[\beta_2]_{i_2,p^{m},q^{-m}}q^{-mi_2\beta_1}}{[\beta]_{i_2,p^{m},q^{-m}}},\,\,i_2\in\{0,1,\cdots,n\},
	\end{eqnarray*}
	and
	\begin{eqnarray*}
		E\big([Y_1]_{i_1,p^{-m},q^m}[Y_2]_{i_2,p^{-m},q^m}\big)=\frac{[n]_{i_1+i_2,p^{-m},q^m}[\beta_1]_{i_1,p^{m},q^{-m}}[\beta_2]_{i_2,p^{m},q^{-m}}}{q^{mi_2\beta_1}\,[\beta]_{i_1+i_2,p^{m},q^{-m})}},
	\end{eqnarray*}
	where $i_1\in\{0,1,\cdots,n-i_2\}$ and  $i_2\in\{0,1,\cdots,n\}.$
\end{small}
Moreover,  the $(p^{-1},q)$- covariance of $[Y_1]_{p^{-m},q^m}$ and $[Y_2]_{p^{-m},q^m}$ is given by:
\begin{small}
	\begin{eqnarray*}
		Cov\big([Y_1]_{p^{-m},q^m},[Y_2]_{p^{-m},q^m}\big)=\frac{[n]_{p^{-m},q^m}[\beta_1]_{p^{m},q^{-m}}[\beta_2]_{p^{m},q^{-m}}}{q^{m\beta_1}[\beta]_{p^{m},q^{-m}}}\Delta(n,m,\beta),
	\end{eqnarray*}
	where 
	\begin{eqnarray*}
		\Delta(n,m,\beta)=\frac{[n-1]_{p^{-m},q^m}}{[\beta-1]_{p^{m},q^{-m}}}-\frac{[n]_{p^{-m},q^m}}{[\beta_1]_{p^{m},q^{-m}}}.
	\end{eqnarray*}
\end{small}
\item[(d)]The mass function of the  multivariate  Hounkonnou-Ngompe generalized $q$- Quesne P\'olya  distribution, with parameters $n$, $(\beta_1,\beta_2,\ldots,\beta_k)$, $\beta$, $p,$  and $q$, is presented as follows:
\begin{small}
	\begin{eqnarray*}
		P(\underline{Y}=\underline{y})=\Psi_k(p,q)\genfrac{[}{]}{0pt}{}{n}{y_1,\ldots,y_k}^Q_{p^{-m},q^{-m}}\frac{\prod_{j=1}^{k+1}[\beta_j]^Q_{y_j, p^{-m},q^{-m}}}{ [\beta]^Q_{n, p^{-m},q^{-m}}},
	\end{eqnarray*}
	or
	\begin{eqnarray*}
		P(Y_1=y_1,\ldots,Y_k=y_k)=\Psi_k(p,q)
		\frac{\prod_{j=1}^{k+1} \genfrac{[}{]}{0pt}{}{\beta_j}{y_j}^Q_{p^{-m},q^{-m}}}{\genfrac{[}{]}{0pt}{}{\beta}{n}^Q_{p^{-m},q^{-m}}},
	\end{eqnarray*}
\end{small}
where $\Psi_k(p,q)=p^{-m\sum_{j=1}^{k}y_{j}(\beta_{j+1}-y_{j+1})}\,q^{m\sum_{j=1}^{k}(n-x_j)(\beta_j-y_j)},$
$y_j\in\{0,1,\ldots,n\}$, $j\in\{1,2,\ldots,k\},$  $\sum_{j=1}^k y_j\leq n$, $y_{k+1}=n-\sum_{j=1}^k y_j$, $\beta_{k+1}=\beta-\sum_{j=1}^k \beta_j$, and $x_j=\sum_{i=1}^{j}y_i.$

The probability function of the bivariate   Hounkonnou-Ngompe generalized $q$- Quesne P\'olya distribution with parameters $n$, $\underline{\beta}=(\beta_1,\beta_2),$ $p$ and $q,$  is given by:
\begin{small}
	\begin{equation*}
	P(Y_1=y_1,Y_2=y_2)=\Psi_2(p,q)\genfrac{[}{]}{0pt}{}{n}{y_1,y_2}^Q_{p^{-m},q^{-m}}\frac{\prod_{j=1}^{3}[\beta_j]^Q_{y_j,p^{-m},q^{-m}}}{ [\beta]^Q_{n,p^{-m},q^{-m}}},
	\end{equation*}
\end{small}
where $\Psi_2(p,q)=p^{-m\sum_{j=1}^{2}y_{j}(\beta_{j+1}-y_{j+1})}q^{m\sum_{j=1}^{2}(n-x_j)(\beta_j-y_j)},$
$y_j\in\{0,1,\ldots,n\}$, $j\in\{1,2\},$  $y_1+y_2\leq n$, $y_{3}=n-y_1+y_2$, $\beta_{3}=\beta-\beta_1-\beta_2$, $x_1=y_1,$
and $x_2=y_1+y_2.$
Besides, its 
	related factorial moments  are:
	\begin{small}
	\begin{eqnarray*}
	E\big([Y_1]^Q_{i_1,p^m,q^m}\big)=\frac{[n]^Q_{i_1,p^m,q^m}[\beta_1]^Q_{i_1,p^{-m},q^{-m}}}{[\beta]^Q_{i_1,p^{-m},q^{-m}}},\,i_1\in\{0,1,\cdots,n\},
	\end{eqnarray*}
	\begin{eqnarray*}
	E\big([Y_2]^Q_{i_1,p^m,q^m}|Y_1=y_1\big)=\frac{[n-y_1]^Q_{i_2,p^m,q^m}[\beta_2]^Q_{i_2,p^{-m},q^{-m}}}{[\beta-\beta_1]^Q_{i_2,p^{-m},q^{-m}}},\,i_2\in\{0,1,\cdots,n-y_1\},
	\end{eqnarray*}
	\begin{eqnarray*}
	E\big([Y_2]^Q_{i_2,p^m,q^m}\big)=\frac{[n]^Q_{i_2,p^m,q^m}[\beta_2]^Q_{i_2,p^{-m},q^{-m}}q^{mi_2\beta_1}}{[\beta]^Q_{i_2,p^{-m},q^{-m}}},\,\,i_2\in\{0,1,\cdots,n\},
	\end{eqnarray*}
	and
	\begin{eqnarray*}
	E\big([Y_1]^Q_{i_1,p^m,q^m}[Y_2]^Q_{i_2,p^m,q^m}\big)=\frac{[n]^Q_{i_1+i_2,p^m,q^m}[\beta_1]^Q_{i_1,p^{-m},q^{-m}}[\beta_2]^Q_{i_2,p^{-m},q^{-m}}}{q^{-mi_2\beta_1}\,[\beta]^Q_{i_1+i_2,p^{-m},q^{-m})}},
	\end{eqnarray*}
	where $i_1\in\{0,1,\cdots,n-i_2\}$ and  $i_2\in\{0,1,\cdots,n\}.$
\end{small}
Moreover,  the  covariance of the random variables  $[Y_1]^Q_{p^m,q^m}$ and $[Y_2]^Q_{p^m,q^m}$ is deduced as :
\begin{small}
\begin{eqnarray*}
Cov\big([Y_1]^Q_{p^m,q^m},[Y_2]^Q_{p^m,q^m}\big)=\frac{[n]^Q_{p^m,q^m}[\beta_1]^Q_{p^{-m},q^{-m}}[\beta_2]^Q_{p^{-m},q^{-m}}}{q^{-m\beta_1}[\beta]^Q_{p^{-m},q^{-m}}}\Delta^Q(n,m,\beta),
\end{eqnarray*}
where 
\begin{eqnarray*}
\Delta^Q(n,m,\beta)=\frac{[n-1]^Q_{p^m,q^m}}{[\beta-1]^Q_{p^{-m},q^{-m}}}-\frac{[n]^Q_{p^m,q^m}}{[\beta_1]^Q_{p^{-m},q^{-m}}}.
\end{eqnarray*}
\end{small}
\end{enumerate}
\end{remark}
\subsection{Multivariate inverse $\mathcal{R}(p,q)$- P\'olya distribution}\label{sec4}
Let $W_\nu$ be the number of balls of color $c_\nu$ drawn until the $n$th ball of color $c_{k+1}$ is drawn in a multiple $\mathcal{R}(p,q)$-P\'{o}lya urn model, with the conditional probability of drawing a ball of color $c_\nu$ at the $i$th $\mathcal{R}(p,q)$-drawing, given that $j_\nu-1$ balls of color $c_\nu$ and a total of $i_{\nu-1}$ balls of colors $c_1,c_2,\ldots,c_{\nu-1}$ are drawn in the previous $i-1$ $\mathcal{R}(p,q)$-drawings, given by (\ref{eq3.1}), for $\nu\in\{1,2,\cdots,k\}.$ 
\begin{theorem}\label{thm4.1}
	The probability function of the multivariate inverse $\mathcal{R}(p,q)-$ deformed P\'{o}lya distribution, with parameters $n$, $\varTheta$ $(\vartheta_1,\vartheta_2,\cdots,\vartheta_k)$,   $p,$ and $q$, is given by:
	\begin{small}
		\begin{eqnarray}\label{eq4.1}
		P(W_1=w_1,\ldots,W_k=w_k)
		&=&F_k(p,q)\genfrac{[}{]}{0pt}{}{n+w_k-1}{w_1,w_2,\cdots, w_k}_{\mathcal{R}(p^{-m},q^{-m})}\nonumber\\&\times&
		\frac{\prod_{j=1}^k[\beta_j]_{w_j,\mathcal{R}(p^{-m},q^{-m})}[\beta_{k+1}]_{n,\mathcal{R}(p^{-m},q^{-m})}}
		{[\beta]_{n+w_k,\mathcal{R}(p^{-m},q^{-m})}},
		\end{eqnarray}
	\end{small}
	where $F_k(p,q)=\tau_2^{-m\sum_{j=1}^{k}(n+w_k-w_j)(\beta_j-w_j)}$for $w_j\in\mathbb{N}\cup\{0\},$ $j\in\{1,2,\cdots,k\},$ 
	$\beta_{k+1}=\beta-\sum_{j=1}^k\beta_j,$ and $w_j=\sum_{i=1}^jw_i,$ for $j\in\{1,2,\cdots,k\}.$
\end{theorem}
\begin{proof} The probability function of the $k$-variate inverse $\mathcal{R}(p,q)$-P\'{o}lya distribution is  connected to the probability function $k$-variate $\mathcal{R}(p,q)$-P\'{o}lya distribution. Specifically,
\[
P(W_1=w_1,W_2=w_2,\cdots,W_k=w_k)=p_{n+u_k-1}(w_1,w_2,\cdots, w_k)p_{n+u_k,n},
\]
where $p_{n+u_k-1}(w_1,w_2,\cdots, w_k)$ is the probability of drawing $w_\nu$ balls of color $c_\nu$, for all $\nu\in\{1,2,\cdots,k\},$ and $n-1$ balls of color $c_{k+1}$ in $n+w_k-1$ $\mathcal{R}(p,q)$-drawings and $p_{n+w_k,n}=q^{-m(\beta_k-w_k)}[a_{k+1}-n+1]_{\mathcal{R}(p^{-m},q^{-m})}/[a-n-w_k+1]_{\mathcal{R}(p^{-m},q^{-m})}$ is the conditional probability of drawing a ball of color $c_{k+1}$ at the $(n+w_k)$th $q$-drawing, given that $n-1$ balls of color $c_{k+1}$ and a total of $w_k$ balls of colors $c_1,c_2,\ldots,c_k$ are drawn in the previous $n+w_k-1$ $q$-drawings. Thus using (\ref{eq3.2}), expression (\ref{eq4.1}) is deduced. Note that the multivariate inverse $\mathcal{R}(p,q)$-Vandermonde formula (\ref{eq2.11}) guarantees that the probabilities (\ref{eq4.1}) sum to unity.
$\cqfd$
\end{proof}

%
\subsubsection{Bivariate inverse $\mathcal{R}(p,q)-$ P\'{o}lya distribution}
The mass function of the bivariate  inverse $\mathcal{R}(p,q)-$ deformed P\'olya distribution, with parameters $n$,  $\underline{\beta}=(\beta_1,\beta_2),$ $p$ and $q,$ is derived as :
	\begin{small}
		\begin{eqnarray}\label{bip1}
		P(W_1=w_1,W_2=w_2)
		&=&F_2(p,q)\genfrac{[}{]}{0pt}{}{n+w_2-1}{w_1,w_2}_{\mathcal{R}(p^{-m},q^{-m})}\nonumber\\&\times&
		\frac{\prod_{j=1}^2[\beta_j]_{w_j,\mathcal{R}(p^{-m},q^{-m})}[\beta_{3}]_{n,\mathcal{R}(p^{-m},q^{-m})}}
		{[\beta]_{n+w_2,\mathcal{R}(p^{-m},q^{-m})}},
		\end{eqnarray}
	\end{small}
	where $F_2(p,q)=\tau_2^{-m\sum_{j=1}^{2}(n+w_2-w_j)(\beta_j-w_j)}$for $w_j\in\mathbb{N}\cup\{0\},$ $j\in\{1,2\},$ 
	$\beta_{3}=\beta-\sum_{j=1}^2\beta_j,$ and $w_j=\sum_{i=1}^jw_i,$ for $j\in\{1,2\}.$
	\begin{theorem}
For $i_1\in\mathbb{N}\cup\{0\}$ and $i_2\in\mathbb{N}\cup\{0\},$ 	the $\mathcal{R}(p,q)$-factorial moments of the bivariate inverse $\mathcal{R}(p,q)$-P\'olya probability distribution, with parameters $n$, $\underline{\beta}=(\beta_1,\beta_2),$ $p$ and $q,$ are given by:
	\begin{small}
	\begin{eqnarray}\label{bip2}
	E\big([W_2]_{i_2,\mathcal{R}(p^m,q^m)}\big)=\frac{[n+i_2-1]_{i_2,\mathcal{R}(p^{-m},q^{-m})}[\beta_2]_{i_2,\mathcal{R}(p^{m},q^{m})}}{[\beta-\beta_1-\beta_2+i_2]_{i_2,\mathcal{R}(p^{-m},q^{-m})}},
	\end{eqnarray}
	\begin{eqnarray}\label{bip3}
	E\big([W_1]_{i_1,\mathcal{R}(p^m,q^m)}|W_2=w_2\big)=\frac{[n+w_2+i_1-1]_{i_1,\mathcal{R}(p^{-m},q^{-m})}[\beta_1]_{i_1,\mathcal{R}(p^{m},q^{m})}}{[\beta-\beta_1+i_1]_{i_1,\mathcal{R}(p^{-m},q^{-m})}},
	\end{eqnarray}
	\begin{eqnarray}\label{bip4}
	E\bigg(\frac{[W_1]_{i_1,\mathcal{R}(p^m,q^m)}}{[\beta-\beta_1-n-W_2]_{i_1,\mathcal{R}(p^{m},q^{m})}}\bigg)=\frac{[n+i_1-1]_{i_1,\mathcal{R}(p^{-m},q^{-m})}[\beta_1]_{i_1,\mathcal{R}(p^{m},q^{m})}}{\Theta(n,m,\beta)},
	\end{eqnarray}
	 and 
	\begin{eqnarray}\label{bip5}
	E\bigg(\frac{[W_1]_{i_1,\mathcal{R}(p^m,q^m)}[W_2]_{i_2,\mathcal{R}(p^m,q^m)}}{[\beta-\beta_1-n-W_2]_{i_1,\mathcal{R}(p^{m},q^{m})}}\bigg)&=&\frac{[n+i_1+i_2-1]_{i_1+i_2,\mathcal{R}(p^{-m},q^{-m})}}{\Theta(n,m,\beta)}\nonumber\\&\times& \frac{[\beta_1]_{i_1,\mathcal{R}(p^{m},q^{m})}[\beta_2]_{i_2,\mathcal{R}(p^{m},q^{m})}}{[\beta-\beta_1-\beta_2+i_2]_{i_2,\mathcal{R}(p^{-m},q^{-m})}},
	\end{eqnarray}
	where \begin{eqnarray*}
	\Theta(n,m,\beta)=[\beta-\beta_1+i_1]_{i_1,\mathcal{R}(p^{-m},q^{-m})}[\beta-\beta_1-\beta_2-n]_{i_1,\mathcal{R}(p^{m},q^{m})}.
\end{eqnarray*}	 
	\end{small}
	\end{theorem}
	\begin{proof}
The marginal distribution of the $\mathcal{R}(p,q)$- random variable is an inverse $\mathcal{R}(p,q)$-P\'olya distribution, with probability function:
	\begin{small}
		\begin{eqnarray*}
		P(W_1=w_1)
		&=&F_2(p,q)\genfrac{[}{]}{0pt}{}{n+w_2-1}{w_2}_{\mathcal{R}(p^{-m},q^{-m})}\nonumber\\&\times&
		\frac{[\beta_2]_{w_2,\mathcal{R}(p^{-m},q^{-m})}[\beta-\beta_1-\beta_{2}]_{n,\mathcal{R}(p^{-m},q^{-m})}}
		{[\beta-\beta_1]_{n+w_2,\mathcal{R}(p^{-m},q^{-m})}}.
		\end{eqnarray*}
		The $\mathcal{R}(p,q)$- factorial moments $E\big([W_2]_{i_2,\mathcal{R}(p,q)}\big),\,i_2\in\mathbb{N}\cup\{0\}$ are given by:
		\begin{eqnarray}\label{bip6}
		E\big([W_2]_{i_2,\mathcal{R}(p^{m},q^{m})}\big)&=&\sum_{w_2=i_2}^{\infty}[w_2]_{i_2,\mathcal{R}(p^{m},q^{m})}F_2(p,q)\genfrac{[}{]}{0pt}{}{n+w_2-1}{w_2}_{\mathcal{R}(p^{-m},q^{-m})}\nonumber\\&\times&
		\frac{[\beta_2]_{w_2,\mathcal{R}(p^{-m},q^{-m})}[\beta-\beta_1-\beta_{2}]_{n,\mathcal{R}(p^{-m},q^{-m})}}
		{[\beta-\beta_1]_{n+w_2,\mathcal{R}(p^{-m},q^{-m})}}.
		\end{eqnarray}
		Using the relations
		\begin{eqnarray*}
		[w_2]_{i_2,\mathcal{R}(p^{-m},q^{-m})}\genfrac{[}{]}{0pt}{}{n+w_2-1}{w_2}_{\mathcal{R}(p^{-m},q^{-m})}=[n+i_2-1]_{i_2,\mathcal{R}(p^{-m},q^{-m})}\genfrac{[}{]}{0pt}{}{n+w_2-1}{w_2-i_2}_{\mathcal{R}(p^{-m},q^{-m})},
		\end{eqnarray*}
		\begin{eqnarray*}
		[w_2]_{i_2,\mathcal{R}(p^{m},q^{m})}=\big(\tau_1\tau_2\big)^{mi_2\,w_2-m{i_2+1\choose 2}}\,[w_2]_{i_2,\mathcal{R}(p^{-m},q^{-m})},
		\end{eqnarray*}
		\begin{eqnarray*}
		[\beta_2]_{i_2,\mathcal{R}(p^{m},q^{m})}=\big(\tau_1\tau_2\big)^{mi_2\,\beta_2-m{i_2+1\choose 2}}\,[\beta_2]_{i_2,\mathcal{R}(p^{-m},q^{-m})},
		\end{eqnarray*}
		and
		\begin{eqnarray*}
		-mn(\beta_2-w_2)+mi_2\,w_2-mi_2\,\beta_2=-m(n+i_2)(\beta_2-w_2),
		\end{eqnarray*}
		the relation \eqref{bip6} can be rewritten as:
		\begin{eqnarray*}
		E\big([W_2]_{i_2,\mathcal{R}(p^{m},q^{m})}\big)&=&[n+i_2-1]_{i_2,\mathcal{R}(p^{-m},q^{-m})}[\beta_2]_{i_2,\mathcal{R}(p^{m},q^{m})}[\beta-\beta_1-\beta_2]_{n,\mathcal{R}(p^{-m},q^{-m})}\nonumber\\&\times&\sum_{w_2=i_2}^{\infty}F_2(p,q)\genfrac{[}{]}{0pt}{}{n+w_2-1}{w_2-i_2}_{\mathcal{R}(p^{-m},q^{-m})}
		\frac{[\beta_2-i_2]_{w_2-i_2,\mathcal{R}(p^{-m},q^{-m})}}
		{[\beta-\beta_1]_{n+w_2,\mathcal{R}(p^{-m},q^{-m})}}.
		\end{eqnarray*}
		From the negative $\mathcal{R}(p,q)$- Vandermonde formula \cite{HMD},\cite{HMRC}, we have 
		\begin{eqnarray*}
		E\big([W_2]_{i_2,\mathcal{R}(p^{m},q^{m})}\big)=\frac{[n+i_2-1]_{i_2,\mathcal{R}(p^{-m},q^{-m})}[\beta_2]_{i_2,\mathcal{R}(p^{m},q^{m})}[\beta-\beta_1-\beta_2]_{n,\mathcal{R}(p^{-m},q^{-m})}}{[\beta-\beta_1-\beta_2+i_2]_{n+i_2,\mathcal{R}(p^{-m},q^{-m})}}
		\end{eqnarray*}
		and the relation \eqref{bip2} follows by using 
		\begin{eqnarray*}
		[\beta-\beta_1-\beta_2+i_2]_{n+i_2,\mathcal{R}(p^{-m},q^{-m})}=[\beta-\beta_1-\beta_2+i_2]_{i_2,\mathcal{R}(p^{-m},q^{-m})}[\beta-\beta_1-\beta_2]_{n,\mathcal{R}(p^{-m},q^{-m})}.
		\end{eqnarray*}
	\end{small}
	Furthermore, the conditional distribution $W_1,$ given that $W_2=w_2,$ is an inverse $\mathcal{R}(p,q)$-P\'olya probability distribution, with mass function:
	\begin{small}
		\begin{eqnarray*}
		P(W_1=w_1|W_1=w_1)
		&=&F_2(p,q)\genfrac{[}{]}{0pt}{}{n+w_2+w_1-1}{w_1}_{\mathcal{R}(p^{-m},q^{-m})}\nonumber\\&\times&
		\frac{[\beta_1]_{w_1,\mathcal{R}(p^{-m},q^{-m})}[\beta-\beta_1]_{n+w_2,\mathcal{R}(p^{-m},q^{-m})}}
		{[\beta]_{n+w_2+w_1,\mathcal{R}(p^{-m},q^{-m})}}.
		\end{eqnarray*}
		From the same procedure using to obtain \eqref{bip2}, the relation \eqref{bip3} holds.  Moreover, the expected value of the $\mathcal{R}(p,q)$-function of $\underline{W}=\big(W_1,W_2\big)$ 
		\begin{eqnarray*}
		\frac{[W_1]_{i_1,\mathcal{R}(p^{m},q^{m})}}{[\beta-\beta_1-n-W_2]_{i_1,\mathcal{R}(p^{m},q^{m})}}
		\end{eqnarray*}
		can be computed by using
		\begin{eqnarray*}
		 E\bigg(\frac{[W_1]_{i_1,\mathcal{R}(p^{m},q^{m})}}{[\beta-\beta_1-n-W_2]_{i_1,\mathcal{R}(p^{m},q^{m})}}\bigg)=E\bigg(E\bigg(\frac{[W_1]_{i_1,\mathcal{R}(p^{m},q^{m})}}{[\beta-\beta_1-n-W_2]_{i_1,\mathcal{R}(p^{m},q^{m})}}|W_2\bigg)\bigg).
		\end{eqnarray*}
		Since 
		\begin{eqnarray*}
		E\bigg(\frac{[n+i_1+W_2-1]_{i_1,\mathcal{R}(p^{-m},q^{-m})}}{[\beta-\beta_1-n-W_2]_{i_1,\mathcal{R}(p^{m},q^{m})}}\bigg)&=& \sum_{w_2=0}^{\infty}\frac{[n+i_1+w_2-1]_{i_1,\mathcal{R}(p^{-m},q^{-m})}}{[\beta-\beta_1-n-w_2]_{i_1,\mathcal{R}(p^{m},q^{m})}}F_2(p,q)\nonumber\\&\times&\genfrac{[}{]}{0pt}{}{n+w_2-1}{w_2}_{\mathcal{R}(p^{-m},q^{-m})}\nonumber\\&\times&
		\frac{[\beta_2]_{w_2,\mathcal{R}(p^{-m},q^{-m})}[\beta-\beta_1-\beta_{2}]_{n,\mathcal{R}(p^{-m},q^{-m})}}
		{[\beta-\beta_1]_{n+w_2,\mathcal{R}(p^{-m},q^{-m})}}
		\end{eqnarray*}
		and from the relations
			\begin{eqnarray*}
		[n+i_1+w_2-1]_{i_1,\mathcal{R}(p^{-m},q^{-m})}\genfrac{[}{]}{0pt}{}{n+w_2-1}{w_2}_{\mathcal{R}(p^{-m},q^{-m})}&=&[n+i_1-1]_{i_1,\mathcal{R}(p^{-m},q^{-m})}\nonumber\\&\times&\genfrac{[}{]}{0pt}{}{n+i_1+w_2-1}{w_2}_{\mathcal{R}(p^{-m},q^{-m})},
		\end{eqnarray*}
		\begin{eqnarray*}
		[\beta-\beta_1-n-w_2]_{i_1,\mathcal{R}(p^{m},q^{m})}=\frac{[\beta-\beta_1-n-w_2]_{i_1,\mathcal{R}(p^{-m},q^{-m})}}{\big(\tau_1\tau_2\big)^{mi_1(\beta-\beta_1-n-w_2)-m{i_2+1\choose 2}}},
		\end{eqnarray*}
		\begin{eqnarray*}
		[\beta-\beta_1]_{n+i_1+w_2,\mathcal{R}(p^{-m},q^{-m})}=[\beta-\beta_1]_{n+w_2,\mathcal{R}(p^{-m},q^{-m})}[\beta-\beta_1-n-w_2]_{i_1,\mathcal{R}(p^{-m},q^{-m})},
		\end{eqnarray*}
		with the  negative $\mathcal{R}(p,q)$- Vandermonde formula \cite{HMD},\cite{HMRC}, we obtain:
		\begin{eqnarray*}
		E\bigg(\frac{[n+i_1+W_2-1]_{i_1,\mathcal{R}(p^{-m},q^{-m})}}{[\beta-\beta_1-n-W_2]_{i_1,\mathcal{R}(p^{m},q^{m})}}\bigg)&=& \frac{[n+i_1-1]_{i_1,\mathcal{R}(p^{-m},q^{-m})}[\beta-\beta_1-\beta_2]_{n,\mathcal{R}(p^{-m},q^{-m})}}{\big(\tau_1\tau_2\big)^{-mi_1(\beta-\beta_1-n-\beta_2)+m{i_2+1\choose 2}}}\nonumber\\&\times&\sum_{w_2=0}^{\infty}F_2(p,q)\genfrac{[}{]}{0pt}{}{n+i_2+w_2-1}{w_2}_{\mathcal{R}(p^{-m},q^{-m})}\nonumber\\&\times&
		\frac{[\beta_2]_{w_2,\mathcal{R}(p^{-m},q^{-m})}}
		{[\beta-\beta_1]_{n+i_1+w_2,\mathcal{R}(p^{-m},q^{-m})}}\nonumber\\&=&\frac{[n+i_1-1]_{i_1,\mathcal{R}(p^{-m},q^{-m})}}{\big(\tau_1\tau_2\big)^{-mi_1(\beta-\beta_1-n-\beta_2)+m{i_2+1\choose 2}}}\nonumber\\&\times&\frac{[\beta-\beta_1-\beta_2]_{n,\mathcal{R}(p^{-m},q^{-m})}}{[\beta-\beta_1-\beta_2]_{n+i_1,\mathcal{R}(p^{-m},q^{-m})}}.
		\end{eqnarray*}
		Also, according to 
		\begin{eqnarray*}
		[\beta-\beta_1-\beta_2]_{n+i_1,\mathcal{R}(p^{-m},q^{-m})}=[\beta-\beta_1-\beta_2]_{n,\mathcal{R}(p^{-m},q^{-m})}[\beta-\beta_1-\beta_2-n]_{i_1,\mathcal{R}(p^{-m},q^{-m})}
		\end{eqnarray*}
		and 
		\begin{eqnarray*}
		[\beta-\beta_1-\beta_2-n]_{i_1,\mathcal{R}(p^{m},q^{m})}=\frac{[\beta-\beta_1-\beta_2-n]_{i_1,\mathcal{R}(p^{-m},q^{-m})}}{\big(\tau_1\tau_2\big)^{mi_1(\beta-\beta_1-\beta_2-n)-m{i_2+1\choose 2}}},
		\end{eqnarray*}
		we get
		\begin{eqnarray*}
		E\bigg(\frac{[n+i_1+W_2-1]_{i_1,\mathcal{R}(p^{-m},q^{-m})}}{[\beta-\beta_1-n-W_2]_{i_1,\mathcal{R}(p^{m},q^{m})}}\bigg)= \frac{[n+i_1-1]_{i_1,\mathcal{R}(p^{-m},q^{-m})}}{[\beta-\beta_1-\beta_2-n]_{i_1,\mathcal{R}(p^{m},q^{m})}}.
		\end{eqnarray*}
		Then, 
		\begin{eqnarray*}
		E\bigg(\frac{[W_1]_{i_1,\mathcal{R}(p^{m},q^{m})}}{[\beta-\beta_1-n-W_2]_{i_1,\mathcal{R}(p^{m},q^{m})}}\bigg)&=&\frac{[\beta_1]_{i_1,\mathcal{R}(p^{m},q^{m})}}{[\beta-\beta_1+i_1]_{i_1,\mathcal{R}(p^{-m},q^{-m})}}\nonumber\\&\times&E\bigg(\frac{[n+i_1+W_2-1]_{i_1,\mathcal{R}(p^{-m},q^{-m})}}{[\beta-\beta_1-n-W_2]_{i_1,\mathcal{R}(p^{m},q^{m})}}\bigg)\nonumber\\
		&=&  \frac{[n+i_1-1]_{i_1,\mathcal{R}(p^{-m},q^{-m})}[\beta_1]_{i_1,\mathcal{R}(p^{m},q^{m})}}{[\beta-\beta_1-\beta_2-n]_{i_1,\mathcal{R}(p^{m},q^{m})}[\beta-\beta_1+i_1]_{i_1,\mathcal{R}(p^{-m},q^{-m})}}.
		\end{eqnarray*}
		\end{small}
		In addition, we compute the mean of the $\mathcal{R}(p,q)$- function
		\begin{eqnarray*}
		\frac{[W_1]_{i_1,\mathcal{R}(p^{m},q^{m})}[W_2]_{i_2,\mathcal{R}(p^{m},q^{m})}}{[\beta-\beta_1-n-W_2]_{i_1,\mathcal{R}(p^{m},q^{m})}}.
\end{eqnarray*}	
For instance, we use the relation
\begin{small}
\begin{eqnarray*}
	  E\bigg(\frac{[W_1]_{i_1,\mathcal{R}(p^{m},q^{m})}[W_2]_{i_2,\mathcal{R}(p^{m},q^{m})}}{[\beta-\beta_1-n-W_2]_{i_1,\mathcal{R}(p^{m},q^{m})}}\bigg)&=&E\bigg(E\bigg(\frac{[W_1]_{i_1,\mathcal{R}(p^{m},q^{m})}[W_2]_{i_2,\mathcal{R}(p^{m},q^{m})}}{[\beta-\beta_1-n-W_2]_{i_1,\mathcal{R}(p^{m},q^{m})}}|W_2\bigg)\bigg)\nonumber\\
	  &=&\frac{[\beta_1]_{i_1,\mathcal{R}(p^{m},q^{m})}}{[\beta-\beta_1+i_1]_{i_1,\mathcal{R}(p^{-m},q^{-m})}}\nonumber\\&\times&E\bigg(\frac{[W_2]_{i_2,\mathcal{R}(p^{m},q^{m})}[n+i_1+W_2-1]_{i_1,\mathcal{R}(p^{-m},q^{-m})}}{[\beta-\beta_1-n-W_2]_{i_1,\mathcal{R}(p^{m},q^{m})}}\bigg)
	  \end{eqnarray*}
	  Since
	  \begin{eqnarray*}
	  E\big(\overline{W}\big)&:=& E\bigg(\frac{[W_2]_{i_2,\mathcal{R}(p^{m},q^{m})}[n+i_1+W_2-1]_{i_1,\mathcal{R}(p^{-m},q^{-m})}}{[\beta-\beta_1-n-W_2]_{i_1,\mathcal{R}(p^{m},q^{m})}}\bigg)\nonumber\\
	  &=& \sum_{w_2=0}^{\infty}\frac{[w_2]_{i_2,\mathcal{R}(p^{m},q^{m})}[n+i_1+w_2-1]_{i_1,\mathcal{R}(p^{-m},q^{-m})}}{[\beta-\beta_1-n-w_2]_{i_1,\mathcal{R}(p^{m},q^{m})}}F_2(p,q)\nonumber\\&\times&\genfrac{[}{]}{0pt}{}{n+w_2-1}{w_2}_{\mathcal{R}(p^{-m},q^{-m})}
		\frac{[\beta_2]_{w_2,\mathcal{R}(p^{-m},q^{-m})}[\beta-\beta_1-\beta_{2}]_{n,\mathcal{R}(p^{-m},q^{-m})}}
		{[\beta-\beta_1]_{n+w_2,\mathcal{R}(p^{-m},q^{-m})}}.
	  \end{eqnarray*}
	  Using the relations
	  \begin{eqnarray*}
		[w_2]_{i_2,\mathcal{R}(p^{m},q^{m})}[n+i_1+w_2-1]_{i_1,\mathcal{R}(p^{-m},q^{-m})}\genfrac{[}{]}{0pt}{}{n+w_2-1}{w_2}_{\mathcal{R}(p^{-m},q^{-m})}&=&[n+i_1+i_2-1]_{i_1+i_2,\mathcal{R}(p^{-m},q^{-m})}\nonumber\\&\times&\genfrac{[}{]}{0pt}{}{n+i_1+w_2-1}{w_2-i_2}_{\mathcal{R}(p^{-m},q^{-m})},
		\end{eqnarray*}
		\begin{eqnarray*}
		[w_2]_{i_2,\mathcal{R}(p^{m},q^{m})}=\big(\tau_1\tau_2\big)^{mi_2\,w_2-m{i_2+1\choose 2}}\,[w_2]_{i_2,\mathcal{R}(p^{-m},q^{-m})},
		\end{eqnarray*}
		\begin{eqnarray*}
		[\beta_2]_{i_2,\mathcal{R}(p^{m},q^{m})}=\big(\tau_1\tau_2\big)^{mi_2\,\beta_2-m{i_2+1\choose 2}}\,[\beta_2]_{i_2,\mathcal{R}(p^{-m},q^{-m})},
		\end{eqnarray*}
		\begin{eqnarray*}
		[\beta-\beta_1-n-w_2]_{i_1,\mathcal{R}(p^{m},q^{m})}=\frac{[\beta-\beta_1-n-w_2]_{i_1,\mathcal{R}(p^{-m},q^{-m})}}{\big(\tau_1\tau_2\big)^{mi_1(\beta-\beta_1-n-\beta_2)-m{i_2+1\choose 2}}},
		\end{eqnarray*}
		\begin{eqnarray*}
		[\beta-\beta_1]_{n+i_1+w_2,\mathcal{R}(p^{-m},q^{-m})}=[\beta-\beta_1]_{n+w_2,\mathcal{R}(p^{-m},q^{-m})}[\beta-\beta_1-n-w_2]_{i_1,\mathcal{R}(p^{-m},q^{-m})},
		\end{eqnarray*}
		with the  negative $\mathcal{R}(p,q)$- Vandermonde formula \cite{HMD},\cite{HMRC}, we obtain:
		\begin{eqnarray*}
	E\big(\overline{W}\big)&=& \frac{[n+i_1+i_2-1]_{i_1+i_2,\mathcal{R}(p^{-m},q^{-m})}[\beta-\beta_1-\beta_2]_{n,\mathcal{R}(p^{-m},q^{-m})}}{\big(\tau_1\tau_2\big)^{-mi_1(\beta-\beta_1-n-\beta_2)+m{i_2+1\choose 2}}}\nonumber\\&\times&\sum_{w_2=0}^{\infty}F_2(p,q)\genfrac{[}{]}{0pt}{}{n+i_1+w_2-1}{w_2-i_2}_{\mathcal{R}(p^{-m},q^{-m})}\nonumber\\&\times&
		\frac{[\beta_2-i_2]_{w_2-i_2,\mathcal{R}(p^{-m},q^{-m})}[\beta_2]_{i_2,\mathcal{R}(p^{m},q^{m})}}
		{[\beta-\beta_1]_{n+i_1+w_2,\mathcal{R}(p^{-m},q^{-m})}}\nonumber\\&=&\frac{[n+i_1+i_2-1]_{i_1+i_2,\mathcal{R}(p^{-m},q^{-m})}[\beta-\beta_1-\beta_2]_{n,\mathcal{R}(p^{-m},q^{-m})}[\beta_2]_{i_2,\mathcal{R}(p^{m},q^{m})}}{\big(\tau_1\tau_2\big)^{-mi_1(\beta-\beta_1-n-\beta_2)+m{i_2+1\choose 2}}[\beta-\beta_1-\beta_2+i_2]_{n+i_1+i_2,\mathcal{R}(p^{-m},q^{-m})}}.
		\end{eqnarray*}
		Furthermore, using 
		\begin{eqnarray*}
		[\beta-\beta_1-\beta_2+i_2]_{n+i_1+i_2,\mathcal{R}(p^{-m},q^{-m})}&=&[\beta-\beta_1-\beta_2+i_2]_{i_2,\mathcal{R}(p^{-m},q^{-m})}\nonumber\\&\times&[\beta-\beta_1-\beta_2]_{n+i_1,\mathcal{R}(p^{-m},q^{-m})},
		\end{eqnarray*}
		\begin{eqnarray*}
		[\beta-\beta_1-\beta_2]_{n+i_1,\mathcal{R}(p^{-m},q^{-m})}=[\beta-\beta_1-\beta_2]_{n,\mathcal{R}(p^{-m},q^{-m})}[\beta-\beta_1-\beta_2-n]_{i_1,\mathcal{R}(p^{-m},q^{-m})}.
		\end{eqnarray*}
		and 
		\begin{eqnarray*}
		[\beta-\beta_1-\beta_2-n]_{i_1,\mathcal{R}(p^{m},q^{m})}=\frac{[\beta-\beta_1-\beta_2-n]_{i_1,\mathcal{R}(p^{-m},q^{-m})}}{\big(\tau_1\tau_2\big)^{mi_1(\beta-\beta_1-\beta_2-n)-m{i_2+1\choose 2}}},
		\end{eqnarray*}
		we obtain
		\begin{eqnarray*}
	E\big(\overline{W}\big)= \frac{[n+i_1+i_2-1]_{i_1+i_2,\mathcal{R}(p^{-m},q^{-m})}[\beta_2]_{i_2,\mathcal{R}(p^{m},q^{m})}}{[\beta-\beta_1-\beta_2-n]_{i_1,\mathcal{R}(p^{m},q^{m})}[\beta-\beta_1-\beta_2+i_2]_{i_2,\mathcal{R}(p^{-m},q^{-m})}}.
		\end{eqnarray*}
		Then, the relation \eqref{bip5} follows and the proof is achieved.
		\end{small}
	$\cqfd$
	\end{proof}
	\begin{corollary}
	The $\mathcal{R}(p,q)$- covariance of the $\mathcal{R}(p,q)$- random variable $\ddot{W}=[W_2]_{\mathcal{R}(p^{m},q^{m})}$  and $\widehat{W}=\frac{[W_1]_{\mathcal{R}(p^{m},q^{m})}}{[\beta-\beta_1-n-W_2]_{\mathcal{R}(p^{m},q^{m})}}$  is given by:
	\begin{small}
	\begin{eqnarray}\label{bip7}
	Cov\big(\widehat{W}, \ddot{W}\big)&=&\frac{[n]_{\mathcal{R}(p^{-m},q^{-m})}[\beta_2]_{\mathcal{R}(p^{m},q^{m})}[\beta_1]_{\mathcal{R}(p^{m},q^{m})}}{\nabla(n,\beta)}\nonumber\\&\times&\bigg([n+1]_{\mathcal{R}(p^{-m},q^{-m})}-[n]_{\mathcal{R}(p^{-m},q^{-m})}\bigg),
	\end{eqnarray}
	where $$\nabla(n,\beta)=[\beta-\beta_1+1]_{\mathcal{R}(p^{-m},q^{-m})}[\beta-\beta_1-\beta_2-n]_{\mathcal{R}(p^{m},q^{m})}[\beta-\beta_1-\beta_2+1]_{\mathcal{R}(p^{-m},q^{-m})}$$ and the $\mathcal{R}(p,q)$-random vector $\big(W_1,W_2\big)$ satisfy the bivariate $\mathcal{R}(p,q)$- P\'olya distribution, with parameters $n,$ $\underline{\beta}=(\beta_1,\beta_2),$ $p$ and $q.$
	\end{small}
	\end{corollary}
	\begin{proof}
	The $\mathcal{R}(p,q)$-covariance of $\widehat{W}$ and $\ddot{W},$ is defined by
	\begin{eqnarray*}
	Cov\big(\widehat{W}, \ddot{W}\big)= E\big(\widehat{W}\ddot{W}\big)-E\big(\widehat{W}\big)E\big(\ddot{W}\big).
	\end{eqnarray*}
	From the relations \eqref{bip2}, \eqref{bip4}, and \eqref{bip5}, with $i_1=i_2=1,$ we have:
	\begin{eqnarray*}
	E\big(\widehat{W}\ddot{W}\big)=\frac{[n+1]_{\mathcal{R}(p^{-m},q^{-m})}[n]_{\mathcal{R}(p^{-m},q^{-m})}[\beta_1]_{\mathcal{R}(p^{m},q^{m})}[\beta_2]_{\mathcal{R}(p^{m},q^{m})}}{\nabla(n,\beta)} 
	\end{eqnarray*}
	and 
	\begin{eqnarray*}
	E\big(\widehat{W}\big)E\big(\ddot{W}\big)=\frac{[n]_{\mathcal{R}(p^{-m},q^{-m})}[\beta_2]_{\mathcal{R}(p^{m},q^{m})}[n]_{\mathcal{R}(p^{-m},q^{-m})}[\beta_1]_{\mathcal{R}(p^{m},q^{m})}}{\nabla(n,\beta)}.
	\end{eqnarray*}
	After computation, the relation \eqref{bip7} holds and the proff is achieved.
	$\cqfd$
	\end{proof}
	\begin{remark}
	Particular cases of multivariate inverse P\'olya distribution are deduced as:
	\begin{enumerate}
	\item[(a)]
	The probability function of the multivariate inverse $q-$ deformed P\'olya distribution, with parameters $n$, $\varTheta$ $(\vartheta_1,\vartheta_2,\cdots,\vartheta_k)$, and $q$, is given by:
	\begin{small}
		\begin{eqnarray*}
		P(W_1=w_1,\ldots,W_k=w_k)
		&=&F_k(q)\genfrac{[}{]}{0pt}{}{n+w_k-1}{w_1,w_2,\cdots, w_k}_{q^{-m}}\nonumber\\&\times&
		\frac{\prod_{j=1}^k[\beta_j]_{w_j,q^{-m}}[\beta_{k+1}]_{n,q^{-m}}}
		{[\beta]_{n+w_k,q^{-m}}},
		\end{eqnarray*}
	\end{small}
	where $F_k(q)=q^{m\sum_{j=1}^{k}(n+w_k-w_j)(\beta_j-w_j)}$for $w_j\in\mathbb{N}\cup\{0\},$ $j\in\{1,2,\cdots,k\},$ 
	$\beta_{k+1}=\beta-\sum_{j=1}^k\beta_j,$ and $w_j=\sum_{i=1}^jw_i,$ for $j\in\{1,2,\cdots,k\}.$
	
	The probability function of the bivariate  inverse $q-$  P\'olya distribution, with parameters $n$,  $\underline{\beta}=(\beta_1,\beta_2),$  and $q,$ is  :
	\begin{small}
		\begin{eqnarray*}
			P(W_1=w_1,W_2=w_2)
			&=&F_2(q)\genfrac{[}{]}{0pt}{}{n+w_2-1}{w_1,w_2}_{q^{-m}}\nonumber\\&\times&
			\frac{\prod_{j=1}^2[\beta_j]_{w_j,q^{-m}}[\beta_{3}]_{n,q^{-m}}}
			{[\beta]_{n+w_2,q^{-m}}},
		\end{eqnarray*}
	\end{small}
	where $F_2(q)=q^{-m\sum_{j=1}^{2}(n+w_2-w_j)(\beta_j-w_j)}$for $w_j\in\mathbb{N}\cup\{0\},$ $j\in\{1,2\},$ 
	$\beta_{3}=\beta-\sum_{j=1}^2\beta_j,$ and $w_j=\sum_{i=1}^jw_i,$ for $j\in\{1,2\}.$
	Moreover, 
	for $i_1\in\mathbb{N}\cup\{0\}$ and $i_2\in\mathbb{N}\cup\{0\},$ its	factorial moments  are given by:
	\begin{small}
		\begin{eqnarray*}
			E\big([W_2]_{i_2,q^m}\big)=\frac{[n+i_2-1]_{i_2,q^{-m}}[\beta_2]_{i_2,q^{m}}}{[\beta-\beta_1-\beta_2+i_2]_{i_2,q^{-m}}},
		\end{eqnarray*}
		\begin{eqnarray*}
			E\big([W_1]_{i_1,q^m}|W_2=w_2\big)=\frac{[n+w_2+i_1-1]_{i_1,q^{-m}}[\beta_1]_{i_1,q^{m}}}{[\beta-\beta_1+i_1]_{i_1,q^{-m}}},
		\end{eqnarray*}
		\begin{eqnarray*}
			E\bigg(\frac{[W_1]_{i_1,q^m}}{[\beta-\beta_1-n-W_2]_{i_1,q^{m}}}\bigg)=\frac{[n+i_1-1]_{i_1,q^{-m}}[\beta_1]_{i_1,q^{m}}}{[\beta-\beta_1+i_1]_{i_1,q^{-m}}[\beta-\beta_1-\beta_2-n]_{i_1,q^{m}}},
		\end{eqnarray*}
		and 
		\begin{eqnarray*}
			E\bigg(\frac{[W_1]_{i_1,q^m}[W_2]_{i_2,q^m}}{[\beta-\beta_1-n-W_2]_{i_1,q^{m}}}\bigg)&=&\frac{[n+i_1+i_2-1]_{i_1+i_2,q^{-m}}}{[\beta-\beta_1+i_1]_{i_1,q^{-m}}[\beta-\beta_1-\beta_2-n]_{i_1,q^{m}}}\nonumber\\&\times& \frac{[\beta_1]_{i_1,q^{m}}[\beta_2]_{i_2,q^{m}}}{[\beta-\beta_1-\beta_2+i_2]_{i_2,q^{-m}}}.
		\end{eqnarray*} 
	\end{small}
	Furthermore, the $q$- covariance of the $q$- random variable $\ddot{W}=[W_2]_{q^{m}}$  and $\widehat{W}=\frac{[W_1]_{q^{m}}}{[\beta-\beta_1-n-W_2]_{q^{m}}}$  is determined by :
	\begin{small}
		\begin{eqnarray*}
			Cov\big(\widehat{W}, \ddot{W}\big)=\frac{[n]_{q^{-m}}[\beta_2]_{q^{m}}[\beta_1]_{q^{m}}}{\nabla(n,\beta)}\bigg([n+1]_{q^{-m}}-[n]_{q^{-m}}\bigg),
		\end{eqnarray*}
		
		where $$\nabla(n,\beta)=[\beta-\beta_1+1]_{q^{-m}}[\beta-\beta_1-\beta_2-n]_{q^{m}}[\beta-\beta_1-\beta_2+1]_{q^{-m}}$$the $q$-random vector $\big(W_1,W_2\big)$ satisfy the bivariate inverse $q$- P\'olya distribution, with parameters $n,$ $\underline{\beta}=(\beta_1,\beta_2),$  and $q.$
	\end{small}
	\item[(b)]The probability function of the multivariate inverse $(p,q)-$ deformed P\'{o}lya distribution, with parameters $n$, $\varTheta$ $(\vartheta_1,\vartheta_2,\cdots,\vartheta_k)$,   $p,$ and $q$, is given by:
	\begin{small}
		\begin{eqnarray*}
		P(W_1=w_1,\ldots,W_k=w_k)
		&=&F_k(p,q)\genfrac{[}{]}{0pt}{}{n+w_k-1}{w_1,w_2,\cdots, w_k}_{p^{-m},q^{-m}}\nonumber\\&\times&
		\frac{\prod_{j=1}^k[\beta_j]_{w_j,p^{-m},q^{-m}}[\beta_{k+1}]_{n,p^{-m},q^{-m}}}
		{[\beta]_{n+w_k,p^{-m},q^{-m}}},
		\end{eqnarray*}
	\end{small}
	where $F_k(p,q)=q^{-m\sum_{j=1}^{k}(n+w_k-w_j)(\beta_j-w_j)}$for $w_j\in\mathbb{N}\cup\{0\},$ $j\in\{1,2,\cdots,k\},$ 
	$\beta_{k+1}=\beta-\sum_{j=1}^k\beta_j,$ and $w_j=\sum_{i=1}^jw_i,$ for $j\in\{1,2,\cdots,k\}.$
	
	 The probability function of the bivariate  inverse $(p,q)-$  P\'{o}lya distribution, with parameters $n$,  $\underline{\beta}=(\beta_1,\beta_2),$ $p$ and $q,$ is  :
	\begin{small}
		\begin{eqnarray*}
		P(W_1=w_1,W_2=w_2)
		&=&F_2(p,q)\genfrac{[}{]}{0pt}{}{n+w_2-1}{w_1,w_2}_{p^{-m},q^{-m}}\nonumber\\&\times&
		\frac{\prod_{j=1}^2[\beta_j]_{w_j,p^{-m},q^{-m}}[\beta_{3}]_{n,p^{-m},q^{-m}}}
		{[\beta]_{n+w_2,p^{-m},q^{-m}}},
		\end{eqnarray*}
	\end{small}
	where $F_2(p,q)=q^{-m\sum_{j=1}^{2}(n+w_2-w_j)(\beta_j-w_j)}$for $w_j\in\mathbb{N}\cup\{0\},$ $j\in\{1,2\},$ 
	$\beta_{3}=\beta-\sum_{j=1}^2\beta_j,$ and $w_j=\sum_{i=1}^jw_i,$ for $j\in\{1,2\}.$
	Moreover, 
for $i_1\in\mathbb{N}\cup\{0\}$ and $i_2\in\mathbb{N}\cup\{0\},$ it's	factorial moments  are given by:
	\begin{small}
	\begin{eqnarray*}
	E\big([W_2]_{i_2,p^m,q^m}\big)=\frac{[n+i_2-1]_{i_2,p^{-m},q^{-m}}[\beta_2]_{i_2,p^{m},q^{m}}}{[\beta-\beta_1-\beta_2+i_2]_{i_2,p^{-m},q^{-m}}},
	\end{eqnarray*}
	\begin{eqnarray*}
	E\big([W_1]_{i_1,p^m,q^m}|W_2=w_2\big)=\frac{[n+w_2+i_1-1]_{i_1,p^{-m},q^{-m}}[\beta_1]_{i_1,p^{m},q^{m}}}{[\beta-\beta_1+i_1]_{i_1,p^{-m},q^{-m}}},
	\end{eqnarray*}
	\begin{eqnarray*}
	E\bigg(\frac{[W_1]_{i_1,p^m,q^m}}{[\beta-\beta_1-n-W_2]_{i_1,p^{m},q^{m}}}\bigg)=\frac{[n+i_1-1]_{i_1,p^{-m},q^{-m}}[\beta_1]_{i_1,p^{m},q^{m}}}{[\beta-\beta_1+i_1]_{i_1,p^{-m},q^{-m}}[\beta-\beta_1-\beta_2-n]_{i_1,p^{m},q^{m}}},
	\end{eqnarray*}
	 and 
	\begin{eqnarray*}
	E\bigg(\frac{[W_1]_{i_1,p^m,q^m}[W_2]_{i_2,p^m,q^m}}{[\beta-\beta_1-n-W_2]_{i_1,p^{m},q^{m}}}\bigg)&=&\frac{[n+i_1+i_2-1]_{i_1+i_2,p^{-m},q^{-m}}}{[\beta-\beta_1+i_1]_{i_1,p^{-m},q^{-m}}[\beta-\beta_1-\beta_2-n]_{i_1,p^{m},q^{m}}}\nonumber\\&\times& \frac{[\beta_1]_{i_1,p^{m},q^{m}}[\beta_2]_{i_2,p^{m},q^{m}}}{[\beta-\beta_1-\beta_2+i_2]_{i_2,p^{-m},q^{-m}}}.
	\end{eqnarray*} 
	\end{small}
	Furthermore, the $(p,q)$- covariance of the $(p,q)$- random variable $\ddot{W}=[W_2]_{p^{m},q^{m}}$  and $\widehat{W}=\frac{[W_1]_{p^{m},q^{m}}}{[\beta-\beta_1-n-W_2]_{p^{m},q^{m}}}$  is determined by :
	\begin{small}
	\begin{eqnarray*}
	Cov\big(\widehat{W}, \ddot{W}\big)=\frac{[n]_{p^{-m},q^{-m}}[\beta_2]_{p^{m},q^{m}}[\beta_1]_{p^{m},q^{m}}}{\nabla(n,\beta)}\bigg([n+1]_{p^{-m},q^{-m}}-[n]_{p^{-m},q^{-m}}\bigg),
	\end{eqnarray*}
	
	where $$\nabla(n,\beta)=[\beta-\beta_1+1]_{p^{-m},q^{-m}}[\beta-\beta_1-\beta_2-n]_{p^{m},q^{m}}[\beta-\beta_1-\beta_2+1]_{p^{-m},q^{-m}}$$the $(p,q)$-random vector $\big(W_1,W_2\big)$ satisfy the bivariate inverse $(p,q)$- P\'olya distribution, with parameters $n,$ $\underline{\beta}=(\beta_1,\beta_2),$ $p$ and $q.$
	\end{small}
\item[(c)] The probability function of the multivariate inverse $(p^{-1},q)-$ deformed P\'{o}lya distribution, with parameters $n$, $\varTheta$ $(\vartheta_1,\vartheta_2,\cdots,\vartheta_k)$,   $p,$ and $q$, is given by:
\begin{small}
	\begin{eqnarray*}
		P(W_1=w_1,\ldots,W_k=w_k)
		&=&F_k(p^{-1},q)\genfrac{[}{]}{0pt}{}{n+w_k-1}{w_1,w_2,\cdots, w_k}_{p^{m},q^{-m}}\nonumber\\&\times&
		\frac{\prod_{j=1}^k[\beta_j]_{w_j,p^{m},q^{-m}}[\beta_{k+1}]_{n,p^{m},q^{-m}}}
		{[\beta]_{n+w_k,p^{m},q^{-m}}},
	\end{eqnarray*}
\end{small}
where $F_k(p^{-1},q)=q^{-m\sum_{j=1}^{k}(n+w_k-w_j)(\beta_j-w_j)}$for $w_j\in\mathbb{N}\cup\{0\},$ $j\in\{1,2,\cdots,k\},$ 
$\beta_{k+1}=\beta-\sum_{j=1}^k\beta_j,$ and $w_j=\sum_{i=1}^jw_i,$ for $j\in\{1,2,\cdots,k\}.$

The probability function of the bivariate  inverse $(p^{-1},q)-$  P\'olya distribution, with parameters $n$,  $\underline{\beta}=(\beta_1,\beta_2),$ $p$ and $q,$ is  :
\begin{small}
	\begin{eqnarray*}
		P(W_1=w_1,W_2=w_2)
		&=&F_2(p^{-1},q)\genfrac{[}{]}{0pt}{}{n+w_2-1}{w_1,w_2}_{p^{m},q^{-m}}\nonumber\\&\times&
		\frac{\prod_{j=1}^2[\beta_j]_{w_j,p^{m},q^{-m}}[\beta_{3}]_{n,p^{m},q^{-m}}}
		{[\beta]_{n+w_2,p^{m},q^{-m}}},
	\end{eqnarray*}
\end{small}
where $F_2(p^{-1},q)=q^{-m\sum_{j=1}^{2}(n+w_2-w_j)(\beta_j-w_j)}$for $w_j\in\mathbb{N}\cup\{0\},$ $j\in\{1,2\},$ 
$\beta_{3}=\beta-\sum_{j=1}^2\beta_j,$ and $w_j=\sum_{i=1}^jw_i,$ for $j\in\{1,2\}.$
Moreover, 
for $i_1\in\mathbb{N}\cup\{0\}$ and $i_2\in\mathbb{N}\cup\{0\},$ its	factorial moments  are given by:
\begin{small}
	\begin{eqnarray*}
		E\big([W_2]_{i_2,p^{-m},q^m}\big)=\frac{[n+i_2-1]_{i_2,p^{m},q^{-m}}[\beta_2]_{i_2,p^{-m},q^{m}}}{[\beta-\beta_1-\beta_2+i_2]_{i_2,p^{m},q^{-m}}},
	\end{eqnarray*}
	\begin{eqnarray*}
		E\big([W_1]_{i_1,p^{-m},q^m}|W_2=w_2\big)=\frac{[n+w_2+i_1-1]_{i_1,p^{m},q^{-m}}[\beta_1]_{i_1,p^{-m},q^{m}}}{[\beta-\beta_1+i_1]_{i_1,p^{m},q^{-m}}},
	\end{eqnarray*}
	\begin{eqnarray*}
		E\bigg(\frac{[W_1]_{i_1,p^{-m},q^m}}{[\beta-\beta_1-n-W_2]_{i_1,p^{-m},q^{m}}}\bigg)=\frac{[n+i_1-1]_{i_1,p^{m},q^{-m}}[\beta_1]_{i_1,p^{-m},q^{m}}}{[\beta-\beta_1+i_1]_{i_1,p^{m},q^{-m}}[\beta-\beta_1-\beta_2-n]_{i_1,p^{-m},q^{m}}},
	\end{eqnarray*}
	and 
	\begin{eqnarray*}
		E\bigg(\frac{[W_1]_{i_1,p^{-m},q^m}[W_2]_{i_2,p^{-m},q^m}}{[\beta-\beta_1-n-W_2]_{i_1,p^{-m},q^{m}}}\bigg)&=&\frac{[n+i_1+i_2-1]_{i_1+i_2,p^{m},q^{-m}}}{[\beta-\beta_1+i_1]_{i_1,p^{m},q^{-m}}[\beta-\beta_1-\beta_2-n]_{i_1,p^{-m},q^{m}}}\nonumber\\&\times& \frac{[\beta_1]_{i_1,p^{-m},q^{m}}[\beta_2]_{i_2,p^{-m},q^{m}}}{[\beta-\beta_1-\beta_2+i_2]_{i_2,p^{m},q^{-m}}}.
	\end{eqnarray*} 
\end{small}
Furthermore, the $(p^{-1},q)$- covariance of the $(p^{-1},q)$- random variable $\ddot{W}=[W_2]_{p^{-m},q^{m}}$  and $\widehat{W}=\frac{[W_1]_{p^{-m},q^{m}}}{[\beta-\beta_1-n-W_2]_{p^{-m},q^{m}}}$  is determined by :
\begin{small}
	\begin{eqnarray*}
		Cov\big(\widehat{W}, \ddot{W}\big)=\frac{[n]_{p^{m},q^{-m}}[\beta_2]_{p^{-m},q^{m}}[\beta_1]_{p^{-m},q^{m}}}{\nabla(n,\beta)}\bigg([n+1]_{p^{m},q^{-m}}-[n]_{p^{m},q^{-m}}\bigg),
	\end{eqnarray*}
	
	where $$\nabla(n,\beta)=[\beta-\beta_1+1]_{p^{m},q^{-m}}[\beta-\beta_1-\beta_2-n]_{p^{-m},q^{m}}[\beta-\beta_1-\beta_2+1]_{p^{m},q^{-m}}$$the $(p^{-1},q)$-random vector $\big(W_1,W_2\big)$ satisfy the bivariate inverse $(p^{-1},q)$- P\'olya distribution, with parameters $n,$ $\underline{\beta}=(\beta_1,\beta_2),$ $p$ and $q.$
\end{small}
	\item[(d)]
	The probability function of the multivariate inverse Hounkonnou-Ngompe generalized $q$- Quesne P\'{o}lya distribution, with parameters $n$, $\varTheta$ $(\vartheta_1,\vartheta_2,\cdots,\vartheta_k)$,   $p,$ and $q$, is given by:
	\begin{small}
		\begin{eqnarray*}
			P(W_1=w_1,\ldots,W_k=w_k)
			&=&F_k(p,q)\genfrac{[}{]}{0pt}{}{n+w_k-1}{w_1,w_2,\cdots, w_k}^Q_{p^{-m},q^{-m}}\nonumber\\&\times&
			\frac{\prod_{j=1}^k[\beta_j]^Q_{w_j,p^{-m},q^{-m}}[\beta_{k+1}]^Q_{n,p^{-m},q^{-m}}}
			{[\beta]^Q_{n+w_k,p^{-m},q^{-m}}},
		\end{eqnarray*}
	\end{small}
	where $F_k(p,q)=q^{m\sum_{j=1}^{k}(n+w_k-w_j)(\beta_j-w_j)}$for $w_j\in\mathbb{N}\cup\{0\},$ $j\in\{1,2,\cdots,k\},$ 
	$\beta_{k+1}=\beta-\sum_{j=1}^k\beta_j,$ and $w_j=\sum_{i=1}^jw_i,$ for $j\in\{1,2,\cdots,k\}.$
	
	The mass function of the bivariate  inverse  Hounkonnou-Ngompe generalized $q$- Quesne   P\'{o}lya distribution, with parameters $n$,  $\underline{\beta}=(\beta_1,\beta_2),$ $p$ and $q,$ is  :
	\begin{small}
		\begin{eqnarray*}
		P(\underline{W}=\underline{w})
		=F_2(p,q)\genfrac{[}{]}{0pt}{}{n+w_2-1}{w_1,w_2}^Q_{p^{-m},q^{-m}}
		\frac{\prod_{j=1}^2[\beta_j]^Q_{w_j,p^{-m},q^{-m}}[\beta_{3}]^Q_{n,p^{-m},q^{-m}}}
		{[\beta]^Q_{n+w_2,p^{-m},q^{-m}}},
		\end{eqnarray*}
	\end{small}
	where $F_2(p,q)=q^{-m\sum_{j=1}^{2}(n+w_2-w_j)(\beta_j-w_j)}$for $w_j\in\mathbb{N}\cup\{0\},$ $j\in\{1,2\},$ 
	$\beta_{3}=\beta-\sum_{j=1}^2\beta_j,$ and $w_j=\sum_{i=1}^jw_i,$ for $j\in\{1,2\}.$
	Moreover, 
for $i_1\in\mathbb{N}\cup\{0\}$ and $i_2\in\mathbb{N}\cup\{0\},$ 	the corresponding factorial moments  are given by:
	\begin{small}
	\begin{eqnarray*}
	E\big([W_2]^Q_{i_2,p^m,q^m}\big)=\frac{[n+i_2-1]^Q_{i_2,p^{-m},q^{-m}}[\beta_2]^Q_{i_2,p^{m},q^{m}}}{[\beta-\beta_1-\beta_2+i_2]^Q_{i_2,p^{-m},q^{-m}}},
	\end{eqnarray*}
	\begin{eqnarray*}
	E\big([W_1]^Q_{i_1,p^m,q^m}|W_2=w_2\big)=\frac{[n+w_2+i_1-1]^Q_{i_1,p^{-m},q^{-m}}[\beta_1]^Q_{i_1,p^{m},q^{m}}}{[\beta-\beta_1+i_1]^Q_{i_1,p^{-m},q^{-m}}},
	\end{eqnarray*}
	\begin{eqnarray*}
	E\bigg(\frac{[W_1]^Q_{i_1,p^m,q^m}}{[\beta-\beta_1-n-W_2]^Q_{i_1,p^{m},q^{m}}}\bigg)=\frac{[n+i_1-1]^Q_{i_1,p^{-m},q^{-m}}[\beta_1]^Q_{i_1,p^{m},q^{m}}}{[\beta-\beta_1+i_1]^Q_{i_1,p^{-m},q^{-m}}[\beta-\beta_1-\beta_2-n]^Q_{i_1,p^{m},q^{m}}},
	\end{eqnarray*}
	 and 
	\begin{eqnarray*}
	E\bigg(\frac{[W_1]^Q_{i_1,p^m,q^m}[W_2]^Q_{i_2,p^m,q^m}}{[\beta-\beta_1-n-W_2]^Q_{i_1,p^{m},q^{m}}}\bigg)&=&\frac{[n+i_1+i_2-1]^Q_{i_1+i_2,p^{-m},q^{-m}}}{[\beta-\beta_1+i_1]^Q_{i_1,p^{-m},q^{-m}}[\beta-\beta_1-\beta_2-n]^Q_{i_1,p^{m},q^{m}}}\nonumber\\&\times& \frac{[\beta_1]^Q_{i_1,p^{m},q^{m}}[\beta_2]^Q_{i_2,p^{m},q^{m}}}{[\beta-\beta_1-\beta_2+i_2]^Q_{i_2,p^{-m},q^{-m}}}.
	\end{eqnarray*} 
	\end{small}
	Furthermore, the  covariance of the random variable $\ddot{W}=[W_2]^Q_{p^{m},q^{m}}$  and $\widehat{W}=\frac{[W_1]^Q_{p^{m},q^{m}}}{[\beta-\beta_1-n-W_2]_{p^{m},q^{m}}}$  is given by:
	\begin{small}
	\begin{eqnarray*}
	Cov\big(\widehat{W}, \ddot{W}\big)=\frac{[n]^Q_{p^{-m},q^{-m}}[\beta_2]^Q_{p^{m},q^{m}}[\beta_1]^Q_{p^{m},q^{m}}}{\nabla^Q(n,\beta)}\bigg([n+1]^Q_{p^{-m},q^{-m}}-[n]^Q_{p^{-m},q^{-m}}\bigg),
	\end{eqnarray*}
	
	where $$\nabla^Q(n,\beta)=[\beta-\beta_1+1]^Q_{p^{-m},q^{-m}}[\beta-\beta_1-\beta_2-n]^Q_{p^{m},q^{m}}[\beta-\beta_1-\beta_2+1]^Q_{p^{-m},q^{-m}}.$$
	\end{small}
	\end{enumerate}
	\end{remark}
	
\subsection{Multivariate $\mathcal{R}(p,q)-$ hypergeometric distribution}
In this section, the multivariate $\mathcal{R}(p,q)$- hypergeometric and the negative multivariate $\mathcal{R}(p,q)$- hypergeometric distribution are determined.
 
We 
consider a sequence of independent Bernoulli trials and suppose that the probability of success at the $i$th trial is given as follows:
\begin{equation*}
p_i=\frac{\theta \tau_2^{i-1}}{\tau^{i-1}_1+\theta \tau_2^{i-1}}, \ \ i\in\mathbb{
N},\,\, \ 0<\theta<1.
\end{equation*}

We denote by  $H_j$  the number of successes after the $(s_{j-1})$th trial and until the $(s_j)$th trial, with $j\in\{1,2,\ldots,k+1\},$ $s_0=0$, $s_j=\sum_{i=1}^jr_i$, $j\in\{1,2,\ldots,k+1\}$, and $s_{k+1}=r$. Thus, the $\mathcal{R}(p,q)$ - random  variables $H_j$
 are independent and the probability distribution of the $\mathcal{R}(p,q)$- binomial probability distribution of the first kind is given by:
\begin{equation}\label{rpqbf}
P(H_j=h_j)=\genfrac{[}{]}{0pt}{}{r_j}{h_j}_{\mathcal{R}(p,q)}\frac{(\theta \tau_2^{s_{j-1}})^{h_j}\tau^{\binom{r_j-h_j}{2}}_1\tau_2^{\binom{h_j}{2}}}{\big(\tau_1^{s_{j-1}}\oplus\theta \tau_2^{s_{j-1}}\big)^{r_j}_{\mathcal{R}(p,q)}},
\ \ h_j\in\{0,1,\ldots,r_j\}.
\end{equation}
\begin{theorem}\label{thm3.4}
	 The conditional probability function of the $\mathcal{R}(p,q)$- random vector $(H_1,H_2,\ldots,H_k)$, given that $H_1+H_2+\ldots+H_{k+1}=n$, is the multivariate $\mathcal{R}(p,q)$-hypergeometric distribution. Its  probability distribution is given by:
	 \begin{small}
	 	\begin{equation}\label{eq3.2a}
	 	P(H_1=h_1,\ldots,H_k=h_k)=\mathcal{H}_k(p,q)\genfrac{[}{]}{0pt}{}{n}{h_1,\cdots,h_k}_{\mathcal{R}(p,q)}{\prod_{j=1}^{k+1}\genfrac{[}{]}{0pt}{}{\beta_j}{h_j}_{\mathcal{R}(p,q)}
	 		\over \genfrac{[}{]}{0pt}{}{\beta}{n}_{\mathcal{R}(p,q)}},
	 	\end{equation}
	 \end{small}
	 where $\mathcal{H}_k(p,q)=\tau^{\sum_{j=1}^{k}h_{j}(\beta_{j+1}-h_{j+1})}_1\tau_2^{\sum_{j=1}^{k}(n-y_j)(\beta_j-h_j)}$
	 for $h_j\in\{0,1,\cdots,n\}$, $j\in\{1,2,\cdots,k\},$ with $\sum_{j=1}^k h_j\leq n$, $h_{k+1}=n-\sum_{j=1}^k h_j$, $\beta_{k+1}=\beta-\sum_{j=1}^k \beta_j$, and $y_j=\sum_{i=1}^{j}h_i$, for $j\in\{1,2,\cdots,k\}.$
\end{theorem}
\begin{proof} 
From the relation (\ref{rpqbf}), the probability distribution of the sum $Y=H_1+H_2+\cdots+H_{k+1}$, which is the number of successes in $r$ trials, is
\begin{equation*}
P(Y=n)=\genfrac{[}{]}{0pt}{}{r}{n}_{\mathcal{R}(p,q)}\frac{\theta^n\tau_1^{\binom{r-n}{2}} \tau_2^{\binom{n}{2}}}{\big(1\oplus \theta \big)^{r}_{\mathcal{R}(p,q)}},\ \ n\in\{0,1,\ldots,r\}.
\end{equation*}
Then, the joint conditional probability distribution of the $\mathcal{R}(p,q)$- random vector $(H_1,H_2,\ldots,H_k)$, given that $Y=n,$
\begin{equation*}
P(H_1=h_1,\ldots,H_k=h_k|Y=n)=\frac{P(H_1=h_1)\ldots P(H_k=h_k)P(H_{k+1}=n-y_k)}{P(Y=n)},
\end{equation*}
on using these expressions, is obtained as:
\begin{equation*}
P(H_1=h_1,H_2=h_2,\ldots,H_k=h_k|Y=n)=\big(\tau_1\tau_2\big)^{c_k}\frac{\prod_{j=1}^{k+1}\genfrac{[}{]}{0pt}{}{r_j}{x_j}_{\mathcal{R}(p,q)}}{\genfrac{[}{]}{0pt}{}{r}{n}_{\mathcal{R}(p,q)}},
\end{equation*}
where
\[
c_k=\sum_{i=1}^kx_is_{i-1}+(n-y_k)s_k+\sum_{j=1}^k\binom{x_j}{2}+\binom{n-y_k}{2}-\binom{n}{2}.
\]
Thus, after some algebraic manipulations, it reduces to
\begin{align*}
c_k&=n\sum_{j=1}^kr_j-\sum_{i=1}^kx_i(s_k-s_{i-1})+\sum_{j=1}^k\binom{x_j}{2}+\binom{y_k+1}{2}-ny_k\\
&=\sum_{j=1}^kr_j(n-y_j)-\sum_{j=1}^kx_j(n-y_j)=\sum_{j=1}^k(n-y_j)(y_j-x_j),
\end{align*}
and the proof is completed. $\cqfd$
\end{proof}
\begin{remark}
	The multivariate $\mathcal{R}(p,q)$- hypergeometric can also be obtained by taking $m=-1$ in the relation (\ref{eq3.2}).
\end{remark}
\subsubsection{Bivariate $\mathcal{R}(p,q)$-hypergeometric distribution}
The probability distribution of the $\mathcal{R}(p,q)$- random vector $\underline{H}=\big(H_1,H_2\big)$ is called the $\mathcal{R}(p,q)$- P\'olya distribution with parameters $n$, $\underline{\beta}=(\beta_1,\beta_2),$ $p$ and $q.$ Its probability function is given by the following relation:
\begin{small}
	\begin{equation}\label{bh1}
	P(H_1=h_1,H_2=h_2)=\Psi_2(p,q)\genfrac{[}{]}{0pt}{}{n}{h_1,h_2}_{\mathcal{R}(p,q)}\frac{\prod_{j=1}^{3}[\beta_j]_{h_j,\mathcal{R}(p,q)}}{ [\beta]_{n,\mathcal{R}(p,q)}},
	\end{equation}
\end{small}
where $\Psi_2(p,q)=\tau^{\sum_{j=1}^{2}h_{j}(\beta_{j+1}-h_{j+1})}_1\tau_2^{\sum_{j=1}^{2}(n-x_j)(\beta_j-h_j)},$
$h_j\in\{0,1,\ldots,n\}$, $j\in\{1,2\},$  $h_1+h_2\leq n$, $h_{3}=n-h_1+h_2$, $\beta_{3}=\beta-\beta_1+\beta_2$, $x_1=h_1,$
and $x_2=h_1+h_2.$
\begin{proposition}
	The $\mathcal{R}(p,q)$-factorial moments of the bivariate $\mathcal{R}(p,q)$-hypergeometric  distribution, with parameters $n$, $\underline{\beta}=(\beta_1,\beta_2),$ $p$ and $q,$ are presented as follows:
	\begin{small}
	\begin{eqnarray}\label{bh2}
	E\big([H_1]_{i_1,\mathcal{R}(p^{-1},q^{-1})}\big)=\frac{[n]_{i_1,\mathcal{R}(p^{-1},q^{-1})}[\beta_1]_{i_1,\mathcal{R}(p,q)}}{[\beta]_{i_1,\mathcal{R}(p,q)}},\,i_1\in\{0,1,\cdots,n\},
	\end{eqnarray}
	\begin{eqnarray}\label{bh3}
	E\big([H_2]_{i_1,\mathcal{R}(p^{-1},q^{-1})}|H_1=h_1\big)=\frac{[n-h_1]_{i_2,\mathcal{R}(p^{-1},q^{-1})}[\beta_2]_{i_2,\mathcal{R}(p,q)}}{[\beta-\beta_1]_{i_2,\mathcal{R}(p,q)}},\,i_2\in\{0,1,\cdots,n-h_1\},
	\end{eqnarray}
	\begin{eqnarray}\label{bh4}
	E\big([H_2]_{i_2,\mathcal{R}(p^{-1},q^{-1})}\big)=\frac{[n]_{i_2,\mathcal{R}(p^{-1},q^{-1})}[\beta_2]_{i_2,\mathcal{R}(p,q)}\tau^{i_2\beta_1}_2}{[\beta]_{i_2,\mathcal{R}(p,q)}},\,\,i_2\in\{0,1,\cdots,n\},
	\end{eqnarray}
	and
	\begin{eqnarray}\label{bh5}
	E\big([H_1]_{i_1,\mathcal{R}(p^{-1},q^{-1})}[H_2]_{i_2,\mathcal{R}(p^{-1},q^{-1})}\big)=\frac{[n]_{i_1+i_2,\mathcal{R}(p^{-1},q^{-1})}[\beta_1]_{i_1,\mathcal{R}(p,q)}[\beta_2]_{i_2,\mathcal{R}(p,q)}}{\tau^{-i_2\beta_1}_2\,[\beta]_{i_1+i_2,\mathcal{R}(p,q)}},
	\end{eqnarray}
	where $i_1\in\{0,1,\cdots,n-i_2\}$ and  $i_2\in\{0,1,\cdots,n\}.$
\end{small}
\end{proposition}
\begin{corollary}
 The $\mathcal{R}(p,q)$- covariance of $[H_1]_{\mathcal{R}(p^{-1},q^{-1})}$ and $[H_2]_{\mathcal{R}(p^{-1},q^{-1})}$ is derived by:
\begin{small}
\begin{eqnarray*}
Cov\big([H_1]_{\mathcal{R}(p^{-1},q^{-1})},[H_2]_{\mathcal{R}(p^{-1},q^{-1})}\big)=\frac{[n]_{\mathcal{R}(p^{-1},q^{-1})}[\beta_1]_{\mathcal{R}(p,q)}[\beta_2]_{\mathcal{R}(p,q)}}{\tau^{-\beta_1}_2[\beta]_{\mathcal{R}(p,q)}}\Delta(n,\beta),
\end{eqnarray*}
where 
\begin{eqnarray*}
\Delta(n,\beta)=\frac{[n-1]_{\mathcal{R}(p^{-1},q^{-1})}}{[\beta-1]_{\mathcal{R}(p,q)}}-\frac{[n]_{\mathcal{R}(p^{-1},q^{-1})}}{[\beta_1]_{\mathcal{R}(p,q)}}
\end{eqnarray*}
and  $\underline{H}=\big(
H_1,H_2\big)$  a $\mathcal{R}(p,q)$-random vector verifying   the bivariate $\mathcal{R}(p,q)$-hypergeometric probability distribution, with parameters $n$, $\underline{\beta}=(\beta_1,\beta_2),$ $p$ and $q.$
\end{small}
\end{corollary}
\begin{remark}
Differents  multivariate hypergeometric distributions are deduced as:
\begin{enumerate}
	\item[(a)]
	The multivariate $q$-hypergeometric distribution. Its  probability distribution is given by:
	\begin{small}
		\begin{eqnarray*}
		P(H_1=h_1,\ldots,H_k=h_k)=\mathcal{H}_k(q)\genfrac{[}{]}{0pt}{}{n}{h_1,\cdots,h_k}_{q}{\prod_{j=1}^{k+1}\genfrac{[}{]}{0pt}{}{\alpha_j}{h_j}_{q}
			\over \genfrac{[}{]}{0pt}{}{\alpha}{n}_{q}},
		\end{eqnarray*}
	\end{small}
	where $\mathcal{H}_k(q)=q^{\sum_{j=1}^{k}h_{j}(\alpha_{j+1}-h_{j+1})}\,q^{-\sum_{j=1}^{k}(n-y_j)(\alpha_j-h_j)}$
	for $h_j\in\{0,1,\cdots,n\}$, $j\in\{1,2,\ldots,k\},$ with $\sum_{j=1}^k h_j\leq n$, $h_{k+1}=n-\sum_{j=1}^k h_j$, $\alpha_{k+1}=\alpha-\sum_{j=1}^k \alpha_j$, and $y_j=\sum_{i=1}^{j}h_i$, for $j\in\{1,2,\cdots,k\}.$
	
	The probability function of the  bivariate $(p,q)$- hypergeometric distribution with parameters $n$, $\underline{\beta}=(\beta_1,\beta_2),$  and $q.$ is given by the following relation:
	\begin{small}
		\begin{equation*}
		P(H_1=h_1,H_2=h_2)=\Psi_2(q)\genfrac{[}{]}{0pt}{}{n}{h_1,h_2}_{q}\frac{\prod_{j=1}^{3}[\beta_j]_{h_j,q}}{ [\beta]_{n,q}},
		\end{equation*}
	\end{small}
	where $\Psi_2(q)=q^{\sum_{j=1}^{2}h_{j}(\beta_{j+1}-h_{j+1})}\,q^{-\sum_{j=1}^{2}(n-x_j)(\beta_j-h_j)},$
	$h_j\in\{0,1,\ldots,n\}$, $j\in\{1,2\},$  $h_1+h_2\leq n$, $h_{3}=n-h_1+h_2$, $\beta_{3}=\beta-\beta_1+\beta_2$, $x_1=h_1,$
	and $x_2=h_1+h_2.$
	
	Besides, its  $q$-factorial moments are presented as follows:
	\begin{small}
		\begin{eqnarray*}
			E\big([H_1]_{i_1,q^{-1}}\big)=\frac{[n]_{i_1,q^{-1}}[\beta_1]_{i_1,q}}{[\beta]_{i_1,q}},\,i_1\in\{0,1,\cdots,n\},
		\end{eqnarray*}
		\begin{eqnarray*}
			E\big([H_2]_{i_1,q^{-1}}|H_1=h_1\big)=\frac{[n-h_1]_{i_2,q^{-1}}[\beta_2]_{i_2,q}}{[\beta-\beta_1]_{i_2,q}},\,i_2\in\{0,1,\cdots,n-h_1\},
		\end{eqnarray*}
		\begin{eqnarray*}
			E\big([H_2]_{i_2,q^{-1}}\big)=\frac{[n]_{i_2,q^{-1}}[\beta_2]_{i_2,q}q^{-i_2\beta_1}}{[\beta]_{i_2,q}},\,\,i_2\in\{0,1,\cdots,n\},
		\end{eqnarray*}
		and
		\begin{eqnarray*}
			E\big([H_1]_{i_1,q^{-1}}[H_2]_{i_2,q^{-1}}\big)=\frac{[n]_{i_1+i_2, q^{-1}}[\beta_1]_{i_1,q}[\beta_2]_{i_2,q}}{q^{i_2\beta_1}\,[\beta]_{i_1+i_2,q}},
		\end{eqnarray*}
		where $i_1\in\{0,1,\cdots,n-i_2\}$ and  $i_2\in\{0,1,\cdots,n\}.$
	\end{small}
	Furthermore, 
	the $q$- covariance of $[H_1]_{q^{-1}}$ and $[H_2]_{q^{-1}}$ is derived by:
	\begin{small}
		\begin{eqnarray*}
			Cov\big([H_1]_{q^{-1}},[H_2]_{q^{-1}}\big)=\frac{[n]_{q^{-1}}[\beta_1]_{q}[\beta_2]_{q}}{q^{\beta_1}[\beta]_{q}}\Delta(n,\beta),
		\end{eqnarray*}
		where 
		\begin{eqnarray*}
			\Delta(n,\beta)=\frac{[n-1]_{q^{-1}}}{[\beta-1]_{q}}-\frac{[n]_{q^{-1}}}{[\beta_1]_{q}}.
		\end{eqnarray*}
	\end{small}
\item[(b)] The probability distribution of the  multivariate $(p,q)$-hypergeometric distribution  is given by:
\begin{small}
	\begin{eqnarray*}
	P(H_1=h_1,\ldots,H_k=h_k)=\mathcal{H}_k(p,q)\genfrac{[}{]}{0pt}{}{n}{h_1,\cdots,h_k}_{p,q}{\prod_{j=1}^{k+1}\genfrac{[}{]}{0pt}{}{\alpha_j}{h_j}_{p,q}
		\over \genfrac{[}{]}{0pt}{}{\alpha}{n}_{p,q}},
	\end{eqnarray*}
\end{small}
where $\mathcal{H}_k(p,q)=p^{\sum_{j=1}^{k}h_{j}(\alpha_{j+1}-h_{j+1})}\,q^{\sum_{j=1}^{k}(n-y_j)(\alpha_j-h_j)}$
for $h_j\in\{0,1,\cdots,n\}$, $j\in\{1,2,\ldots,k\},$ with $\sum_{j=1}^k h_j\leq n$, $h_{k+1}=n-\sum_{j=1}^k h_j$, $\alpha_{k+1}=\alpha-\sum_{j=1}^k \alpha_j$, and $y_j=\sum_{i=1}^{j}h_i$, for $j\in\{1,2,\cdots,k\}.$

 The probability function of the  bivariate $(p,q)$- hypergeometric distribution with parameters $n$, $\underline{\beta}=(\beta_1,\beta_2),$ $p$ and $q.$ is given by the following relation:
\begin{small}
	\begin{equation*}
	P(H_1=h_1,H_2=h_2)=\Psi_2(p,q)\genfrac{[}{]}{0pt}{}{n}{h_1,h_2}_{p,q}\frac{\prod_{j=1}^{3}[\beta_j]_{h_j,p,q}}{ [\beta]_{n,p,q}},
	\end{equation*}
\end{small}
where $\Psi_2(p,q)=p^{\sum_{j=1}^{2}h_{j}(\beta_{j+1}-h_{j+1})}\,q^{\sum_{j=1}^{2}(n-x_j)(\beta_j-h_j)},$
$h_j\in\{0,1,\ldots,n\}$, $j\in\{1,2\},$  $h_1+h_2\leq n$, $h_{3}=n-h_1+h_2$, $\beta_{3}=\beta-\beta_1+\beta_2$, $x_1=h_1,$
and $x_2=h_1+h_2.$

Besides, its  $(p,q)$-factorial moments are presented as follows:
	\begin{small}
	\begin{eqnarray*}
	E\big([H_1]_{i_1,p^{-1},q^{-1}}\big)=\frac{[n]_{i_1,p^{-1},q^{-1}}[\beta_1]_{i_1,p,q}}{[\beta]_{i_1,p,q}},\,i_1\in\{0,1,\cdots,n\},
	\end{eqnarray*}
	\begin{eqnarray*}
	E\big([H_2]_{i_1,p^{-1},q^{-1}}|H_1=h_1\big)=\frac{[n-h_1]_{i_2,p^{-1},q^{-1}}[\beta_2]_{i_2,p,q}}{[\beta-\beta_1]_{i_2,p,q}},\,i_2\in\{0,1,\cdots,n-h_1\},
	\end{eqnarray*}
	\begin{eqnarray*}
	E\big([H_2]_{i_2,p^{-1},q^{-1}}\big)=\frac{[n]_{i_2,p^{-1},q^{-1}}[\beta_2]_{i_2,p,q}q^{i_2\beta_1}}{[\beta]_{i_2,p,q}},\,\,i_2\in\{0,1,\cdots,n\},
	\end{eqnarray*}
	and
	\begin{eqnarray*}
	E\big([H_1]_{i_1,p^{-1},q^{-1}}[H_2]_{i_2,p^{-1},q^{-1}}\big)=\frac{[n]_{i_1+i_2, p^{-1},q^{-1}}[\beta_1]_{i_1,p,q}[\beta_2]_{i_2,p,q}}{q^{-i_2\beta_1}\,[\beta]_{i_1+i_2,p,q}},
	\end{eqnarray*}
	where $i_1\in\{0,1,\cdots,n-i_2\}$ and  $i_2\in\{0,1,\cdots,n\}.$
\end{small}
Furthermore, 
the $(p,q)$- covariance of $[H_1]_{p^{-1},q^{-1}}$ and $[H_2]_{p^{-1},q^{-1}}$ is derived by:
\begin{small}
\begin{eqnarray*}
Cov\big([H_1]_{p^{-1},q^{-1}},[H_2]_{p^{-1},q^{-1}}\big)=\frac{[n]_{p^{-1},q^{-1}}[\beta_1]_{p,q}[\beta_2]_{p,q}}{q^{-\beta_1}[\beta]_{p,q}}\Delta(n,\beta),
\end{eqnarray*}
where 
\begin{eqnarray*}
\Delta(n,\beta)=\frac{[n-1]_{p^{-1},q^{-1}}}{[\beta-1]_{p,q}}-\frac{[n]_{p^{-1},q^{-1}}}{[\beta_1]_{p,q}}.
\end{eqnarray*}
\end{small}
\item[(c)]The probability distribution of the  multivariate $(p^{-1},q)$-hypergeometric distribution  is given by:
\begin{small}
	\begin{eqnarray*}
		P(H_1=h_1,\ldots,H_k=h_k)=\mathcal{H}_k(p^{-1},q)\genfrac{[}{]}{0pt}{}{n}{h_1,\cdots,h_k}_{p^{-1},q}{\prod_{j=1}^{k+1}\genfrac{[}{]}{0pt}{}{\alpha_j}{h_j}_{p^{-1},q}
			\over \genfrac{[}{]}{0pt}{}{\alpha}{n}_{p^{-1},q}},
	\end{eqnarray*}
\end{small}
where $\mathcal{H}_k(p^{-1},q)=p^{-\sum_{j=1}^{k}h_{j}(\alpha_{j+1}-h_{j+1})}\,q^{\sum_{j=1}^{k}(n-y_j)(\alpha_j-h_j)}$
for $h_j\in\{0,1,\cdots,n\}$, $j\in\{1,2,\ldots,k\},$ with $\sum_{j=1}^k h_j\leq n$, $h_{k+1}=n-\sum_{j=1}^k h_j$, $\alpha_{k+1}=\alpha-\sum_{j=1}^k \alpha_j$, and $y_j=\sum_{i=1}^{j}h_i$, for $j\in\{1,2,\cdots,k\}.$

The probability function of the  bivariate $(p^{-1},q)$- hypergeometric distribution with parameters $n$, $\underline{\beta}=(\beta_1,\beta_2),$ $p$ and $q$ is given by the following relation:
\begin{small}
	\begin{equation*}
	P(H_1=h_1,H_2=h_2)=\Psi_2(p^{-1},q)\genfrac{[}{]}{0pt}{}{n}{h_1,h_2}_{p^{-1},q}\frac{\prod_{j=1}^{3}[\beta_j]_{h_j,p^{-1},q}}{ [\beta]_{n,p^{-1},q}},
	\end{equation*}
\end{small}
where $\Psi_2(p^{-1},q)=p^{-\sum_{j=1}^{2}h_{j}(\beta_{j+1}-h_{j+1})}\,q^{\sum_{j=1}^{2}(n-x_j)(\beta_j-h_j)},$
$h_j\in\{0,1,\ldots,n\}$, $j\in\{1,2\},$  $h_1+h_2\leq n$, $h_{3}=n-h_1+h_2$, $\beta_{3}=\beta-\beta_1+\beta_2$, $x_1=h_1,$
and $x_2=h_1+h_2.$

Besides, its  $(p^{-1},q)$-factorial moments are presented as follows:
\begin{small}
	\begin{eqnarray*}
		E\big([H_1]_{i_1,p^{1},q^{-1}}\big)=\frac{[n]_{i_1,p^{1},q^{-1}}[\beta_1]_{i_1,p^{-1},q}}{[\beta]_{i_1,p^{-1},q}},\,i_1\in\{0,1,\cdots,n\},
	\end{eqnarray*}
	\begin{eqnarray*}
		E\big([H_2]_{i_1,p^{1},q^{-1}}|H_1=h_1\big)=\frac{[n-h_1]_{i_2,p^{1},q^{-1}}[\beta_2]_{i_2,p^{-1},q}}{[\beta-\beta_1]_{i_2,p^{-1},q}},\,i_2\in\{0,1,\cdots,n-h_1\},
	\end{eqnarray*}
	\begin{eqnarray*}
		E\big([H_2]_{i_2,p^{1},q^{-1}}\big)=\frac{[n]_{i_2,p^{1},q^{-1}}[\beta_2]_{i_2,p^{-1},q}q^{i_2\beta_1}}{[\beta]_{i_2,p^{-1},q}},\,\,i_2\in\{0,1,\cdots,n\},
	\end{eqnarray*}
	and
	\begin{eqnarray*}
		E\big([H_1]_{i_1,p^{1},q^{-1}}[H_2]_{i_2,p^{1},q^{-1}}\big)=\frac{[n]_{i_1+i_2, p^{1},q^{-1}}[\beta_1]_{i_1,p^{-1},q}[\beta_2]_{i_2,p^{-1},q}}{q^{-i_2\beta_1}\,[\beta]_{i_1+i_2,p^{-1},q}},
	\end{eqnarray*}
	where $i_1\in\{0,1,\cdots,n-i_2\}$ and  $i_2\in\{0,1,\cdots,n\}.$
\end{small}
Furthermore, 
the $(p^{-1},q)$- covariance of $[H_1]_{p^{1},q^{-1}}$ and $[H_2]_{p^{1},q^{-1}}$ is derived by:
\begin{small}
	\begin{eqnarray*}
		Cov\big([H_1]_{p^{1},q^{-1}},[H_2]_{p^{1},q^{-1}}\big)=\frac{[n]_{p^{1},q^{-1}}[\beta_1]_{p^{-1},q}[\beta_2]_{p,q}}{q^{-\beta_1}[\beta]_{p^{-1},q}}\Delta(n,\beta),
	\end{eqnarray*}
	where 
	\begin{eqnarray*}
		\Delta(n,\beta)=\frac{[n-1]_{p^{1},q^{-1}}}{[\beta-1]_{p^{-1},q}}-\frac{[n]_{p^{1},q^{-1}}}{[\beta_1]_{p^{-1},q}}.
	\end{eqnarray*}
\end{small}
\item[(d)]The probability distribution of the  multivariate Hounkonnou-Ngompe generalized $q$-Quesne hypergeometric distribution  is given by:
\begin{small}
	\begin{eqnarray*}
		P(H_1=h_1,\ldots,H_k=h_k)=\mathcal{H}^Q_k(p,q)\genfrac{[}{]}{0pt}{}{n}{h_1,\cdots,h_k}^Q_{p,q}{\prod_{j=1}^{k+1}\genfrac{[}{]}{0pt}{}{\alpha_j}{h_j}^Q_{p,q}
			\over \genfrac{[}{]}{0pt}{}{\alpha}{n}^Q_{p,q}},
	\end{eqnarray*}
\end{small}
where $\mathcal{H}^Q_k(p,q)=p^{\sum_{j=1}^{k}h_{j}(\alpha_{j+1}-h_{j+1})}\,q^{-\sum_{j=1}^{k}(n-y_j)(\alpha_j-h_j)}$
for $h_j\in\{0,1,\cdots,n\}$, $j\in\{1,2,\ldots,k\},$ with $\sum_{j=1}^k h_j\leq n$, $h_{k+1}=n-\sum_{j=1}^k h_j$, $\alpha_{k+1}=\alpha-\sum_{j=1}^k \alpha_j$, and $y_j=\sum_{i=1}^{j}h_i$, for $j\in\{1,2,\cdots,k\}.$

 The probability function of the  bivariate Hounkonnou-Ngompe generalized $q$-Quesne  hypergeometric distribution with parameters $n$, $\underline{\beta}=(\beta_1,\beta_2),$ $p$ and $q$ is given by the following relation:
\begin{small}
	\begin{equation*}
	P(H_1=h_1,H_2=h_2)=\Psi^Q_2(p,q)\genfrac{[}{]}{0pt}{}{n}{h_1,h_2}^Q_{p,q}\frac{\prod_{j=1}^{3}[\beta_j]^Q_{h_j,p,q}}{ [\beta]^Q_{n,p,q}},
	\end{equation*}
\end{small}
where $\Psi^Q_2(p,q)=p^{\sum_{j=1}^{2}h_{j}(\beta_{j+1}-h_{j+1})}\,q^{-\sum_{j=1}^{2}(n-x_j)(\beta_j-h_j)},$
$h_j\in\{0,1,\ldots,n\}$, $j\in\{1,2\},$  $h_1+h_2\leq n$, $h_{3}=n-h_1+h_2$, $\beta_{3}=\beta-\beta_1+\beta_2$, $x_1=h_1,$
and $x_2=h_1+h_2.$

Besides, its  factorial moments are presented as follows:
	\begin{small}
	\begin{eqnarray*}
	E\big([H_1]^Q_{i_1,p^{-1},q^{-1}}\big)=\frac{[n]^Q_{i_1,p^{-1},q^{-1}}[\beta_1]^Q_{i_1,p,q}}{[\beta]^Q_{i_1,p,q}},\,i_1\in\{0,1,\cdots,n\},
	\end{eqnarray*}
	\begin{eqnarray*}
	E\big([H_2]^Q_{i_1,p^{-1},q^{-1}}|H_1=h_1\big)=\frac{[n-h_1]^Q_{i_2,p^{-1},q^{-1}}[\beta_2]^Q_{i_2,p,q}}{[\beta-\beta_1]^Q_{i_2,p,q}},\,i_2\in\{0,1,\cdots,n-h_1\},
	\end{eqnarray*}
	\begin{eqnarray*}
	E\big([H_2]^Q_{i_2,p^{-1},q^{-1}}\big)=\frac{[n]^Q_{i_2,p^{-1},q^{-1}}[\beta_2]^Q_{i_2,p,q}q^{-i_2\beta_1}}{[\beta]^Q_{i_2,p,q}},\,\,i_2\in\{0,1,\cdots,n\},
	\end{eqnarray*}
	and
	\begin{eqnarray*}
	E\big([H_1]^Q_{i_1,p^{-1},q^{-1}}[H_2]^Q_{i_2,p^{-1},q^{-1}}\big)=\frac{[n]^Q_{i_1+i_2, p^{-1},q^{-1}}[\beta_1]^Q_{i_1,p,q}[\beta_2]^Q_{i_2,p,q}}{q^{i_2\beta_1}\,[\beta]^Q_{i_1+i_2,p,q}},
	\end{eqnarray*}
	where $i_1\in\{0,1,\cdots,n-i_2\}$ and  $i_2\in\{0,1,\cdots,n\}.$
\end{small}
Furthermore, 
the  covariance of $[H_1]^Q_{p^{-1},q^{-1}}$ and $[H_2]^Q_{p^{-1},q^{-1}}$ is derived by:
\begin{small}
\begin{eqnarray*}
Cov\big([H_1]^Q_{p^{-1},q^{-1}},[H_2]^Q_{p^{-1},q^{-1}}\big)=\frac{[n]^Q_{p^{-1},q^{-1}}[\beta_1]^Q_{p,q}[\beta_2]^Q_{p,q}}{q^{\beta_1}[\beta]^Q_{p,q}}\Delta(n,\beta),
\end{eqnarray*}
where 
\begin{eqnarray*}
\Delta(n,\beta)=\frac{[n-1]^Q_{p^{-1},q^{-1}}}{[\beta-1]^Q_{p,q}}-\frac{[n]^Q_{p^{-1},q^{-1}}}{[\beta_1]^Q_{p,q}}.
\end{eqnarray*}
\end{small}
\end{enumerate}
\end{remark}
\subsection{Multivariate negative $\mathcal{R}(p,q)$-hypergeometric distribution}
We consider a sequence of independent Bernoulli trials and suppose that the conditional probability of success at a trial, given that $j-1$ successes occur in the previous trials, is determined by:
\begin{equation*}
p_j=1-\theta \tau_1^{1-j}\tau_2^{j-1}, \ \ j\in\mathbb{N},\,\, \ 0<\theta<1.
\end{equation*}

We denotes by $V_j$  the number of failures after the $(s_{j-1})$th success and until the occurrence of the $(s_j)$th success, for $j\in\{1,2,\cdots,k+1\}$, with $s_0=0$, $s_j=\sum_{i=1}^jr_i$, $j\in\{1,2,\cdots,k+1\}.$
 Thus, the $\mathcal{R}(p,q)$- random  variables $V_j$ are independent and the $\mathcal{R}(p,q)$- binomial probability distribution of the second kind  is presented by: 
\begin{small}
\begin{equation}\label{rpqbs}
P(V_j=v_j)=\genfrac{[}{]}{0pt}{}{r_j+v_j-1}{v_j}_{\mathcal{R}(p,q)}(\theta \tau_2^{s_{j-1}})^{v_j}\big(\tau_1^{s_{j-1}} \ominus\theta \tau_2^{s_{j-1}}\big)^{r_j}_{\mathcal{R}(p,q)},
\,  v_j\in\mathbb{N}\cup\{0\}.
\end{equation}
\end{small}
\begin{theorem}\label{thm3.5}
	 The conditional probability function of the $\mathcal{R}(p,q)$- random vector $(V_1,V_2,\cdots,V_k)$, given that $V_1+V_2+\cdots+V_{k+1}=n$, is the multivariate negative $\mathcal{R}(p,q)$-hypergeometric distribution with probability function:
	\begin{equation}\label{eq3.9}
	P(V_1=v_1,\cdots,V_k=v_k)=	\big(\tau_1\tau_2\big)^{\sum_{j=1}^kr_j(n-y_j)}\frac{\prod_{j=1}^{k+1}\genfrac{[}{]}{0pt}{}{r_j+v_j-1}{v_j}_{\mathcal{R}(p,q)}
	}{\genfrac{[}{]}{0pt}{}{r+n-1}{n}_{\mathcal{R}(p,q)}}.
\end{equation}
Equivalently, 
\begin{eqnarray}
	P(V_1=v_1,\cdots,V_k=v_k)&=&\genfrac{[}{]}{0pt}{}{n}{v_1,v_2,\cdots, v_k}_{\mathcal{R}(p,q)}
	\big(\tau_1\tau_2\big)^{\sum_{j=1}^{k}r_j(n-y_j)}\nonumber\\&\times&\frac{\prod_{j=1}^{k+1}[r_j+v_j-1]_{v_j,\mathcal{R}(p,q)}}{[r+n-1]_{n,\mathcal{R}(p,q)}},
	\end{eqnarray}
	for $x_j\in\{0,1,\cdots,n\},$ $j\in\{1,2,\cdots,k\},$ with $\sum_{j=1}^k x_j\leq n$, and, where $x_{k+1}=n-\sum_{j=1}^k x_j$, $r_{k+1}=r-\sum_{j=1}^k r_j$, and $y_j=\sum_{i=1}^{j}x_i$, $j\in\{1,2,\cdots,k\}$.
\end{theorem}
\begin{proof}  According to the relation (\ref{rpqbs}),  the probability function of the sum $U=V_1+V_2+\cdots+V_{k+1}$, which is the number of failures until the occurrence of the $r$th success, is
\begin{equation*}
P(U=n)=\genfrac{[}{]}{0pt}{}{r+n-1}{n}_{\mathcal{R}(p,q)}\theta^n \big(1\ominus \theta \big)^{r}_{\mathcal{R}(p,q)},\ \ n\in\mathbb{N}\cup\{0\}.
\end{equation*}
Then, the joint conditional probability function of the random vector $(W_1,W_2,\ldots,W_k)$, given that $V=n$,
\[
P(W_1=w_1,\ldots,W_k=w_k|V=n)=\frac{P(W_1=w_1)\cdots P(W_k=w_k)P(W_{k+1}=n-v_k)}{P(V=n)},
\]
on using these expressions, is obtained as
\begin{equation*}
P(W_1=w_1,\ldots,W_k=w_k|V=n)=\big(\tau_1\tau_2\big)^{c_k}\frac{\prod_{j=1}^{k+1}\genfrac{[}{]}{0pt}{}{r_j+w_j-1}{w_j}_{\mathcal{R}(p,q)}}{\genfrac{[}{]}{0pt}{}{r+n-1}{n}_{\mathcal{R}(p,q)}},
\end{equation*}
where
\[
c_k=\sum_{i=1}^kw_is_{i-1}+(n-v_k)s_k, \ \ v_j=\sum_{i=1}^jw_i,\ \ j\in\{1,2,\ldots,k\}.
\]
Thus, after some algebra, it reduces to
\[
c_k=n\sum_{j=1}^kr_j-\sum_{i=1}^kw_i(s_k-s_{i-1})=n\sum_{j=1}^kr_j-\sum_{j=1}^kr_jv_j=\sum_{j=1}^kr_j(n-v_j)
\]
and the proof of (\ref{eq3.9}) is completed.
\end{proof}
\subsubsection{Bivariate negative $\mathcal{R}(p,q)$-hypergeometric distribution}
Let $\underline{V}=\big(V_1,V_2\big)$ be the $\mathcal{R}(p,q)$- random vector. Then, the probability function  of  the bivariate  negative $\mathcal{R}(p,q)$- hypergeometric probability distribution with parameters $n$, $\underline{\beta}=(\beta_1,\beta_2),$ $p$ and $q,$ is given by the following relation:
\begin{small}
	\begin{equation*}
	P(V_1=v_1,V_2=v_2)=\Psi_2(p,q)\genfrac{[}{]}{0pt}{}{n}{v_1,v_2}_{\mathcal{R}(p^{-1},q^{-1})}\frac{\prod_{j=1}^{3}[\beta_j]_{x_j,\mathcal{R}(p^{-1},q^{-1})}}{ [\beta]_{n,\mathcal{R}(p^{-1},q^{-1})}},
	\end{equation*}
\end{small}
where $\Psi_2(p,q)=\tau^{-\sum_{j=1}^{2}v_{j}(\beta_{j+1}-v_{j+1})}_1\tau_2^{-\sum_{j=1}^{2}(n-x_j)(\beta_j-v_j)},$
$v_j\in\{0,1,\ldots,n\}$, $j\in\{1,2\},$  $v_1+v_2\leq n$, $v_{3}=n-v_1+v_2$, $\beta_{3}=\beta-\beta_1+\beta_2$, $x_1=v_1,$
and $x_2=v_1+v_2.$
\begin{proposition}
	The $\mathcal{R}(p,q)$-factorial moments of the bivariate negative  $\mathcal{R}(p,q)$-hypergeometric probability distribution, with parameters $n$, $\underline{\beta}=(\beta_1,\beta_2),$ $p$ and $q,$ are derived as follows:
	\begin{small}
	\begin{eqnarray*}
	E\big([V_1]_{i_1,\mathcal{R}(p,q)}\big)=\frac{[n]_{i_1,\mathcal{R}(p,q)}[\beta_1]_{i_1,\mathcal{R}(p^{-1},q^{-1})}}{[\beta]_{i_1,\mathcal{R}(p^{-1},q^{-1})}},\,i_1\in\{0,1,\cdots,n\},
	\end{eqnarray*}
	\begin{eqnarray*}
	E\big([V_2]_{i_1,\mathcal{R}(p,q)}|V_1=v_1\big)=\frac{[n-v_1]_{i_2,\mathcal{R}(p,q)}[\beta_2]_{i_2,\mathcal{R}(p^{-1},q^{-1})}}{[\beta-\beta_1]_{i_2,\mathcal{R}(p^{-1},q^{-1})}},\,i_2\in\{0,1,\cdots,n-v_1\},
	\end{eqnarray*}
	\begin{eqnarray*}
	E\big([V_2]_{i_2,\mathcal{R}(p,q)}\big)=\frac{[n]_{i_2,\mathcal{R}(p,q)}[\beta_2]_{i_2,\mathcal{R}(p^{-1},q^{-1})}\tau^{-i_2\beta_1}_2}{[\beta]_{i_2,\mathcal{R}(p^{-1},q^{-1})}},\,\,i_2\in\{0,1,\cdots,n\},
	\end{eqnarray*}
	and
	\begin{eqnarray*}
	E\big([V_1]_{i_1,\mathcal{R}(p,q)}[V_2]_{i_2,\mathcal{R}(p,q)}\big)=\frac{[n]_{i_1+i_2,\mathcal{R}(p,q)}[\beta_1]_{i_1,\mathcal{R}(p^{-1},q^{-1})}[\beta_2]_{i_2,\mathcal{R}(p^{-1},q^{-1})}}{\tau^{i_2\beta_1}_2\,[\beta]_{i_1+i_2,\mathcal{R}(p^{-1},q^{-1})}},
	\end{eqnarray*}
	where $i_1\in\{0,1,\cdots,n-i_2\}$ and  $i_2\in\{0,1,\cdots,n\}.$
\end{small}
\end{proposition}
\begin{corollary}
 The $\mathcal{R}(p,q)$- covariance of $[V_1]_{\mathcal{R}(p,q)}$ and $[V_2]_{\mathcal{R}(p,q)}$ is determined by:
\begin{small}
\begin{eqnarray*}
Cov\big([V_1]_{\mathcal{R}(p,q)},[V_2]_{\mathcal{R}(p,q)}\big)=\frac{[n]_{\mathcal{R}(p,q)}[\beta_1]_{\mathcal{R}(p^{-1},q^{-1})}[\beta_2]_{\mathcal{R}(p^{-1},q^{-1})}}{\tau^{\beta_1}_2[\beta]_{\mathcal{R}(p^{-1},q^{-1})}}\Delta(n,\beta),
\end{eqnarray*}
where 
\begin{eqnarray*}
\Delta(n,\beta)=\frac{[n-1]_{\mathcal{R}(p,q)}}{[\beta-1]_{\mathcal{R}(p^{-1},q^{-1})}}-\frac{[n]_{\mathcal{R}(p,q)}}{[\beta_1]_{\mathcal{R}(p^{-1},q^{-1})}},
\end{eqnarray*}
and  $\underline{V}=\big(
V_1,V_2\big)$ is a $\mathcal{R}(p,q)$-random vector obeying   the bivariate negative $\mathcal{R}(p,q)$-hypergeometric probability distribution, with parameters $n$, $\underline{\beta}=(\beta_1,\beta_2),$ $p$ and $q.$
\end{small}
\end{corollary}
\begin{remark}
The bivariate negative hypergeometric distribution related to the quantum algebras is also interesting for the lecture. Then, 
\begin{enumerate}
	\item[(a)]The mass function  of  the bivariate  negative $q$- hypergeometric probability distribution with parameters $n$, $\underline{\beta}=(\beta_1,\beta_2),$  and $q,$ is given by the following relation:
	\begin{small}
		\begin{equation*}
		P(V_1=v_1,V_2=v_2)=\Psi_2(q)\genfrac{[}{]}{0pt}{}{n}{v_1,v_2}_{q^{-1}}\frac{\prod_{j=1}^{3}[\beta_j]_{y_j,q^{-1}}}{ [\beta]_{n,q^{-1}}},
		\end{equation*}
	\end{small}
	where $\Psi_2(q)=q^{-\sum_{j=1}^{2}v_{j}(\beta_{j+1}-v_{j+1})}q^{\sum_{j=1}^{2}(n-x_j)(\beta_j-v_j)},$
	$v_j\in\{0,1,\ldots,n\}$, $j\in\{1,2\},$  $v_1+v_2\leq n$, $v_{3}=n-v_1+v_2$, $\beta_{3}=\beta-\beta_1+\beta_2$, $x_1=v_1,$
	and $x_2=v_1+v_2.$
	Furthermore, its 
	factorial moments  are derived as follows:
	\begin{small}
		\begin{eqnarray*}
			E\big([V_1]_{i_1,q}\big)=\frac{[n]_{i_1,q}[\beta_1]_{i_1,q^{-1}}}{[\beta]_{i_1,q^{-1}}},\,i_1\in\{0,1,\cdots,n\},
		\end{eqnarray*}
		\begin{eqnarray*}
			E\big([V_2]_{i_1,q}|V_1=v_1\big)=\frac{[n-v_1]_{i_2,q}[\beta_2]_{i_2,q^{-1}}}{[\beta-\beta_1]_{i_2,q^{-1}}},\,i_2\in\{0,1,\cdots,n-v_1\},
		\end{eqnarray*}
		\begin{eqnarray*}
			E\big([V_2]_{i_2,q}\big)=\frac{[n]_{i_2,q}[\beta_2]_{i_2,q^{-1}}q^{i_2\beta_1}}{[\beta]_{i_2,q^{-1}}},\,\,i_2\in\{0,1,\cdots,n\},
		\end{eqnarray*}
		and
		\begin{eqnarray*}
			E\big([V_1]_{i_1,q}[V_2]_{i_2,q}\big)=\frac{[n]_{i_1+i_2,q}[\beta_1]_{i_1,q^{-1}}[\beta_2]_{i_2,q^{-1}}}{q^{-i_2\beta_1}\,[\beta]_{i_1+i_2,q^{-1}}},
		\end{eqnarray*}
		where $i_1\in\{0,1,\cdots,n-i_2\}$ and  $i_2\in\{0,1,\cdots,n\}.$
	\end{small}
	Moreover, 
	the $q$- covariance of $[V_1]_{q}$ and $[V_2]_{q}$ is determined by:
	\begin{small}
		\begin{eqnarray*}
			Cov\big([V_1]_{q},[V_2]_{q}\big)=\frac{[n]_{q}[\beta_1]_{q^{-1}}[\beta_2]_{q^{-1}}}{q^{-\beta_1}[\beta]_{q^{-1}}}\Delta(n,\beta),
		\end{eqnarray*}
		where 
		\begin{eqnarray*}
			\Delta(n,\beta)=\frac{[n-1]_{q}}{[\beta-1]_{q^{-1}}}-\frac{[n]_{q}}{[\beta_1]_{q^{-1}}},
		\end{eqnarray*}
		and  $\underline{V}=\big(
		V_1,V_2\big)$ is a $q$-random vector obeying   the bivariate negative $q$-hypergeometric probability distribution, with parameters $n$, $\underline{\beta}=(\beta_1,\beta_2),$  and $q.$
	\end{small}
\item[(b)]
The mass function  of  the bivariate  negative $(p,q)$- hypergeometric probability distribution with parameters $n$, $\underline{\beta}=(\beta_1,\beta_2),$ $p$ and $q,$ is given by the following relation:
\begin{small}
	\begin{equation*}
	P(V_1=v_1,V_2=v_2)=\Psi_2(p,q)\genfrac{[}{]}{0pt}{}{n}{v_1,v_2}_{p^{-1},q^{-1}}\frac{\prod_{j=1}^{3}[\beta_j]_{y_j,p^{-1},q^{-1}}}{ [\beta]_{n,p^{-1},q^{-1}}},
	\end{equation*}
\end{small}
where $\Psi_2(p,q)=p^{-\sum_{j=1}^{2}v_{j}(\beta_{j+1}-v_{j+1})}q^{-\sum_{j=1}^{2}(n-x_j)(\beta_j-v_j)},$
$v_j\in\{0,1,\ldots,n\}$, $j\in\{1,2\},$  $v_1+v_2\leq n$, $v_{3}=n-v_1+v_2$, $\beta_{3}=\beta-\beta_1+\beta_2$, $x_1=v_1,$
and $x_2=v_1+v_2.$
Furthermore, its 
	 $(p,q)$-factorial moments  are derived as follows:
	\begin{small}
	\begin{eqnarray*}
	E\big([V_1]_{i_1,p,q}\big)=\frac{[n]_{i_1,p,q}[\beta_1]_{i_1,p^{-1},q^{-1}}}{[\beta]_{i_1,p^{-1},q^{-1}}},\,i_1\in\{0,1,\cdots,n\},
	\end{eqnarray*}
	\begin{eqnarray*}
	E\big([V_2]_{i_1,p,q}|V_1=v_1\big)=\frac{[n-v_1]_{i_2,p,q}[\beta_2]_{i_2,p^{-1},q^{-1}}}{[\beta-\beta_1]_{i_2,p^{-1},q^{-1}}},\,i_2\in\{0,1,\cdots,n-v_1\},
	\end{eqnarray*}
	\begin{eqnarray*}
	E\big([V_2]_{i_2,p,q}\big)=\frac{[n]_{i_2,p,q}[\beta_2]_{i_2,p^{-1},q^{-1}}q^{-i_2\beta_1}}{[\beta]_{i_2,p^{-1},q^{-1}}},\,\,i_2\in\{0,1,\cdots,n\},
	\end{eqnarray*}
	and
	\begin{eqnarray*}
	E\big([V_1]_{i_1,p,q}[V_2]_{i_2,p,q}\big)=\frac{[n]_{i_1+i_2,p,q}[\beta_1]_{i_1,p^{-1},q^{-1}}[\beta_2]_{i_2,p^{-1},q^{-1}}}{q^{i_2\beta_1}\,[\beta]_{i_1+i_2,p^{-1},q^{-1}}},
	\end{eqnarray*}
	where $i_1\in\{0,1,\cdots,n-i_2\}$ and  $i_2\in\{0,1,\cdots,n\}.$
\end{small}
Moreover, 
 the $(p,q)$- covariance of $[V_1]_{p,q}$ and $[V_2]_{p,q}$ is determined by:
\begin{small}
\begin{eqnarray*}
Cov\big([V_1]_{p,q},[V_2]_{p,q}\big)=\frac{[n]_{p,q}[\beta_1]_{p^{-1},q^{-1}}[\beta_2]_{p^{-1},q^{-1}}}{q^{\beta_1}[\beta]_{p^{-1},q^{-1}}}\Delta(n,\beta),
\end{eqnarray*}
where 
\begin{eqnarray*}
\Delta(n,\beta)=\frac{[n-1]_{p,q}}{[\beta-1]_{p^{-1},q^{-1}}}-\frac{[n]_{p,q}}{[\beta_1]_{p^{-1},q^{-1}}},
\end{eqnarray*}
and  $\underline{V}=\big(
V_1,V_2\big)$ is a $(p,q)$-random vector obeying   the bivariate negative $(p,q)$-hypergeometric probability distribution, with parameters $n$, $\underline{\beta}=(\beta_1,\beta_2),$ $p$ and $q.$
\end{small}
\item[(c)]
The mass function  of  the bivariate  negative $(p,q)$- hypergeometric probability distribution with parameters $n$, $\underline{\beta}=(\beta_1,\beta_2),$ $p$ and $q,$ is given by the following relation:
\begin{small}
	\begin{equation*}
	P(V_1=v_1,V_2=v_2)=\Psi_2(p,q)\genfrac{[}{]}{0pt}{}{n}{v_1,v_2}_{p^{-1},q^{-1}}\frac{\prod_{j=1}^{3}[\beta_j]_{y_j,p^{-1},q^{-1}}}{ [\beta]_{n,p^{-1},q^{-1}}},
	\end{equation*}
\end{small}
where $\Psi_2(p,q)=p^{-\sum_{j=1}^{2}v_{j}(\beta_{j+1}-v_{j+1})}q^{-\sum_{j=1}^{2}(n-x_j)(\beta_j-v_j)},$
$v_j\in\{0,1,\ldots,n\}$, $j\in\{1,2\},$  $v_1+v_2\leq n$, $v_{3}=n-v_1+v_2$, $\beta_{3}=\beta-\beta_1+\beta_2$, $x_1=v_1,$
and $x_2=v_1+v_2.$
Furthermore, its 
$(p,q)$-factorial moments  are derived as follows:
\begin{small}
	\begin{eqnarray*}
		E\big([V_1]_{i_1,p,q}\big)=\frac{[n]_{i_1,p,q}[\beta_1]_{i_1,p^{-1},q^{-1}}}{[\beta]_{i_1,p^{-1},q^{-1}}},\,i_1\in\{0,1,\cdots,n\},
	\end{eqnarray*}
	\begin{eqnarray*}
		E\big([V_2]_{i_1,p,q}|V_1=v_1\big)=\frac{[n-v_1]_{i_2,p,q}[\beta_2]_{i_2,p^{-1},q^{-1}}}{[\beta-\beta_1]_{i_2,p^{-1},q^{-1}}},\,i_2\in\{0,1,\cdots,n-v_1\},
	\end{eqnarray*}
	\begin{eqnarray*}
		E\big([V_2]_{i_2,p,q}\big)=\frac{[n]_{i_2,p,q}[\beta_2]_{i_2,p^{-1},q^{-1}}q^{-i_2\beta_1}}{[\beta]_{i_2,p^{-1},q^{-1}}},\,\,i_2\in\{0,1,\cdots,n\},
	\end{eqnarray*}
	and
	\begin{eqnarray*}
		E\big([V_1]_{i_1,p,q}[V_2]_{i_2,p,q}\big)=\frac{[n]_{i_1+i_2,p,q}[\beta_1]_{i_1,p^{-1},q^{-1}}[\beta_2]_{i_2,p^{-1},q^{-1}}}{q^{i_2\beta_1}\,[\beta]_{i_1+i_2,p^{-1},q^{-1}}},
	\end{eqnarray*}
	where $i_1\in\{0,1,\cdots,n-i_2\}$ and  $i_2\in\{0,1,\cdots,n\}.$
\end{small}
Moreover, 
the $(p,q)$- covariance of $[V_1]_{p,q}$ and $[V_2]_{p,q}$ is determined by:
\begin{small}
	\begin{eqnarray*}
		Cov\big([V_1]_{p,q},[V_2]_{p,q}\big)=\frac{[n]_{p,q}[\beta_1]_{p^{-1},q^{-1}}[\beta_2]_{p^{-1},q^{-1}}}{q^{\beta_1}[\beta]_{p^{-1},q^{-1}}}\Delta(n,\beta),
	\end{eqnarray*}
	where 
	\begin{eqnarray*}
		\Delta(n,\beta)=\frac{[n-1]_{p,q}}{[\beta-1]_{p^{-1},q^{-1}}}-\frac{[n]_{p,q}}{[\beta_1]_{p^{-1},q^{-1}}},
	\end{eqnarray*}
	and  $\underline{V}=\big(
	V_1,V_2\big)$ is a $(p,q)$-random vector obeying   the bivariate negative $(p,q)$-hypergeometric probability distribution, with parameters $n$, $\underline{\beta}=(\beta_1,\beta_2),$ $p$ and $q.$
\end{small}
\item[(d)]The probability  function  of  the bivariate  negative generalized $q$- Quesne  hypergeometric probability distribution, with parameters $n$, $\underline{\beta}=(\beta_1,\beta_2),$ $p$ and $q,$ is deduced as :
\begin{small}
	\begin{equation*}
	P(V_1=v_1,V_2=v_2)=\Psi_2(p,q)\genfrac{[}{]}{0pt}{}{n}{v_1,v_2}^Q_{p^{-1},q^{-1}}\frac{\prod_{j=1}^{3}[\beta_j]^Q_{y_j,p^{-1},q^{-1}}}{ [\beta]^Q_{n,p^{-1},q^{-1}}},
	\end{equation*}
\end{small}
where $\Psi_2(p,q)=p^{-\sum_{j=1}^{2}v_{j}(\beta_{j+1}-v_{j+1})}q^{\sum_{j=1}^{2}(n-x_j)(\beta_j-v_j)},$
$v_j\in\{0,1,\cdots,n\}$, $j\in\{1,2\},$  $v_1+v_2\leq n$, $v_{3}=n-v_1+v_2$, $\beta_{3}=\beta-\beta_1+\beta_2$, $x_1=v_1,$
and $x_2=v_1+v_2.$
Furthermore, its 
	 generalized $q-$ Quesne factorial moments  are derived as follows:
	\begin{small}
	\begin{eqnarray*}
	E\big([V_1]^Q_{i_1,p,q}\big)=\frac{[n]^Q_{i_1,p,q}[\beta_1]^Q_{i_1,p^{-1},q^{-1}}}{[\beta]^Q_{i_1,p^{-1},q^{-1}}},\,i_1\in\{0,1,\cdots,n\},
	\end{eqnarray*}
	\begin{eqnarray*}
	E\big([V_2]^Q_{i_1,p,q}|V_1=v_1\big)=\frac{[n-v_1]^Q_{i_2,p,q}[\beta_2]^Q_{i_2,p^{-1},q^{-1}}}{[\beta-\beta_1]^Q_{i_2,p^{-1},q^{-1}}},\,i_2\in\{0,1,\cdots,n-v_1\},
	\end{eqnarray*}
	\begin{eqnarray*}
	E\big([V_2]^Q_{i_2,p,q}\big)=\frac{[n]^Q_{i_2,p,q}[\beta_2]^Q_{i_2,p^{-1},q^{-1}}\,q^{i_2\beta_1}}{[\beta]^Q_{i_2,p^{-1},q^{-1}}},\,\,i_2\in\{0,1,\cdots,n\},
	\end{eqnarray*}
	and
	\begin{eqnarray*}
	E\big([V_1]^Q_{i_1,p,q}[V_2]^Q_{i_2,p,q}\big)=\frac{[n]^Q_{i_1+i_2,p,q}[\beta_1]^Q_{i_1,p^{-1},q^{-1}}[\beta_2]^Q_{i_2,p^{-1},q^{-1}}}{q^{-i_2\beta_1}\,[\beta]^Q_{i_1+i_2,p^{-1},q^{-1}}},
	\end{eqnarray*}
	where $i_1\in\{0,1,\cdots,n-i_2\}$ and  $i_2\in\{0,1,\cdots,n\}.$
	
Besides, 
 the generalized $q$- Quesne covariance of $[V_1]^Q_{p,q}$ and $[V_2]^Q_{p,q}$ is given by:
\begin{eqnarray*}
Cov\big([V_1]^Q_{p,q},[V_2]^Q_{p,q}\big)=\frac{[n]^Q_{p,q}[\beta_1]^Q_{p^{-1},q^{-1}}[\beta_2]^Q_{p^{-1},q^{-1}}}{q^{-\beta_1}[\beta]^Q_{p^{-1},q^{-1}}}\Delta^Q(n,\beta),
\end{eqnarray*}
where 
\begin{eqnarray*}
\Delta^Q(n,\beta)=\frac{[n-1]^Q_{p,q}}{[\beta-1]^Q_{p^{-1},q^{-1}}}-\frac{[n]^Q_{p,q}}{[\beta_1]^Q_{p^{-1},q^{-1}}},
\end{eqnarray*}
and  $\underline{V}=\big(
V_1,V_2\big)$ is a generalized $q$- Quesne random vector verifying   the bivariate negative generalized $q$- Quesne hypergeometric  distribution, with parameters $n$, $\underline{\beta}=(\beta_1,\beta_2),$ $p$ and $q.$
\end{small}
\end{enumerate}
\end{remark}
\section*{Acknowledgements}
This research was partly supported by the SNF Grant No. IZSEZ0\_206010.

\end{document}